\documentclass[11pt,reqno]{amsart}

\usepackage{graphicx}         

\usepackage{comment}

\usepackage{amsmath, amsopn, amssymb}
\usepackage{a4wide}
\usepackage{xspace}
\usepackage{float}
\usepackage{natbib}
\usepackage{verbatim}
\usepackage[latin1]{inputenc}

\usepackage{color}
\usepackage{bbm}
\usepackage{multirow}

\usepackage{latexsym, amsbsy, epsfig, enumerate}
\usepackage{mathsym}

\newcommand{\PP}{\mathsf{P}}

\newcommand{\suman}{\ensuremath{\sum_{i=1}^{n}}}

\newcommand{\E}{\ensuremath{\mathsf{E}}}
\newcommand{\var}{\ensuremath{\mathsf{var}}}
\newcommand{\Es}{\ensuremath{\mathsf{E}\,}}

\newcommand{\ind}{\ensuremath{\mathbbm{1}}}

\renewcommand{\aa}{\ensuremath{\mathbf{a}}}
\renewcommand{\AA}{\ensuremath{\mathbf{A}}}  
\newcommand{\BB}{\ensuremath{\mathbf{B}}}

\newcommand{\ee}{\ensuremath{\mathbf{e}_{1}}}
\newcommand{\hh}{\ensuremath{\mathbf{h}}}
\newcommand{\JJ}{\ensuremath{\mathbf{J}}}
\newcommand{\ii}{\ensuremath{\mathbf{i}}}

\newcommand{\kk}{\ensuremath{\mathbf{k}}}

\renewcommand{\QQ}{\ensuremath{\mathbf{Q}}}
\renewcommand{\tt}{\ensuremath{\mathbf{t}}}

\newcommand{\XX}{\ensuremath{\boldsymbol{X}}}
\newcommand{\xx}{\ensuremath{\mathbf{x}}}  

\newcommand{\uu}{\ensuremath{\mathbf{u}}}
\newcommand{\vv}{\ensuremath{\mathbf{v}}}

\newcommand{\bbeta}{\ensuremath{\boldsymbol{\beta}}}
\newcommand{\bgamma}{\ensuremath{\boldsymbol{\gamma}}}
\newcommand{\bGamma}{\ensuremath{\boldsymbol{\Gamma}}}
\newcommand{\bdelta}{\ensuremath{\boldsymbol{\delta}}}
\newcommand{\beps}{\ensuremath{\boldsymbol{\eps}}}

\newcommand{\bphi}{\ensuremath{\boldsymbol{\phi}}}
\newcommand{\bpsi}{\ensuremath{\boldsymbol{\psi}}}

\newcommand{\bsigma}{\ensuremath{\boldsymbol{\sigma}}}
\newcommand{\bSigma}{\boldsymbol{\Sigma}}
\newcommand{\btheta}{\ensuremath{\boldsymbol{\theta}}}

\newcommand{\hatbtheta}{\ensuremath{\boldsymbol{\widehat{\theta}}}}

\newcommand{\setI}{\ensuremath{\mathbb{I}}}

\newcommand{\hatm}{\ensuremath{\widehat{m}}}
\newcommand{\hatsigma}{\ensuremath{\widehat{\sigma}}}
\newcommand{\hatF}{\ensuremath{\widehat{F}}}
\newcommand{\hatG}{\ensuremath{\widehat{G}}}
\newcommand{\hata}{\ensuremath{\widehat{a}}}
\newcommand{\hatb}{\ensuremath{\widehat{b}}}

\newcommand{\Uhat}{\ensuremath{\widehat{U}_{1i}}\xspace}
\newcommand{\Utilde}{\ensuremath{\widetilde{U}_{1i}}\xspace}
\newcommand{\Vhat}{\ensuremath{\widehat{U}_{2i}}\xspace}
\newcommand{\Vtilde}{\ensuremath{\widetilde{U}_{2i}}\xspace}
\newcommand{\hateps}{\ensuremath{\widehat{\eps}}\xspace}

\newcommand{\cond}{\ensuremath{\,|\,}}

\DeclareRobustCommand{\inDist}{\ensuremath{\xrightarrow[n\rightarrow\infty]{d}}}

\newtheorem{theorem}{Theorem}

\newtheorem{lemma}{Lemma}

\theoremstyle{remark}
\newtheorem{remark}{Remark}

\DeclareRobustCommand{\tr}{^\mathsf{T}}

\renewcommand{\baselinestretch}{1.24}

\makeatletter
\def\singlespace{\def\baselinestretch{1}\@normalsize}

\def\boxit#1{\vbox{\hrule\hbox{\vrule\kern6pt
          \vbox{\kern6pt#1\kern6pt}\kern6pt\vrule}\hrule}}

\begin{document}

\vspace*{-0.8 cm}

\noindent
\title[A copula approach for dependence modeling]
{A copula approach for dependence modeling in multivariate nonparametric time series}
\author{Natalie Neumeyer$^{\MakeLowercase{a}}$, Marek  Omelka$^{\MakeLowercase{b}}$, 
\v{S}\'{a}rka Hudecov\'{a}$^{\MakeLowercase{b}}$}

\maketitle

\begin{center}
$^a$ 
Department of Mathematics, University of Hamburg, Bundesstrasse 55, 20146 Hamburg, Germany \\
$^b$
Department of Probability and Statistics,
Faculty of Mathematics and Physics,  
Charles University, 
Sokolovsk\'a 83,
186\,75 Praha 8, Czech Republic
\\

\end{center}
\vspace*{0.6 cm}

\begin{center}
\today
\end{center}

\begin{abstract}
This paper is concerned with modeling the dependence structure of two (or more) time-series 
in the presence of a (possibly multivariate) covariate which may include 
past values of the time series. We assume that the covariate influences 
only the conditional mean and the conditional variance of each of the time series 
but the distribution of the standardized innovations is not influenced by 
the covariate and is stable in time. The joint distribution of the time series is then determined by the conditional means, the 
conditional variances and the marginal distributions of the innovations, which we estimate nonparametrically, and the copula of the innovations, which represents the dependency structure. We consider 
a nonparametric as well as a semiparametric estimator based 
on the estimated residuals. We show that under suitable assumptions 
these copula estimators are asymptotically equivalent to estimators that would be based 
on the unobserved innovations. The theoretical results are illustrated by simulations 
and a real data example. 
\end{abstract}
\vspace*{0.6 cm}


%

\date{\today}

\maketitle

\boxit{\textcolor{red}{NOTICE: This is the version of the manuscript accepted to  
Journal of Multivariate Analysis.}
\textcolor{blue}{DOI: 10.1016/j.jmva.2018.11.016} 
}

\section{Introduction}

Modeling the dependency of $k$ observed time series can be of utmost importance for applications, e.\,g.\ in risk management (for instance to model the dependence between several exchange rates).
We will take the approach to model $k$ dependent nonparametric AR-ARCH time series
$$ Y_{ji} = m_{j}(\XX_{i}) + \sigma_{j}(\XX_{i})\,\eps_{ji},\; i=1,\dotsc,n,\,  j=1,\dotsc,k,
$$
where the covariate $\XX_{i}$ may include past values of the process, $Y_{j\,i-1},Y_{j\, i-2},\dots$ ($j=1,\dotsc,k$), or other exogenous variables. Further the  innovations $(\eps_{1i},\dotsc,\eps_{ki})$, $ i \in \mathbb{Z}$, 
are assumed to be independent and identically distributed random vectors and $(\eps_{1i},\dotsc,\eps_{ki})$ is independent 
of the past and present covariates $\XX_{\ell}$, $\ell \leq i$ for all  $i$. For identifiability we assume $\Es \eps_{ji} = 0$, $\var(\eps_{ji}) = 1$ ($j=1,\dotsc,k$), such that the functions $m_j$ and $\sigma_j$ represent the conditional mean and volatility function of the $j$th time series. 
Such models are also called multivariate nonparametric CHARN (conditional heteroscedastic autoregressive nonlinear) models and have gained much attention over the last decades, see \cite{fan_yao_book2005} and \cite{gao2007nonlinear}  
 for extensive overviews. 

Note that due to the structure of the model and Sklar's theorem \citep[see e.g.,][]{nelsen_2006}, for $z_j=(y_j-m_j(\xx))/\sigma_j(\xx)$ ($j=1,\dotsc,k$) one has
$$
 \PP(Y_{1i} \leq y_{1},\dotsc, Y_{ki} \leq y_{k} \cond \XX_i=\xx) = 
  \PP(\eps_{1i} \leq z_{1},\dotsc, \eps_{ki} \leq z_{k})  
=
  C(F_{1\eps}(z_1),\dotsc, F_{k\eps}(z_k)), 
$$
where $F_{j\eps}$ ($j=1,\dotsc,k$) denote the marginal distributions of the innovations and $C$ their copula. Thus the joint
 conditional distribution of the observations, given the covariate, is completely specified by the individual conditional mean and variance functions, the marginal distributions of the innovations, and their copula.
The copula $C$ describes the dependence structure of the $k$ time series, conditional on the covariates, after removing influences of the conditional means and variances as well as marginal distributions. 

We will model the  conditional mean and variance function nonparametrically like 
\cite{hardle1998nonparametric}, among others. 
Semiparametric estimation, e.\,g.\ with additive structure for $m_j$ 
and multiplicative structure for $\sigma_j^2$ as in \cite{yang1999nonparametric}, 
 can be considered as well and all presented results remain valid under appropriate changes for the estimators and assumptions. 
Further we will model the marginal distributions of the innovations nonparametrically, whereas we will take two different approaches to estimate the copula $C$: nonparametrically and parametrically. As the innovations are not observable, both estimators will be based on estimated residuals. We will show that the asymptotic distribution is not affected by the necessary pre-estimation of the mean and variance functions. This remarkable result is intrinsic for copula estimation and it was already observed in (semi)parametric 
estimation of copula (see the references in the next paragraph). It is worth 
noting that on the other hand the asymptotic distribution of empirical distribution functions is typically influenced by pre-estimation of mean and variance functions. 
 Moreover, comparison of the nonparametric and parametric copula estimator gives us the possibility to test goodness-of-fit of a parametric class of copulas. 

Our approach extends the following parametric and semiparametric approaches in time series contexts. \cite{chen2006estimation}
introduced SCOMDY (semiparametric copula-based multivariate dynamic) models, which are very similar to the model considered here. However, the conditional mean and variance functions are modeled parametrically, while the marginal distributions of innovations are estimated nonparametrically and a parametric copula model is applied to model the dependence. See also  \cite{kim2007semiparametric}
for similar methods for some parametric time series models including nonlinear GARCH models, \cite{remillard2012copula}, \cite{kim2008estimating} and the review by \cite{patton2012review}.
\cite{chan2009statistical} give (next to the parametric estimation of a copula) even a goodness-of fit test for the innovation copula in the GARCH context. Further, in an i.i.d.\ setting \cite{ogv_sjs_2015} 
show that in nonparametric location-scale models the asymptotic distribution of the empirical copula is not influenced by pre-estimation of the mean and variance function. 
This results was further generalized by \cite{portier2018weak} to a completely nonparametric model 
for the marginals.  

The remainder of the paper is organized as follows. In Section~\ref{section2} we define the estimators and state some regularity assumptions. In Subsection \ref{section2.1} we show the weak convergence of the copula process, while in Subsection \ref{section2.2} we show asymptotic normality of a parameter estimator when considering a parametric class of copulas. Subsection  \ref{section2.3} is devoted to goodness-of-fit testing. In Section \ref{sec: simul study} we present simulation results and 
in Section~\ref{sec: application} a real data example. All proofs are given in the Appendix. 


\section{Main results}\label{section2}

For the ease of presentation we will focus on the case of two time series, i.\,e.\ $k=2$, but all results can be extended to general $k\geq 2$ in an obvious manner.
Suppose we have observed for $i=1,\dotsc,n$ a section of the stationary stochastic process $\big\{Y_{1i}, Y_{2i}, \XX_{i}\big\}_{i \in \mathbb{Z}}$ 
that satisfies 
\begin{equation} \label{eq: nonpar location scale}
 Y_{1i} = m_{1}(\XX_{i}) + \sigma_{1}(\XX_{i})\,\eps_{1i}, 
 \quad 
 Y_{2i} = m_{2}(\XX_{i}) + \sigma_{2}(\XX_{i})\,\eps_{2i},  
 \quad 
 i=1,\dotsc,n, 
\end{equation}
where $\XX_{i}=(X_{i1},\dotsc,X_{id})\tr$ is a $d$-dimensional 
covariate and the  innovations $\big\{(\eps_{1i},\eps_{2i})\big\}_{i \in \mathbb{Z}}$ 
are independent identically distributed random vectors.  Further $(\eps_{1i},\eps_{2i})$ is independent 
of the past and present covariates $\XX_{k}$, $k \leq i,\, \forall i$, and $ \Es \eps_{1i} = \Es \eps_{2i} = 0$, $ \var(\eps_{1i}) = \var(\eps_{2i}) = 1$.
If the marginal distribution functions 
$F_{1\eps}$ and $F_{2\eps}$ of the innovations are continuous, then the copula 
function~$C$ of the innovations is unique and can be expressed as 
\begin{equation} \label{eq: explicit expression for copula}
 C(u_1,u_2) = F_{\eps}\big(F_{1\eps}^{-1}(u_1), F_{2\eps}^{-1}(u_2) \big), 
 \quad (u_1,u_2) \in [0,1]^{2}.
\end{equation}
As the innovations $(\eps_{1i},\eps_{2i})$ are unobserved, 
the inference about the copula function~$C$ is based 
on the estimated residuals 
\begin{equation} \label{eq: estim residuals}
 \hateps_{ji} 
  = \frac{Y_{ji} - \hatm_{j}(\XX_i)}{\hatsigma_{j}(\XX_i)},
 \qquad i = 1,\dotsc,n, \quad j=1,2, 
\end{equation}
where $\hatm_{j}$ and $\hatsigma_{j}$ are the estimates of 
the unknown functions~$m_{j}$ and $\sigma_{j}$. In what follows 
we will consider the local polynomial estimators of order~$p$; see 
\cite{fan_gijbels_1996} or \cite{masry1996multivariate}, among others. 
Here, for a given~$\xx=(x_1,\dotsc,x_d)\tr$,  $\hatm_{j}(\xx)$  is 
defined as $\widehat{\beta}_\mathbf{0}$, the component of $\widehat{\bbeta}$ with multi-index 
$\mathbf{0}=(0,\dotsc,0)$, where $\widehat{\bbeta}$ is the solution to the minimization problem 
\begin{equation} \label{minimization}
 \min_{\bbeta=(\beta_{\ii})_{\ii \in \setI}}\, \sum_{\ell =1}^n \Big[Y_{j\ell} 
  - \sum_{\ii \in \setI} \beta_{\ii}\,\psi_{\ii,\hh_n}\big(\XX_\ell-\xx\big) \Big]^2 
  K_{\hh_{n}}(\XX_{\ell}-\xx). 
\end{equation}
Here $\setI=\setI(d,p)$ denotes the set of multi-indices $\ii=(i_1,\dotsc,i_d)$ with $i.=i_1+\dots+i_d\leq p$ and 
$\psi_{\ii,\hh_n}(\xx)=\prod_{k=1}^d \big(\frac{x_k}{h_n^{(k)}}\big)^{i_k}\frac{1}{i_k!}$.
Further
$$
 K_{\hh_{n}}(\XX_{\ell}-\xx) = \prod_{k=1}^{d} \tfrac{1}{h_n^{(k)}}\,k\Big(\tfrac{X_{\ell k}-x_{k}}{h_n^{(k)}}\Big),  
$$
with $k$ being a kernel function and $\hh_{n} 
= \big(h_n^{(1)}, \dotsc, h_n^{(d)}\big)$ the smoothing parameter. 

Further $\sigma_{j}^{2}(\xx)$ is estimated as 
$$
 \hatsigma_{j}^{2}(\xx) = \widehat{s}_{j}(\xx) - \hatm_{j}^{2}(\xx), 
$$
where $\widehat{s}_{j}(\xx)$ is obtained in the same way as 
$\hatm_{j}(\xx)$ but with $Y_{j\ell}$ replaced with $Y_{j\ell}^2$. 

For any function $f$ defined on $\JJ$, interval in $\mathbb{R}^d$, define for $\ell\in\mathbb{N}$, $\delta\in (0,1]$,
$$ 
\|f\|_{\ell+\delta} = \max_{\ii \in \setI(d,\ell)} \sup_{\xx \in \JJ} |D^{\ii}f(\xx)| 
 + \max_{\ii \in \setI(d,\ell)\atop i.=\ell} \sup_{\xx,\xx' \in \JJ\atop \xx\neq\xx'} \frac{|D^{\ii}f(\xx)-D^{\ii}f(\xx')|}{\|\xx-\xx'\|^{\delta}}, $$
where $  D^{\ii} = \frac{\partial^{i.}}{\partial x_1^{i_1} \ldots \partial x_d^{i_d}}, $
and $\|\cdot\|$ is the Euclidean norm on $\mathbb{R}^d$.
Denote by $C_M^{\ell+\delta}(\JJ)$ the set of $\ell$-times differentiable functions $f$ on $\JJ$, such that $\|f\|_{\ell+\delta}\leq M$. 
Denote by $\widetilde{C}_2^{\ell+\delta}(\JJ)$ the subset of~$C_2^{\ell+\delta}(\JJ)$ of the functions that satisfy  $\inf_{\xx\in\JJ} f(\xx)\geq \frac 12$.


In what follows we are going to prove that under appropriate 
regularity assumptions using the estimated 
residuals~\eqref{eq: estim residuals} instead of the 
(true) unobserved innovations~$\eps_{ji}$ affects  
neither the asymptotic distribution of the empirical copula 
estimator nor the parametric estimator of a copula.

\subsection{Empirical copula estimation}\label{section2.1}
Mimicking~$\eqref{eq: explicit expression for copula}$ 
the copula function~$C$ can be estimated nonparametrically as 
\begin{equation}  \label{copula estimates based on residuals}
 \widetilde{C}_{n}(u_{1},u_{2})
= \hatF_{\hateps} 
\Big(\hatF_{1\hateps}^{-1}(u_1), 
\hatF_{2\hateps}^{-1}(u_2)
\Big), 
\end{equation}
where
\begin{equation} \label{eq: ecdf of residuals Wn}
 \hatF_{\hateps}(y_1,y_2) =   
  \frac{1}{W_n} 
  \suman w_{ni} \, \ind \big\{\hateps_{1i} \leq y_1,  
 \hateps_{2i} \leq y_2 \big\}, 
\end{equation}
is the estimate of the joint distribution function~$F_{\eps}(y_1,y_2)$ 
and 
$$
\hatF_{j\hateps}(y) = \frac{1}{W_n} 
  \suman w_{ni} \, \ind \big\{\hateps_{ji} \leq y \big\}, \quad j=1,2 ,
$$
%
the corresponding marginal empirical cumulative distribution functions. 
Here we make use of a weight function $w_n(\xx)=\ind\{\xx\in \JJ_n\}$ and put $w_{ni}=w_{n}(\XX_{i})$ 
as well as $ W_{n}=\sum_{j=1}^{n} w_{nj}  $. For some real positive sequence $c_n\to\infty$ we set $\JJ_n=[-c_n,c_n]^d$.

Now let $C_{n}^{(or)}$ be the `oracle' estimator
based on the unobserved innovations, i.e. 
\begin{equation} \label{eq: oracle estim}
 C_{n}^{(or)} (u_1,u_2) = 
\hatF_{\eps} \Big(\hatF_{1\eps}^{-1}(u_1), 
\hatF_{2\eps}^{-1}(u_2)
\Big),  
%
\end{equation}
where 
$\hatF_{\eps}(z_1,z_2) = \frac{1}{n} 
\suman \ind \big\{ \eps_{1i} \leq z_1, \eps_{2i} \leq z_2 \big\}$  
is the estimator of $F_{\eps}(z_1,z_2)$ based on the unobserved 
innovations and 
%
%
$\hatF_{j\eps}$ $(j=1,2)$ the corresponding marginal empirical cumulative distribution functions. 



\subsubsection*{Regularity assumptions}
\begin{itemize}
\item[$\mathbf{(\boldsymbol{\beta})}$] The process $(\XX_{i},Y_{1i},Y_{2i})_{i \in \mathbb{Z}}$ is strictly stationary and 
absolutely regular ($\beta$-mixing) with the mixing coefficient
$\beta_{i}$ that satisfies 
$\beta_{i}=O(i^{-b})$ with $b>d+3$. 
\item[$\mathbf{(F_{\beps})}$] 
The  second-order partial derivatives
$F_{\eps}^{(1,1)}$, $F_{\eps}^{(1,2)}$ and $F_{\eps}^{(2,2)}$ of 
the joint cumulative distribution function 
$F_{\eps}(y_1,y_2) = \PP(\eps_{1} \leq y_1, \eps_2 \leq y_2)$,  
with
$F_{\eps}^{(j,k)}(y_1, y_2)=\frac{\partial^{2} F_{\eps}(y_1,y_2)}{\partial y_{j} \partial y_{k}}$, satisfy 
$$
 \max_{j,k \in \{1,2\}}\sup_{y_1,y_2 \in \RR} \big|F_{\eps}^{(j,k)}(y_1,y_2)
  \big|(1 + |y_j|)(1 + |y_k|) \big| < \infty. 
$$
Further the innovation density $f_{j\eps}$ $(j=1,2)$ 
satisfies 
%
$$
\qquad  \lim_{u \to 0_{+}}\big(1+F_{j\eps}^{-1}(u)\big)f_{j\eps}\big(F_{j\eps}^{-1}(u)\big)=0  
\quad \text{ and }  \quad 
\lim_{u \to 1_{-}}\big(1+F_{j\eps}^{-1}(u)\big)f_{j\eps}\big(F_{j\eps}^{-1}(u)\big) = 0. 
$$
%
\item[$\mathbf{(F_{\XX})}$]  The observations $\XX_i$ ($i \in \ZZ$) have density $f_{\XX}$ that is bounded and  differentiable with bounded uniformly continuous first order partial derivatives. 
Suppose that the sequence $c_n$ which is of order $O\big((\log n)^{1/d}\big)$ 
is chosen in such a way  that $\inf_{\xx\in \JJ_{n}}f_{\XX}(\xx)$  converges 
to zero not faster than some negative power of $\log n$. 
%
\item[$\mathbf{(M)}$]
For some $s>\frac{2b-2-d}{b-3-d}$ with $b$ from assumption ($\boldsymbol{\beta}$), for $j=1,2$, $\Es |\varepsilon_{j0}|^{2s} < \infty$, 
the functions $\sigma_j^{2s}f_{\XX}$ and $|m_j\sigma_j|^sf_{\XX}$ are bounded and
 there are some $i^*\in\mathbb{N}$, $B>0$ such that for all $i\geq i^*$,
\begin{eqnarray*}
&& \sup_{\xx_0,\xx_i}\sigma_j^2(\xx_0)\sigma_j^2(\xx_i)f_{\XX_{0},\XX_{i}}(\xx_0,\xx_i)\leq B, \\
&& \sup_{\xx_0,\xx_i}\big|m_j(\xx_0)m_j(\xx_i)\big|\,\sigma_j(\xx_0)\sigma_j(\xx_i)f_{\XX_{0},\XX_{i}}(\xx_0,\xx_i)\leq B, 
\end{eqnarray*}
 where $f_{\XX_{0},\XX_{i}}$ denotes the joint density of $(\XX_0,\XX_i)$ and is bounded 
(for  $i\geq i^*$). 
%
%
%
%
\item[$\mathbf{(m\bsigma)}$] 
Let, for $j=1,2$ and for each $n\in\mathbb{N}$, $m_{j}$ and $\sigma_{j}$  be elements of $C_{M_n}^{p+1}(\JJ_n)$ for some sequence $M_n$ that is either bounded or diverges to infinity not faster than some power of $\log n$. 
Further, assume $\E[\sigma_j^4(\XX_1)]<\infty$
and  
that $\min_{j=1,2}\inf_{\xx\in \JJ_n}\sigma_j(\xx)$ is either bounded away from zero or converges to zero not faster than a negative power of $\log n$. 
%
\item[$\mathbf{(Bw)}$] 
There exists a sequence $h_{n}$ such that $\tfrac{h_{n}^{(k)}}{h_n} \to a_{k}$, 
where $a_{k} \in (0,\infty)$, $k=1,\dotsc,d$. 
Further, 
there exists some $\delta > \tfrac{d}{b-1}$ such that
\begin{equation}\label{Hansen3}
nh_n^{2p+2} (\log n)^D=o(1),\quad nh_n^{3d+2\delta} (\log n)^{-D}\to\infty
\end{equation}
for all $D>0$. 
%
%
%

\item[$\mathbf{(k)}$] 
$k:\mathbb{R}\to\mathbb{R}$ is a symmetric ($d+2$)-times continuously differentiable probability density function supported on $[-1, 1]$.
%
 
%
\end{itemize}

\begin{remark}
Using $F_{\eps}(y_1,y_2) = C\big(F_{1\eps}(y_1), F_{2\eps}(y_2)\big)$  
assumption $\mathbf{(F_{\beps})}$ requires that 
\begin{multline*}
\max_{j,k \in \{1,2\}} \sup_{u_1,u_2 \in [0,1]}\Big|C^{(j,k)}(u_1,u_2)\,f_{j\eps}\big(F_{j\eps}^{-1}(u_j)\big)\, f_{k\eps}\big(F_{k\eps}^{-1}(u_k)\big)
\\ 
 + C^{(j)}(u_1,u_2)\,f'_{j\eps}\big(F_{j\eps}^{-1}(u_j)\big)
\ind\{j=k\} \Big|
 \Big(1 + \big|F_{j\eps}^{-1}(u_j)\big|\Big) 
 \Big(1 + \big|F_{k\eps}^{-1}(u_k)\big| \Big) < \infty,  
\end{multline*}
where $C^{(j)}(u_1,u_2) = \frac{\partial C(u_1,u_2)}{\partial u_j}$ 
and  $C^{(j,k)}(u_1,u_2) = \frac{\partial^2 C(u_1,u_2)}{\partial u_j \partial u_k}$ 
stand for the first and second order partial derivatives of the copula function. 

Thus provided that for some $\eta > 0$
\begin{equation*}
 C^{(j,k)}(u_1,u_2) = O\Big(\tfrac{1}{u_j^{2\,\eta}(1-u_j)^{2\,\eta}
u_k^{2\,\eta}(1-u_k)^{2\,\eta}}\Big),
\end{equation*}
then we need that the functions 
$f_{j\eps}\big(F_{j\eps}^{-1}(u)\big)(1 + \big|F_{j\eps}^{-1}(u)\big|)$ 
are of order $O(u^{\eta}(1-u)^{\eta})$ and the functions   
$f'_{j\eps} \big(F_{j\eps}^{-1}(u_j) \big) \big(1 + |F_{j\eps}^{-1}(u_j)| \big)^{2}$ 
are bounded. 
%
\end{remark}

\begin{remark}
Parts of our assumptions are reproduced from \cite{hansen2008} because we apply his results about uniform rates of convergence for kernel estimators several times in our proofs. Note that in his Theorem 2 we set $q=\infty$ to simplify the assumptions. Further note that 
if beta mixing coefficients are diminishing exponentially fast then it is sufficient 
to assume $s > 2$ in $\mathbf{(M)}$. 
\end{remark}

\begin{remark} \label{remark-bandwidth}
Note that the bandwidth conditions (\ref{Hansen3}) can be fulfilled iff $2p+2>3d+2\delta$, i.e. in view of assumption $\mathbf{(Bw)}$ iff $2p+2>3d+\frac{2d}{b-1}$. 
Thus if $b > 2d +1$, then for $d=1$ it is sufficient to take $p=1$ and for $d=2$ one can take $p = 3$. In general with increasing dimension $d$ higher smoothness of the unknown functions has to be assumed and higher order local polynomial estimators have to be used. This phenomenon is well known 
in the context of nonparametric inference.  

So in general one can choose  the bandwidth as $h_n\sim n^{-\frac{1}{a}}$, where 
$a \in (3d + \tfrac{2d}{b-1}, 2p+2)$. The problem is that if one wants to take $p$ 
as small as possible, the range of possible values of~$a$ is rather short which 
makes the choice of $a$ rather delicate. To make the choice of $a$ more flexible 
in practice one can for instance assume that $b > 10d + 1$ which (among others) 
includes models for beta mixing coefficients diminishing exponentially fast. 
Now for $d=1$ and $p=1$ one can take $a$ in the interval $(3.1, 4)$. See also 
the bandwidth choice in our simulation study in Section~\ref{sec: simul study}.  
\end{remark}

\begin{remark}
The choice of $c_n$ is a delicate problem in practice. As far as we know 
even in analogous settings \citep[see e.g.][and the references therein]{MSW2009,dette2009goodness,koul2015goodness}   
this problem has not been touched yet. Note that the weight function  
$w_n(\xx)$ is chosen in the simplest possible form 
in order to simplify the presentation of the proof. 
On the other hand in practice it is of interest to use more general 
forms of $\JJ_n$. Further as the density~$f_{\Xb}$ is unknown, 
data-driven procedures to choose~$\JJ_n$ are of interest. 
In the simulation study in Section~\ref{sec: simul study} 
we suggest a data-driven procedure for the choice of the weighting function  
in the case $d=1$. Nevertheless the data driven choice of~$\JJ_n$ 
(in particular for general~$d$) and its theoretical justification calls 
for further research. 
\end{remark}

\medskip

\begin{theorem} \label{thm equiv of Cn}
Suppose that assumptions $\mathbf{(\boldsymbol{\beta})}$,  
$\mathbf{(F_{\beps})}$, $\mathbf{(F_{\XX})}$, $\mathbf{(Bw)}$, $\mathbf{(M)}$,
$\mathbf{(k)}$, $\mathbf{(J_{n})}$ and $\mathbf{(m\bsigma)}$ are satisfied. Then 
$$
 \sup_{(u_1,u_2) \in [0,1]^2} \Big|\sqrt{n}\,\big[\widetilde{C}_{n}(u_1,u_2) - C_{n}^{(or)}(u_1,u_2)\big] \Big| 
 = o_{P}(1).  
$$
\end{theorem}
Note that Theorem~\ref{thm equiv of Cn} together with 
the weak convergence of $\sqrt{n}\,\big[C_{n}^{(or)}- C\big]$ 
\citep[see e.g., Proposition~3.1 of][]{Segers:2012}
implies that that process 
$\widetilde{\mathbb{C}}_{n}=\sqrt{n}\,\big[\widetilde{C}_{n} - C\big]$  
weakly converges in the space of bounded functions 
$\ell^{\infty}([0,1]^2)$ to a centred Gaussian process $G_{C}$,
which can be written as
\begin{equation*} 
 G_{C}(u_1, u_2)  = B_{C}(u_1, u_2)
  - C^{(1)}(u_{1},u_{2})\,B_{C}(u_{1},1) -
 C^{(2)}(u_1, u_2) \,B_{C}(1,u_{2}) \, ,
\end{equation*}
%
where $B_{C}$ is a Brownian bridge on $[0,1]^{2}$ with
covariance function
\begin{equation*} 
 \Es \big[B_{C}(u_{1},u_{2}) B_{C}(u_{1}',u_{2}')\big] = C(u_{1} \wedge u_{1}', u_{2} \wedge u_{2}') - C(u_{1},u_{2})\, C(u_{1}',u_{2}')\,.
\end{equation*}
Nevertheless when one uses this result in applications 
for statistical inference we recommend to replace  
the sample size $n$ with  $W_{n} = \suman w_{ni}$ in the formulas. 
The thing is that the copula is estimated in fact only from $W_{n}$ 
observations and this should be reflected in order to improve 
the finite sample performance of asymptotic inference procedures. 

\bigskip


\subsection{Semiparametric copula estimation}\label{section2.2}
The copula $C$ describes the dependency between the two time series of interest, given the covariate. For applications modeling this dependency structure parametrically is advantageous because a parametric model often gives easier access to interpretations. Goodness-of-fit testing will be considered in the next section. 
\\
Suppose that the joint distribution of 
$(\eps_{1i},\eps_{2i})$ is given by the copula function 
$C(u_1,u_2; \btheta)$, where $\btheta=(\theta_1,\dotsc,\theta_p)\tr$ is an unknown 
parameter that belongs to a parametric space~$\Theta \subset \RR^{p}$. 
In copula settings we are often interested in semiparametric estimation of the parameter~$\btheta$, i.e. 
estimation of $\btheta$ without making any parametric assumption on the marginal 
distributions~$F_{1\eps}$ and~$F_{2\eps}$. The methods of semiparametric 
estimation for i.i.d.\ settings are summarized in \cite{tsukahara_2005}. 
The question of interest is 
what happens if we use the estimated residuals~\eqref{eq: estim residuals} 
instead of the unobserved innovations~$\eps_{ji}$. 
Generally speaking, thanks to Theorem~\ref{thm equiv of Cn} the answer is that using $\hateps_{ji}$ instead of $\eps_{ji}$ does not change the asymptotic 
distribution provided that the parameter of interest can be written as 
a Hadamard differentiable functional of a copula. 

\subsubsection{Method-of-Moments using rank correlation} \label{eq: subsubsec methods of moments}
This method is in a general way described for instance 
in \citet[Section~5.5.1]{embrechts2005quantitative}. To illustrate 
the application of Theorem~\ref{thm equiv of Cn} for this 
method consider that the parameter $\theta$ is one-dimensional.  
Then the method of the inversion of Kendall's tau is a 
very popular method of estimating the unknown parameter. 
For this method the estimator of $\theta$ is given by 
$$
 \widehat{\theta}_{n}^{(ik)} = \tau^{-1}(\widehat{\tau}_{n}), 
$$
where 
$$
 \tau(\theta) = 4\int_{0}^{1}\int_{0}^{1} C(u_1,u_2;\theta)\,dC(u_1,u_2;\theta) - 1
$$
is the theoretical Kendall's tau 
and $\widehat{\tau}_{n}$ is an estimate of Kendall's tau. In our 
settings the Kendall's tau would be computed from the estimated 
residuals $(\hateps_{1i}, \hateps_{2i})$ for which $w_{ni} > 0$. 
By Theorem~\ref{thm equiv of Cn} 
and Hadamard differentiability of Kendall's tau proved in \citet[Lemma~1]{ogv_sjs_2011}, 
the estimators of Kendall's tau based on $\hateps_{ji}$ or on $\eps_{ji}$ are asymptotically equivalent. Thus provided that 
$\tau'(\theta) \neq 0$ one gets that 
$$
 \sqrt{n}\,\big(\widehat{\theta}_{n}^{(ik)} - \theta \big) 
 \inDist \mathsf{N}\Big(0, \tfrac{\sigma_{\tau}^{2}}{[\tau'(\theta)]^2}\Big), 
 \quad \text{where} \quad 
 \sigma_{\tau}^{2} = 
   \var\Big\{ 8\,C(U_{11},U_{21};\theta) - 4\,U_{11} -  4\,U_{21}\Big\},  
$$
and 
\begin{equation} \label{eq: Ui}
 \big(U_{11}, U_{21}\big) = \big(F_{1\eps}(\eps_{11}), F_{2\eps}(\eps_{21})\big). 
\end{equation}
Analogously one can show that working with residuals has asymptotically negligible effects 
also for the method of moments introduced in \cite{brahimi2012semiparametric}. 


\smallskip 

\subsubsection{Minimum distance estimation}
Here one can follow for instance \citet[Section~3.2]{tsukahara_2005}.  
Note that thanks to Theorem~\ref{thm equiv of Cn} 
the proof of Theorem~3 of \cite{tsukahara_2005} does not change 
when $C_{n}^{(or)}$ is replaced with $\widetilde{C}_{n}$. 
Thus provided assumptions (B.1)-(B.5) of \cite{tsukahara_2005} are 
satisfied with 
$\bdelta(u_1,u_2;\btheta) = \frac{\partial C(u_1,u_2;\btheta)}{\partial \btheta}$, then the estimator defined as 
$$
 \hatbtheta_n^{(md)} = \argmin_{\tt \in \Theta} \iint_{[0,1]^2} \big(\widetilde{C}_{n}(u_1,u_2) - C(u_1,u_2; \tt)\big)^2\, 
 du_1\,du_2 
$$
is asymptotically normal and satisfies 
$$
 \sqrt{n}\Big(\hatbtheta_n^{(md)} - \btheta \Big) 
 \inDist \mathsf{N}\Big(\mathbf{0}_{p}, \bSigma^{(md)} \Big), 
$$
where 
%
\begin{align*}
 \bSigma^{(md)} &= \var\bigg\{ \iint_{[0,1]^2}  \bgamma(u_1,u_2;\btheta) 
 \Big[\ind\{U_{11} \leq u_1,U_{21} \leq u_2\} 
\\ 
& \qquad \qquad \qquad \qquad  \qquad \qquad    
 - \sum_{j=1}^{2}  C^{(j)}(u_1,u_2; \btheta)\, \ind\{U_{j1} \leq u_j\}
   \Big]\,du_1\,du_2  
   \Bigg\},  
\end{align*}
with 
$$
 \bgamma(u_1,u_2;\btheta) = \Bigg[\iint_{[0,1]^2} 
 \bdelta(v_1,v_2;\btheta)\, \bdelta\tr(v_1,v_2;\btheta)\,   dv_1\,dv_2 
 \Bigg]^{-1} \bdelta(u_1,u_2;\btheta). 
$$



\subsubsection{M-estimator, rank approximate Z-estimators}
To define a general $M$-estimator let us introduce 
\begin{equation} \label{eq: Utilde}
  \big(\widetilde{U}_{1i}, \widetilde{U}_{2i} \big) = \tfrac{W_n}{W_n+1}\Big(\hatF_{1\hateps}\big(\hateps_{1i}\big), \hatF_{2\hateps}\big(\hateps_{2i}\big)\Big)
\end{equation}
that can be viewed as estimates of the unobserved $(U_{1i}, U_{2i})$. 
Note that the multiplier $\tfrac{W_n}{W_n+1}$ is introduced 
in order to have both of the coordinates of the vector 
$\big(\widetilde{U}_{1i}, \widetilde{U}_{2i} \big)$ 
bounded away from zero and one. The $M$-estimator of the parameter~$\btheta$ 
is now defined as  
\begin{equation*} 
 \hatbtheta_n = \argmin_{\tt \in \Theta} \, 
 \suman w_{ni}\, 
 \rho \big(\widetilde{U}_{1i}, \widetilde{U}_{2i}; \tt \big)
\end{equation*}
where  $\rho(u_1,u_2;\btheta)$ is a given loss function. 
This class of estimators includes among others the 
\textit{pseudo-maximum likelihood estimators} 
($\hatbtheta_n^{(pl)}$), 
for which $\rho(u_1,u_2;\btheta) = - \log c(u_1,u_2;\btheta)$, 
with $c(\cdot)$ being the copula density function.  

Note that the estimator $\hatbtheta_n$ is usually searched for 
as a solution to the estimating equations 
\begin{equation}
 \label{eq: estim equations} 
 \suman w_{ni}\, 
 \bphi\big(\widetilde{U}_{1i}, \widetilde{U}_{2i};\hatbtheta_n
 \big) = \boldsymbol{0}_{p}, 
\end{equation}
where $\bphi(u_1,u_2;\btheta) = \partial \rho(u_1,u_2;\btheta)/\partial \btheta$. In \cite{tsukahara_2005} the estimator defined as the 
solution of~\eqref{eq: estim equations} is called a rank approximate 
$Z$-estimator.

In what follows we give general assumptions under which 
there exists a consistent root ($\hatbtheta_n$) of the estimating 
equations~\eqref{eq: estim equations} that is asymptotically equivalent 
to the consistent root ($\hatbtheta_n^{(or)}$) of 
the `oracle' estimating equations given by 
%
\begin{equation}
 \label{eq: estim equations orac} 
\suman 
\bphi\Big(\widehat{U}_{1i}, \widehat{U}_{2i}; \hatbtheta_n^{(or)}
 \Big) = \boldsymbol{0}_{p},  
\end{equation}
%
where 
\begin{equation} \label{eq: Uhat}
  \big(\widehat{U}_{1i}, \widehat{U}_{2i} \big) = \tfrac{n}{n+1}\big(\hatF_{1\eps}(\eps_{1i}), \hatF_{2\eps}(\eps_{2i})\big)
\end{equation}
are the standard pseudo-observations calculated from the 
unobserved innovations and their marginal empirical 
distribution functions $\hatF_{j\eps}(y)$.  

Unfortunately, these general assumptions exclude some useful models (e.g.\ pseudo-maxi\-mum likelihood estimator 
in the Clayton family of copulas) for which the 
function $\bphi(u_1,u_2;\btheta)$ viewed as a function of $(u_1,u_2)$ 
is unbounded. 
The reason is that for empirical distribution 
functions calculated from estimated residuals $\hateps_{ji}$  
we lack some of the sophisticated 
results that are available for empirical distribution functions calculated from 
(true) innovations~$\eps_{ji}$. For such copula families one can use for instance the 
Method-of-Moments using rank correlation (see Section~\ref{eq: subsubsec methods of moments}) 
to stay on the safe side. Nevertheless the simulation study in Section~\ref{sec: simul study} 
suggests that the pseudo-maxi\-mum likelihood estimation can be used also for the Clayton copula 
(and probably also for other families of copulas with non-zero tail dependence) 
provided that the dependence is not very strong. 


\subsubsection*{Regularity assumptions}
In what follows let $\btheta$ stand for the true value of the 
parameter and $V(\btheta)$ for an open neighbourhood of~$\btheta$.  
 
\begin{itemize}
\item[$\mathbf{(Id)}$] $\btheta$ is a unique minimizer of  the 
function $r(\tt) = \Es \rho(U_{1i}, U_{2i}; \tt)$ and $\btheta$ is an 
inner point of $\Theta$. 
{ 
\item[$\mathbf{(\bphi)}$] There exists~$V(\btheta)$ 
such that for each $l_1,l_2 \in \{1,\dotsc,p\}$ the functions 
$\phi_{l_1}(u_1,u_2;\tt) = \tfrac{\partial \rho(u_1,u_2;\tt)}{\partial t_{l_1}}$  
and 
$\phi_{l_1,l_2}(u_1,u_2;\tt) = \tfrac{\partial \rho(u_1,u_2;\tt)}{\partial t_{l_1}
\partial t_{l_2}}$ are uniformly continuous in $(u_1,u_2)$ uniformly in~$\tt \in V(\btheta)$ 
and of uniformly bounded Hardy-Kraus variation \citep[see e.g.,][]{berghaus2017weak}. 
\item[$\mathbf{(\bphi^{(j)})}$] There exists~$V(\btheta)$ and a function $h(u_1,u_2)$ 
such that for each $\tt \in V(\btheta)$
$$
 \max_{j=1,2}\max_{l =1,\dotsc,p} \big|\phi_{l}^{(j)}(u_1,u_2;\tt)\big| \leq h(u_1,u_2), 
\text{ where } \phi_{l}^{(j)}(u_1,u_2;\tt) = \tfrac{\partial \phi_{l}(u_1,u_2;\tt)}{\partial u_j}
$$
and $ \Es  h(U_{11},U_{21}) < \infty$. 
}
\item[$\mathbf{(\bGamma)}$] 
Each element of the (matrix) function 
$\bGamma(\tt) = \Es \frac{\partial \bphi(U_1, U_2; \tt)}{\partial \tt \tr}$ 
is a continuous function on $V(\thetab)$ and the matrix $\bGamma = \bGamma(\btheta)$  
is positively definite. 
\end{itemize}

\begin{theorem} \label{thm equiv of param estim}
Suppose that the assumptions of Theorem~\ref{thm equiv of Cn} 
are satisfied and that 
also $\mathbf{(Id)}$, $\mathbf{(\bphi)}$,  $\mathbf{(\bphi^{(j)})}$, 
and $\mathbf{(\bGamma)}$ hold.  
{Then with probability going to one there exists a consistent root $\hatbtheta_n$ 
of the estimating equations~\eqref{eq: estim equations},  
which satisfies 
\begin{equation} \label{eq: thetan converges in distrib}
 \sqrt{n}\,\Big(\hatbtheta_n - \btheta \Big) \inDist 
  \mathsf{N}_{p}\big(\mathbf{0}_{p}, \bGamma^{-1} \,\bSigma\, \bGamma^{-1}\big), 
\end{equation}
where 
\begin{align*}
 \bSigma &= \var\Bigg\{\bphi\big(U_{11}, U_{21}; \btheta\big) 
  + \iint \big[\ind\{U_{11} \leq v_1\} - v_{1}\big]\,
    \tfrac{\partial \bphi(v_{1}, v_{2}; \btheta)}{\partial v_1}\,dC(v_1,v_2; \btheta) 
\\
 & \qquad \qquad  + \iint \big[\ind\{U_{21} \leq v_2\} - v_{2} \big]\,\tfrac{\partial \bphi(v_{1}, v_{2}; \btheta)}{\partial v_2}\,dC(v_1,v_2; \btheta) 
 \Bigg\}.  
\end{align*}
}
\end{theorem}
{The proof of Theorem~\ref{thm equiv of param estim} is given in Appendix~B. 
Note that the asymptotic distribution of the estimator $\hatbtheta_n$ coincides 
with the distribution given in Section 4 of \cite{genest_et_al_1995} that 
corresponds to the  consistent root $\hatbtheta_n^{(or)}$ of the estimating equations \eqref{eq: estim equations orac}. 
Thus using the residuals instead of the true innovations has asymptotically negligible effect  
on the (first-order) asymptotic properties. In fact, it can be even shown that 
both $\hatbtheta_n$ and $\hatbtheta_n^{(or)}$ have the same asymptotic representations  
and thus 
$$
 \sqrt{n}\big(\hatbtheta_n - \hatbtheta_n^{(or)} \big) = o_{P}(1). 
$$
}
%


\subsection{Goodness-of-fit testing}\label{section2.3}
When modeling multivariate data using copulas parametrically one needs to choose 
a suitable family of copulas. When choosing the copula family 
tests of goodness-of-fit are often a useful tool. Thus we are 
interested in testing 
%
%
$ H_{0}:\, C \in \mathcal{C}_{0}$, 
where $\mathcal{C}_{0} = \{C_{\btheta}, \btheta \in \Theta \}$ is a
given parametric family of copulas.

Many testing methods have been proposed 
\citep[see e.g.][and the references therein]{genest-et-al-gof-2008, 
kojadinovic2009tests}. The most standard ones are based on the 
comparison of nonparametric and parametric estimators of a copula. 
For instance the Cram\'{e}r-von Mises statistic is given by 
\begin{equation} \label{gof-statistics test}
 S_n = \iint \big[\widetilde{C}_n(u_1,u_2) - C(u_1,u_2;\hatbtheta_{n})\big]^2 
   \,d \widetilde{C}_n(u_1,u_2), 
\end{equation}
where $\hatbtheta_{n}$ is an estimate of the unknown
parameter~$\btheta$. As the asymptotic distributions 
of $\widetilde{C}_n(u_1,u_2)$ and $\hatbtheta_{n}$ are 
the same as the asymptotic distribution of 
$\widetilde{C}_n^{(or)}(u_1,u_2)$ and $\hatbtheta_{n}^{(or)}$ 
we suggest that the significance of the test statistic can be 
assessed in the same way as in i.i.d.\ settings. Thus one can use for instance  
the parametric bootstrap by simply generating 
independent and identically distributed observations from 
the copula function $C(u_1,u_2;\hatbtheta_{n})$. The test 
statistic is then simply recalculated from this observations 
in the same way as if we directly observed the innovations. 
The only difference 
is that instead of generating $n$ observations we recommend to 
generate only $W_n$ observations. 

Similar remarks hold when testing other hypotheses about the copula such as symmetry, for instance. 
Note that testing $H_0: C(u_1,u_2)\equiv u_1u_2 $ provides a test for conditional independence 
of the two time series, given the covariate.

\section{Simulation study}\label{sec: simul study}

A small Monte Carlo study was conducted in order to compare the 
semiparametric estimators based on the residuals   
with the `oracle' estimators based on (unobserved) innovations. The inversion of Kendall's tau (IK) method and  the  maximum pseudo-likelihood (MPL) method were considered for  
the following five copula families: Clayton, Frank, Gumbel, normal, and Student with 4 degrees of freedom. 
The values of the parameters are chosen so that they correspond to the  Kendall's tau 
$\tau=0.25$, $0.50$ and $0.75$.  
The data were simulated from the following four models:
{\small %
 \begin{align*}
Y_{1i}&= (0.5+0.4 \mathrm{e}^{-0.8X_i^2})X_i + \sqrt{1+0.2 X_i^2} \varepsilon_{1i}, & 
Y_{2i}&= 0.5-0.5 X_i+\sqrt{1+0.4 X_i^2} \varepsilon_{2i},\tag{Mod 1} \label{m-2h}\\
%
 Y_{1i}&=0.7 Y_{1,i-1} + \eps_{1i}, &\quad Y_{2i}&=-0.5Y_{2,i-1} + \eps_{2i},\label{m-3}\tag{Mod 2}
\\
 Y_{1i}&=0.5 \frac{Y_{1,i-1}}{ 1+0.1Y_{1,i-1}^2}+ \eps_{1i}, &\quad Y_{2i}&=-0.4 Y_{2,i-1} + \eps_{2i},\label{m-4}\tag{Mod 3}
\\
 Y_{1i}&=\sigma_{1i}\eps_{1i}, \quad \sigma_{1i}^2=1+0.3 Y_{1,i-1}^2,&\quad Y_{2i}&=\sigma_{2i}\eps_{2i}, \quad \sigma_{2i}^2=5+0.2 Y_{2,i-1}^2\tag{Mod 4},
 \label{m-arch} 
\end{align*}}
where  the innovations $\eps_{ji}$, $j=1,2$, follow 
marginally the  standard  normal distribution, 
and  $X_i$ is an exogenous variable following  the AR  model $X_i=0.6X_{i-1}+\xi_i$ with $\xi_i$ 
being i.i.d.\ from a standard normal distribution. The simulations were conducted also for innovations $\eps_{ji}$, $j=1,2$ with Student marginals with 5 degrees of freedom, but the corresponding results are very similar. For brevity of the paper we do not present them here. 

\begin{table}[ht]
\centering
\begin{footnotesize}
\begin{tabular}{ccc|rrr|rrr|rrr}
    &&& \multicolumn{3}{c|}{$n=200$}&\multicolumn{3}{c|}{$n=500$}& \multicolumn{3}{c}{$n=1000$}\\
 Model &$\tau$& estim & bias & SD &RMSE& bias & SD &RMSE& bias & SD &RMSE \\
 \hline
\hline 
\parbox[c]{2mm}{\multirow{6}{*}{\vspace{-2mm}\rotatebox[origin=c]{90}{\footnotesize{Known innovations}}}} & 0.25 &$\widehat{\theta}_{n}^{(ik,or)}$& -0.03 & 3.25 & 3.25 & 0.10 & 2.47 & 2.47 & 0.12 & 1.86 & 1.87 \\ 
   & 0.25 & $\widehat{\theta}_{n}^{(pl,or)}$  & 0.51 & 3.00 & 3.04 & 0.35 & 2.15 & 2.18 & 0.24 & 1.69 & 1.70 \\ 
   & 0.50& $\widehat{\theta}_{n}^{(ik,or)}$ & 0.01 & 2.64 & 2.64 & 0.06 & 2.03 & 2.03 & 0.07 & 1.52 & 1.52 \\ 
   & 0.50& $\widehat{\theta}_{n}^{(pl,or)}$  & 0.09 & 2.47 & 2.47 & 0.08 & 1.84 & 1.85 & 0.04 & 1.39 & 1.39 \\ 
   & 0.75 & $\widehat{\theta}_{n}^{(ik,or)}$ & 0.01 & 1.58 & 1.58 & 0.05 & 1.19 & 1.19 & 0.02 & 0.89 & 0.89 \\ 
   & 0.75 & $\widehat{\theta}_{n}^{(pl,or)}$ & -0.28 & 1.48 & 1.50 & -0.17 & 1.10 & 1.11 & -0.12 & 0.80 & 0.81 \\ 
   \hline
\hline
\multirow{6}{*}{\large 1} & 0.25 & $\widehat{\theta}_{n}^{(ik)}$ & -0.08 & \bf 4.66 & \bf 4.66 & \bf -0.22 & \bf 2.97 & \bf 2.97 & \bf -0.16 & \bf 2.06 & \bf 2.06 \\ 
   & 0.25 & $\widehat{\theta}_{n}^{(pl)}$ & \bf 0.62 & 4.15 & 4.19 & 0.07 & 2.62 & 2.62 & -0.02 & 1.82 & 1.82 \\ 
   & 0.50& $\widehat{\theta}_{n}^{(ik)}$ & -0.46 & \bf 3.94 & \bf 3.97 & -0.41 & \bf 2.48 & \bf 2.51 & -0.25 & \bf 1.74 & \bf 1.76 \\ 
   & 0.50& $\widehat{\theta}_{n}^{(pl)}$ & \bf -0.90 & 3.59 & 3.70 & \bf -0.81 & 2.25 & 2.39 & \bf -0.55 & 1.60 & 1.69 \\ 
   & 0.75 & $\widehat{\theta}_{n}^{(ik)}$ & -1.04 & 2.45 & 2.66 & -0.85 & 1.55 & 1.77 & -0.59 & 1.07 & 1.22 \\ 
   & 0.75 & $\widehat{\theta}_{n}^{(pl)}$ & \bf -3.00 & \bf 2.66 & \bf 4.01 & \bf -2.23 & \bf 1.59 & \bf 2.74 & \bf -1.57 & \bf 1.15 & \bf 1.94 \\ 
   \hline
\multirow{6}{*}{\large 2} & 0.25 & $\widehat{\theta}_{n}^{(ik)}$ & \bf -0.43 & \bf 4.78 & \bf 4.79 & -0.05 & \bf 2.93 & \bf 2.92 & 0.07 & \bf 2.08 & \bf 2.08 \\ 
   & 0.25 & $\widehat{\theta}_{n}^{(pl)}$ & 0.26 & 4.30 & 4.31 & \bf 0.25 & 2.58 & 2.59 & \bf 0.15 & 1.90 & 1.90 \\ 
   & 0.50& $\widehat{\theta}_{n}^{(ik)}$ & -0.91 & \bf 3.93 & \bf 4.03 & -0.24 & \bf 2.40 & \bf 2.41 & -0.09 & \bf 1.71 & \bf 1.72 \\ 
   & 0.50& $\widehat{\theta}_{n}^{(pl)}$ & \bf -1.50 & 3.62 & 3.92 & \bf -0.57 & 2.21 & 2.29 & \bf -0.36 & 1.60 & 1.64 \\ 
   & 0.75 & $\widehat{\theta}_{n}^{(ik)}$ & -1.96 & 2.63 & 3.27 & -0.70 & 1.52 & 1.68 & -0.39 & 1.05 & 1.12 \\ 
   & 0.75 & $\widehat{\theta}_{n}^{(pl)}$ & \bf -4.63 & \bf 3.19 & \bf 5.62 & \bf -2.14 & \bf 1.84 & \bf 2.82 & \bf -1.27 & \bf 1.16 & \bf 1.72 \\ 
   \hline
\multirow{6}{*}{\large 3} & 0.25 & $\widehat{\theta}_{n}^{(ik)}$ & \bf -0.43 & \bf 4.81 & \bf 4.83 & -0.09 & \bf 2.91 & \bf 2.91 & 0.03 & \bf 2.09 & \bf 2.09 \\ 
   & 0.25 & $\widehat{\theta}_{n}^{(pl)}$ & 0.24 & 4.37 & 4.38 & \bf 0.19 & 2.56 & 2.57 & \bf 0.11 & 1.90 & 1.90 \\ 
   & 0.50& $\widehat{\theta}_{n}^{(ik)}$ & -0.93 & \bf 3.97 & \bf 4.07 & -0.32 & \bf 2.41 & \bf 2.43 & -0.16 & \bf 1.72 & \bf 1.72 \\ 
   & 0.50& $\widehat{\theta}_{n}^{(pl)}$ & \bf -1.52 & 3.70 & 4.00 & \bf -0.66 & 2.20 & 2.30 & \bf -0.46 & 1.61 & 1.67 \\ 
   & 0.75 & $\widehat{\theta}_{n}^{(ik)}$ & -1.85 & 2.61 & 3.20 & -0.82 & 1.53 & 1.73 & -0.53 & 1.04 & 1.16 \\ 
   & 0.75 & $\widehat{\theta}_{n}^{(pl)}$ & \bf -4.39 & \bf 3.05 & \bf 5.35 & \bf -2.25 & \bf 1.78 & \bf 2.86 & \bf -1.46 & \bf 1.14 & \bf 1.85 \\ 
   \hline
\multirow{6}{*}{\large 4} & 0.25 & $\widehat{\theta}_{n}^{(ik)}$ & \bf -0.49 & \bf 4.85 & \bf 4.87 & -0.09 & \bf 2.93 & \bf 2.93 & 0.02 & \bf 2.10 & \bf 2.10 \\ 
   & 0.25 & $\widehat{\theta}_{n}^{(pl)}$ & 0.13 & 4.37 & 4.37 & \bf 0.14 & 2.58 & 2.59 & \bf 0.06 & 1.90 & 1.90 \\ 
   & 0.50& $\widehat{\theta}_{n}^{(ik)}$ & -0.82 & \bf 3.99 & \bf 4.07 & -0.25 & \bf 2.40 & \bf 2.41 & -0.12 & \bf 1.73 & \bf 1.73 \\ 
   & 0.50& $\widehat{\theta}_{n}^{(pl)}$ & \bf -1.54 & 3.70 & 4.01 & \bf -0.80 & 2.22 & 2.36 & \bf -0.53 & 1.60 & 1.69 \\ 
   & 0.75 & $\widehat{\theta}_{n}^{(ik)}$ & -1.22 & 2.57 & 2.84 & -0.49 & 1.48 & 1.56 & -0.28 & 1.04 & 1.08 \\ 
   & 0.75 & $\widehat{\theta}_{n}^{(pl)}$ & \bf -3.43 & \bf 2.76 & \bf 4.40 & \bf -1.93 & \bf 1.65 & \bf 2.54 & \bf -1.20 & \bf 1.10 & \bf 1.62 \\ 
   \hline
\hline
\end{tabular}
\caption{Estimation for Clayton copula with  normal marginals (100 multiples of bias, SD and RMSE)} 
\label{T.clayton.norm}
\end{footnotesize}
\end{table}

\begin{table}[ht]
\centering
\begin{footnotesize}
\begin{tabular}{ccc|rrr|rrr|rrr}
   && & \multicolumn{3}{c|}{$n=200$}&\multicolumn{3}{c|}{$n=500$}& \multicolumn{3}{c}{$n=1000$}\\
 Model &$\tau$& estim&bias & SD &RMSE& bias & SD &RMSE& bias & SD &RMSE \\
 \hline
\hline
 \parbox[c]{2mm}{\multirow{6}{*}{\vspace{-2mm}\rotatebox[origin=c]{90}{\footnotesize{Known innovations}}}} & 0.25 &$\widehat{\theta}_{n}^{(ik,or)}$ & -0.01 & 3.16 & 3.16 & -0.05 & 2.33 & 2.33 & -0.14 & 1.70 & 1.71 \\ 
   & 0.25 & $\widehat{\theta}_{n}^{(pl,or)}$ & 0.04 & 3.16 & 3.15 & -0.03 & 2.32 & 2.32 & -0.12 & 1.70 & 1.70 \\ 
   & 0.50&$\widehat{\theta}_{n}^{(ik,or)}$ & -0.02 & 2.37 & 2.37 & -0.01 & 1.73 & 1.73 & -0.09 & 1.28 & 1.28 \\ 
   & 0.50& $\widehat{\theta}_{n}^{(pl,or)}$ & 0.00 & 2.34 & 2.34 & -0.02 & 1.72 & 1.72 & -0.08 & 1.27 & 1.27 \\ 
   & 0.75 &$\widehat{\theta}_{n}^{(ik,or)}$ & -0.02 & 1.18 & 1.18 & 0.00 & 0.87 & 0.87 & -0.03 & 0.64 & 0.64 \\ 
   & 0.75 & $\widehat{\theta}_{n}^{(pl,or)}$ & -0.13 & 1.17 & 1.17 & -0.07 & 0.87 & 0.87 & -0.07 & 0.64 & 0.64 \\   
   \hline
\hline
\multirow{6}{*}{\large 1} & 0.25 & $\widehat{\theta}_{n}^{(ik)}$ & \bf -0.23 & \bf 4.54 & \bf 4.54 & \bf -0.11 & \bf 2.82 & \bf 2.82 & \bf -0.05 & \bf 1.92 & \bf 1.92 \\ 
   & 0.25 & $\widehat{\theta}_{n}^{(pl)}$ & -0.12 & 4.52 & 4.52 & -0.05 & 2.81 & 2.81 & -0.03 & 1.90 & 1.90 \\ 
   & 0.50& $\widehat{\theta}_{n}^{(ik)}$ & \bf -0.49 & \bf 3.46 & \bf 3.50 & \bf -0.32 & \bf 2.18 & \bf 2.20 & \bf -0.22 & \bf 1.43 & \bf 1.44 \\ 
   & 0.50& $\widehat{\theta}_{n}^{(pl)}$ & -0.47 & 3.40 & 3.43 & -0.30 & 2.15 & 2.17 & -0.21 & 1.42 & 1.43 \\ 
   & 0.75 & $\widehat{\theta}_{n}^{(ik)}$ & -0.97 & \bf 1.87 & 2.11 & -0.69 & 1.15 & 1.34 & -0.53 & 0.74 & 0.91 \\ 
   & 0.75 & $\widehat{\theta}_{n}^{(pl)}$ & \bf -1.22 & 1.84 & \bf 2.21 & \bf -0.81 & \bf 1.16 & \bf 1.41 & \bf -0.60 & \bf 0.75 & \bf 0.96 \\ 
   \hline
\multirow{6}{*}{\large 2} & 0.25 & $\widehat{\theta}_{n}^{(ik)}$ & \bf -0.28 & \bf 4.48 & \bf 4.49 & \bf -0.15 & \bf 2.78 & \bf 2.79 & \bf -0.21 & \bf 1.88 & \bf 1.89 \\ 
   & 0.25 & $\widehat{\theta}_{n}^{(pl)}$ & -0.17 & 4.47 & 4.47 & -0.12 & 2.77 & 2.77 & -0.19 & 1.88 & 1.89 \\ 
   & 0.50& $\widehat{\theta}_{n}^{(ik)}$ & \bf -0.77 & \bf 3.44 & \bf 3.53 & -0.29 & \bf 2.13 & \bf 2.14 & \bf -0.24 & \bf 1.41 & \bf 1.43 \\ 
   & 0.50& $\widehat{\theta}_{n}^{(pl)}$ & -0.75 & 3.40 & 3.48 & \bf -0.31 & 2.10 & 2.12 & -0.24 & 1.40 & 1.42 \\ 
   & 0.75 & $\widehat{\theta}_{n}^{(ik)}$ & -1.65 & 2.20 & 2.75 & -0.66 & \bf 1.18 & 1.35 & -0.38 & 0.75 & 0.84 \\ 
   & 0.75 & $\widehat{\theta}_{n}^{(pl)}$ & \bf -1.90 & \bf 2.20 & \bf 2.91 & \bf -0.78 & 1.18 & \bf 1.41 & \bf -0.43 & \bf 0.75 & \bf 0.87 \\ 
   \hline
\multirow{6}{*}{\large 3} & 0.25 & $\widehat{\theta}_{n}^{(ik)}$ & \bf -0.33 & \bf 4.53 & \bf 4.54 & \bf -0.17 & \bf 2.77 & \bf 2.77 & \bf -0.24 & \bf 1.89 & \bf 1.91 \\ 
   & 0.25 & $\widehat{\theta}_{n}^{(pl)}$ & -0.23 & 4.53 & 4.53 & -0.14 & 2.75 & 2.75 & -0.22 & 1.89 & 1.90 \\ 
   & 0.50& $\widehat{\theta}_{n}^{(ik)}$ & \bf -0.83 & \bf 3.48 & \bf 3.58 & -0.37 & \bf 2.09 & \bf 2.12 & \bf -0.32 & \bf 1.42 & \bf 1.45 \\ 
   & 0.50& $\widehat{\theta}_{n}^{(pl)}$ & -0.81 & 3.44 & 3.53 & \bf -0.38 & 2.06 & 2.10 & -0.32 & 1.41 & 1.44 \\ 
   & 0.75 & $\widehat{\theta}_{n}^{(ik)}$ & -1.62 & \bf 2.15 & 2.70 & -0.77 & \bf 1.14 & 1.37 & -0.51 & 0.76 & 0.92 \\ 
   & 0.75 & $\widehat{\theta}_{n}^{(pl)}$ & \bf -1.86 & 2.14 & \bf 2.84 & \bf -0.89 & 1.14 & \bf 1.44 & \bf -0.57 & \bf 0.77 & \bf 0.96 \\ 
   \hline
\multirow{6}{*}{\large 4} & 0.25 & $\widehat{\theta}_{n}^{(ik)}$ & \bf -0.37 & \bf 4.56 & \bf 4.57 & \bf -0.16 & \bf 2.79 & \bf 2.80 & \bf -0.22 & \bf 1.90 & \bf 1.91 \\ 
   & 0.25 & $\widehat{\theta}_{n}^{(pl)}$ & -0.26 & 4.54 & 4.54 & -0.13 & 2.79 & 2.79 & -0.20 & 1.90 & 1.91 \\ 
   & 0.50& $\widehat{\theta}_{n}^{(ik)}$ & \bf -0.76 & \bf 3.48 & \bf 3.56 & -0.30 & \bf 2.13 & \bf 2.15 & \bf -0.25 & \bf 1.43 & \bf 1.46 \\ 
   & 0.50& $\widehat{\theta}_{n}^{(pl)}$ & -0.73 & 3.43 & 3.50 & \bf -0.31 & 2.11 & 2.13 & -0.25 & 1.42 & 1.44 \\ 
   & 0.75 & $\widehat{\theta}_{n}^{(ik)}$ & -1.11 & \bf 2.05 & 2.34 & -0.48 & \bf 1.15 & 1.24 & -0.30 & \bf 0.76 & 0.81 \\ 
   & 0.75 & $\widehat{\theta}_{n}^{(pl)}$ & \bf -1.33 & 2.01 & \bf 2.41 & \bf -0.58 & 1.14 & \bf 1.28 & \bf -0.35 & 0.76 & \bf 0.83 \\ 
   \hline
\hline
\end{tabular}
\caption{Estimation for Frank copula with  normal marginals (100 multiples of bias, SD and RMSE)} 
\label{T.frank.norm}
\end{footnotesize}
\end{table}

\begin{table}[ht]
\centering
\begin{footnotesize}
\begin{tabular}{ccc|rrr|rrr|rrr}
   && &     \multicolumn{3}{c|}{$n=200$}&\multicolumn{3}{c|}{$n=500$}& \multicolumn{3}{c}{$n=1000$}\\
 Model &$\tau$& estim&bias & SD &RMSE& bias & SD &RMSE& bias & SD &RMSE \\
 \hline
\hline
  \parbox[c]{2mm}{\multirow{6}{*}{\vspace{-2mm}\rotatebox[origin=c]{90}{\footnotesize{Known innovations}}}}& 0.25 & $\widehat{\theta}_{n}^{(ik,or)}$ & 0.01 & 3.19 & 3.19 & 0.13 & 2.43 & 2.44 & 0.08 & 1.88 & 1.88 \\ 
   & 0.25 & $\widehat{\theta}_{n}^{(pl,or)}$  & 0.44 & 3.01 & 3.04 & 0.38 & 2.37 & 2.40 & 0.24 & 1.81 & 1.82 \\ 
   & 0.50&  $\widehat{\theta}_{n}^{(ik,or)}$ & 0.02 & 2.58 & 2.58 & 0.11 & 1.96 & 1.97 & 0.02 & 1.49 & 1.49 \\ 
   & 0.50& $\widehat{\theta}_{n}^{(pl,or)}$  & 0.24 & 2.42 & 2.43 & 0.27 & 1.89 & 1.91 & 0.12 & 1.44 & 1.44 \\ 
   & 0.75 & $\widehat{\theta}_{n}^{(ik,or)}$ & 0.02 & 1.48 & 1.48 & 0.06 & 1.12 & 1.12 & 0.00 & 0.84 & 0.84 \\ 
   & 0.75 & $\widehat{\theta}_{n}^{(pl,or)}$  & -0.06 & 1.35 & 1.36 & 0.02 & 1.05 & 1.05 & -0.03 & 0.78 & 0.78 \\ 
   \hline
\hline
\multirow{6}{*}{\large 1}& 0.25 & $\widehat{\theta}_{n}^{(ik)}$ & \bf -0.36 & \bf 4.76 & \bf 4.78 & 0.06 & \bf 3.06 & \bf 3.05 & \bf -0.09 & \bf 2.06 & \bf 2.06 \\ 
   & 0.25 & $\widehat{\theta}_{n}^{(pl)}$ & 0.24 & 4.68 & 4.68 & \bf 0.37 & 2.92 & 2.94 & 0.08 & 2.01 & 2.01 \\ 
   & 0.50& $\widehat{\theta}_{n}^{(ik)}$ & \bf -0.56 & \bf 3.92 & \bf 3.96 & \bf -0.17 & \bf 2.45 & \bf 2.46 & \bf -0.22 & \bf 1.69 & \bf 1.70 \\ 
   & 0.50& $\widehat{\theta}_{n}^{(pl)}$ & -0.36 & 3.83 & 3.84 & -0.10 & 2.35 & 2.35 & -0.20 & 1.65 & 1.66 \\ 
   & 0.75 & $\widehat{\theta}_{n}^{(ik)}$ & -0.85 & \bf 2.36 & 2.50 & -0.52 & \bf 1.42 & 1.51 & -0.49 & \bf 1.01 & 1.12 \\ 
   & 0.75 & $\widehat{\theta}_{n}^{(pl)}$ & \bf -1.35 & 2.32 & \bf 2.69 & \bf -0.84 & 1.36 & \bf 1.60 & \bf -0.73 & 0.99 & \bf 1.22 \\ 
   \hline
\multirow{6}{*}{\large 2} & 0.25 & $\widehat{\theta}_{n}^{(ik)}$ & -0.16 & \bf 4.58 & \bf 4.58 & 0.02 & \bf 2.91 & \bf 2.91 & 0.04 & \bf 2.10 & \bf 2.10 \\ 
   & 0.25 & $\widehat{\theta}_{n}^{(pl)}$ & \bf 0.49 & 4.42 & 4.45 & \bf 0.32 & 2.86 & 2.88 & \bf 0.20 & 2.03 & 2.04 \\ 
   & 0.50& $\widehat{\theta}_{n}^{(ik)}$ & \bf -0.66 & \bf 3.77 & \bf 3.82 & \bf -0.14 & \bf 2.36 & \bf 2.36 & \bf -0.09 & \bf 1.67 & \bf 1.68 \\ 
   & 0.50& $\widehat{\theta}_{n}^{(pl)}$ & -0.50 & 3.61 & 3.64 & -0.09 & 2.30 & 2.30 & -0.05 & 1.62 & 1.62 \\ 
   & 0.75 & $\widehat{\theta}_{n}^{(ik)}$ & -1.61 & 2.50 & 2.97 & -0.52 & 1.43 & 1.52 & -0.32 & \bf 0.99 & 1.04 \\ 
   & 0.75 & $\widehat{\theta}_{n}^{(pl)}$ & \bf -2.37 & \bf 2.52 & \bf 3.46 & \bf -0.95 & \bf 1.45 & \bf 1.73 & \bf -0.55 & 0.98 & \bf 1.13 \\ 
   \hline
\multirow{6}{*}{\large 3} & 0.25 & $\widehat{\theta}_{n}^{(ik)}$ & -0.18 & \bf 4.57 & \bf 4.57 & 0.01 & \bf 2.93 & \bf 2.92 & 0.02 & \bf 2.11 & \bf 2.11 \\ 
   & 0.25 & $\widehat{\theta}_{n}^{(pl)}$ & \bf 0.46 & 4.41 & 4.43 & \bf 0.31 & 2.87 & 2.88 & \bf 0.18 & 2.03 & 2.03 \\ 
   & 0.50& $\widehat{\theta}_{n}^{(ik)}$ & \bf -0.66 & \bf 3.73 & \bf 3.78 & \bf -0.18 & \bf 2.36 & \bf 2.37 & \bf -0.16 & \bf 1.69 & \bf 1.70 \\ 
   & 0.50& $\widehat{\theta}_{n}^{(pl)}$ & -0.50 & 3.59 & 3.62 & -0.13 & 2.31 & 2.32 & -0.13 & 1.63 & 1.64 \\ 
   & 0.75 & $\widehat{\theta}_{n}^{(ik)}$ & -1.52 & \bf 2.48 & 2.90 & -0.58 & \bf 1.41 & 1.53 & -0.42 & \bf 0.98 & 1.07 \\ 
   & 0.75 & $\widehat{\theta}_{n}^{(pl)}$ & \bf -2.20 & 2.44 & \bf 3.29 & \bf -0.98 & 1.40 & \bf 1.71 & \bf -0.64 & 0.96 & \bf 1.15 \\ 
   \hline
\multirow{6}{*}{\large 4} & 0.25 & $\widehat{\theta}_{n}^{(ik)}$ & -0.26 & \bf 4.60 & \bf 4.60 & -0.06 & \bf 2.97 & \bf 2.97 & 0.04 & \bf 2.12 & \bf 2.12 \\ 
   & 0.25 & $\widehat{\theta}_{n}^{(pl)}$ & \bf 0.30 & 4.47 & 4.47 & \bf 0.19 & 2.89 & 2.89 & \bf 0.18 & 2.04 & 2.05 \\ 
   & 0.50& $\widehat{\theta}_{n}^{(ik)}$ & \bf -0.63 & \bf 3.79 & \bf 3.84 & -0.13 & \bf 2.36 & \bf 2.37 & -0.11 & \bf 1.69 & \bf 1.69 \\ 
   & 0.50& $\widehat{\theta}_{n}^{(pl)}$ & -0.56 & 3.63 & 3.67 & \bf -0.16 & 2.31 & 2.32 & \bf -0.13 & 1.65 & 1.65 \\ 
   & 0.75 & $\widehat{\theta}_{n}^{(ik)}$ & -0.83 & \bf 2.38 & 2.52 & -0.29 & 1.40 & 1.43 & -0.21 & \bf 0.97 & 0.99 \\ 
   & 0.75 & $\widehat{\theta}_{n}^{(pl)}$ & \bf -1.51 & 2.35 & \bf 2.79 & \bf -0.71 & \bf 1.41 & \bf 1.57 & \bf -0.45 & 0.95 & \bf 1.05 \\ 
   \hline
\hline
\end{tabular}
\caption{Estimation for Gumbel copula with  normal marginals (100 multiples of bias, SD and RMSE)} 
\label{T.gumbel.norm}
\end{footnotesize}
\end{table}

\begin{table}[ht]
\centering
\begin{footnotesize}
\begin{tabular}{ccc|rrr|rrr|rrr}
    && & \multicolumn{3}{c|}{$n=200$}&\multicolumn{3}{c|}{$n=500$}& \multicolumn{3}{c}{$n=1000$}\\
 Model &$\tau$& estim&   bias & SD &RMSE& bias & SD &RMSE& bias & SD &RMSE \\
 \hline
\hline
\parbox[c]{2mm}{\multirow{6}{*}{\vspace{-2mm}\rotatebox[origin=c]{90}{\footnotesize{Known innovations}}}} & 0.25 & $\widehat{\theta}_{n}^{(ik,or)}$ & -0.02 & 3.13 & 3.13 & -0.05 & 2.32 & 2.31 & -0.03 & 1.78 & 1.77 \\ 
   & 0.25 & $\widehat{\theta}_{n}^{(pl,or)}$ & 0.38 & 2.99 & 3.02 & 0.22 & 2.19 & 2.20 & 0.13 & 1.66 & 1.67 \\ 
   & 0.50& $\widehat{\theta}_{n}^{(ik,or)}$ & -0.01 & 2.44 & 2.44 & -0.04 & 1.81 & 1.81 & -0.02 & 1.39 & 1.39 \\ 
   & 0.50& $\widehat{\theta}_{n}^{(pl,or)}$ & 0.32 & 2.26 & 2.28 & 0.19 & 1.67 & 1.68 & 0.12 & 1.27 & 1.27 \\ 
   & 0.75 & $\widehat{\theta}_{n}^{(ik,or)}$ & -0.01 & 1.36 & 1.36 & -0.02 & 1.01 & 1.01 & -0.01 & 0.77 & 0.77 \\ 
   & 0.75 & $\widehat{\theta}_{n}^{(pl,or)}$ & -0.04 & 1.23 & 1.23 & -0.03 & 0.91 & 0.91 & -0.01 & 0.69 & 0.69 \\   
   \hline
\hline
\multirow{6}{*}{\large 1} & 0.25 & $\widehat{\theta}_{n}^{(ik)}$ & -0.29 & \bf 4.65 & \bf 4.66 & -0.07 & \bf 2.83 & \bf 2.83 & \bf -0.15 & \bf 1.99 & \bf 2.00 \\ 
   & 0.25 & $\widehat{\theta}_{n}^{(pl)}$ & \bf 0.35 & 4.49 & 4.50 & \bf 0.19 & 2.72 & 2.72 & 0.02 & 1.89 & 1.89 \\ 
   & 0.50& $\widehat{\theta}_{n}^{(ik)}$ & \bf -0.48 & \bf 3.67 & \bf 3.70 & \bf -0.23 & \bf 2.22 & \bf 2.23 & \bf -0.25 & \bf 1.56 & \bf 1.58 \\ 
   & 0.50& $\widehat{\theta}_{n}^{(pl)}$ & 0.00 & 3.40 & 3.40 & -0.05 & 2.08 & 2.08 & -0.13 & 1.44 & 1.44 \\ 
   & 0.75 & $\widehat{\theta}_{n}^{(ik)}$ & -0.78 & \bf 2.17 & \bf 2.30 & -0.53 & \bf 1.27 & \bf 1.38 & -0.47 & \bf 0.88 & \bf 1.00 \\ 
   & 0.75 & $\widehat{\theta}_{n}^{(pl)}$ & \bf -0.94 & 2.02 & 2.23 & \bf -0.64 & 1.19 & 1.35 & \bf -0.52 & 0.81 & 0.96 \\ 
   \hline
\multirow{6}{*}{\large 2} & 0.25 & $\widehat{\theta}_{n}^{(ik)}$ & -0.34 & \bf 4.39 & \bf 4.40 & -0.12 & \bf 2.80 & \bf 2.80 & \bf -0.10 & \bf 1.94 & \bf 1.94 \\ 
   & 0.25 & $\widehat{\theta}_{n}^{(pl)}$ & \bf 0.38 & 4.21 & 4.22 & \bf 0.22 & 2.72 & 2.72 & 0.10 & 1.83 & 1.83 \\ 
   & 0.50& $\widehat{\theta}_{n}^{(ik)}$ & \bf -0.70 & \bf 3.47 & \bf 3.54 & \bf -0.25 & \bf 2.20 & \bf 2.21 & \bf -0.16 & \bf 1.53 & \bf 1.54 \\ 
   & 0.50& $\widehat{\theta}_{n}^{(pl)}$ & -0.20 & 3.21 & 3.22 & -0.01 & 2.06 & 2.06 & -0.01 & 1.40 & 1.40 \\ 
   & 0.75 & $\widehat{\theta}_{n}^{(ik)}$ & -1.54 & \bf 2.25 & 2.73 & -0.59 & \bf 1.31 & \bf 1.43 & -0.34 & \bf 0.86 & \bf 0.93 \\ 
   & 0.75 & $\widehat{\theta}_{n}^{(pl)}$ & \bf -1.80 & 2.14 & \bf 2.80 & \bf -0.71 & 1.23 & 1.43 & \bf -0.39 & 0.79 & 0.88 \\ 
   \hline
\multirow{6}{*}{\large 3} & 0.25 & $\widehat{\theta}_{n}^{(ik)}$ & \bf -0.38 & \bf 4.41 & \bf 4.42 & -0.15 & \bf 2.80 & \bf 2.81 & \bf -0.13 & \bf 1.95 & \bf 1.96 \\ 
   & 0.25 & $\widehat{\theta}_{n}^{(pl)}$ & 0.33 & 4.23 & 4.24 & \bf 0.18 & 2.72 & 2.73 & 0.06 & 1.83 & 1.83 \\ 
   & 0.50& $\widehat{\theta}_{n}^{(ik)}$ & \bf -0.71 & \bf 3.48 & \bf 3.55 & \bf -0.32 & \bf 2.19 & \bf 2.21 & \bf -0.22 & \bf 1.52 & \bf 1.53 \\ 
   & 0.50& $\widehat{\theta}_{n}^{(pl)}$ & -0.21 & 3.20 & 3.21 & -0.08 & 2.05 & 2.06 & -0.07 & 1.39 & 1.39 \\ 
   & 0.75 & $\widehat{\theta}_{n}^{(ik)}$ & -1.45 & \bf 2.19 & 2.63 & -0.70 & \bf 1.29 & \bf 1.46 & -0.43 & \bf 0.87 & \bf 0.97 \\ 
   & 0.75 & $\widehat{\theta}_{n}^{(pl)}$ & \bf -1.70 & 2.07 & \bf 2.67 & \bf -0.81 & 1.21 & 1.46 & \bf -0.48 & 0.79 & 0.93 \\ 
   \hline
\multirow{6}{*}{\large 4} & 0.25 & $\widehat{\theta}_{n}^{(ik)}$ & \bf -0.34 & \bf 4.40 & \bf 4.41 & -0.15 & \bf 2.81 & \bf 2.81 & \bf -0.11 & \bf 1.96 & \bf 1.97 \\ 
   & 0.25 & $\widehat{\theta}_{n}^{(pl)}$ & 0.30 & 4.24 & 4.25 & \bf 0.16 & 2.72 & 2.72 & 0.07 & 1.84 & 1.84 \\ 
   & 0.50& $\widehat{\theta}_{n}^{(ik)}$ & \bf -0.69 & \bf 3.47 & \bf 3.53 & \bf -0.27 & \bf 2.20 & \bf 2.22 & \bf -0.18 & \bf 1.54 & \bf 1.55 \\ 
   & 0.50& $\widehat{\theta}_{n}^{(pl)}$ & -0.26 & 3.23 & 3.24 & -0.09 & 2.07 & 2.07 & -0.07 & 1.41 & 1.41 \\ 
   & 0.75 & $\widehat{\theta}_{n}^{(ik)}$ & -0.82 & \bf 2.19 & 2.34 & -0.35 & \bf 1.29 & \bf 1.34 & -0.22 & \bf 0.87 & \bf 0.89 \\ 
   & 0.75 & $\widehat{\theta}_{n}^{(pl)}$ & \bf -1.14 & 2.08 & \bf 2.37 & \bf -0.52 & 1.23 & 1.33 & \bf -0.31 & 0.81 & 0.86 \\ 
   \hline
\hline
\end{tabular}
\caption{Estimation for normal copula with  normal marginals (100 multiples of bias, SD and RMSE)} 
\label{T.normal.norm}
\end{footnotesize}
\end{table}

\begin{table}[ht]
\centering
\begin{footnotesize}
\begin{tabular}{ccc|rrr|rrr|rrr}
 && &     \multicolumn{3}{c|}{$n=200$}&\multicolumn{3}{c|}{$n=500$}& \multicolumn{3}{c}{$n=1000$}\\
Model &$\tau$& estim& bias & SD &RMSE& bias & SD &RMSE& bias & SD &RMSE \\
 \hline
\hline
\parbox[c]{2mm}{\multirow{6}{*}{\vspace{-2mm}\rotatebox[origin=c]{90}{\footnotesize{Known innovations}}}} & 0.25 & $\widehat{\theta}_{n}^{(ik,or)}$ & -0.28 & 3.53 & 3.53 & 0.00 & 2.68 & 2.68 & 0.07 & 1.99 & 1.99 \\ 
   & 0.25 & $\widehat{\theta}_{n}^{(pl,or)}$ & 0.03 & 3.48 & 3.48 & 0.23 & 2.61 & 2.62 & 0.21 & 1.97 & 1.98 \\ 
   & 0.50& $\widehat{\theta}_{n}^{(ik,or)}$ & -0.19 & 2.81 & 2.82 & -0.01 & 2.14 & 2.14 & 0.05 & 1.60 & 1.60 \\ 
   & 0.50& $\widehat{\theta}_{n}^{(pl,or)}$ & 0.05 & 2.66 & 2.66 & 0.19 & 2.00 & 2.00 & 0.17 & 1.51 & 1.52 \\ 
   & 0.75 & $\widehat{\theta}_{n}^{(ik,or)}$ & -0.10 & 1.62 & 1.62 & -0.01 & 1.23 & 1.23 & 0.02 & 0.93 & 0.93 \\ 
   & 0.75 & $\widehat{\theta}_{n}^{(pl,or)}$ & -0.16 & 1.46 & 1.47 & -0.02 & 1.09 & 1.09 & 0.02 & 0.83 & 0.83 \\
   \hline
\hline
\multirow{6}{*}{\large 1} & 0.25 & $\widehat{\theta}_{n}^{(ik)}$ & \bf -0.25 & 4.93 & 4.93 & \bf -0.18 & 3.30 & 3.30 & \bf -0.11 & \bf 2.28 & \bf 2.28 \\ 
   & 0.25 & $\widehat{\theta}_{n}^{(pl)}$ & 0.24 & \bf 4.96 & \bf 4.96 & 0.08 & \bf 3.32 & \bf 3.32 & 0.00 & 2.27 & 2.27 \\ 
   & 0.50& $\widehat{\theta}_{n}^{(ik)}$ & \bf -0.48 & \bf 3.95 & \bf 3.97 & \bf -0.34 & \bf 2.62 & \bf 2.64 & \bf -0.24 & \bf 1.81 & \bf 1.83 \\ 
   & 0.50& $\widehat{\theta}_{n}^{(pl)}$ & -0.17 & 3.82 & 3.82 & -0.18 & 2.57 & 2.57 & -0.20 & 1.74 & 1.75 \\ 
   & 0.75 & $\widehat{\theta}_{n}^{(ik)}$ & -0.79 & \bf 2.33 & 2.46 & -0.64 & \bf 1.56 & 1.68 & -0.49 & \bf 1.06 & 1.17 \\ 
   & 0.75 & $\widehat{\theta}_{n}^{(pl)}$ & \bf -1.13 & 2.22 & \bf 2.48 & \bf -0.83 & 1.47 & \bf 1.69 & \bf -0.66 & 0.99 & \bf 1.19 \\ 
   \hline
\multirow{6}{*}{\large 2} & 0.25 & $\widehat{\theta}_{n}^{(ik)}$ & \bf -0.61 & \bf 4.99 & \bf 5.03 & \bf -0.20 & \bf 3.23 & \bf 3.24 & 0.02 & \bf 2.22 & \bf 2.22 \\ 
   & 0.25 & $\widehat{\theta}_{n}^{(pl)}$ & -0.21 & 4.98 & 4.98 & 0.03 & 3.18 & 3.18 & \bf 0.15 & 2.19 & 2.20 \\ 
   & 0.50& $\widehat{\theta}_{n}^{(ik)}$ & \bf -0.89 & \bf 4.01 & \bf 4.11 & \bf -0.35 & \bf 2.62 & \bf 2.64 & \bf -0.08 & \bf 1.79 & \bf 1.79 \\ 
   & 0.50& $\widehat{\theta}_{n}^{(pl)}$ & -0.80 & 3.86 & 3.94 & -0.24 & 2.45 & 2.46 & -0.01 & 1.69 & 1.69 \\ 
   & 0.75 & $\widehat{\theta}_{n}^{(ik)}$ & -1.66 & \bf 2.57 & 3.06 & -0.70 & \bf 1.55 & 1.70 & -0.30 & \bf 1.06 & \bf 1.10 \\ 
   & 0.75 & $\widehat{\theta}_{n}^{(pl)}$ & \bf -2.37 & 2.48 & \bf 3.42 & \bf -0.99 & 1.44 & \bf 1.75 & \bf -0.46 & 0.97 & 1.07 \\ 
   \hline
\multirow{6}{*}{\large 3} & 0.25 & $\widehat{\theta}_{n}^{(ik)}$ & \bf -0.59 & \bf 5.01 & \bf 5.05 & \bf -0.24 & \bf 3.22 & \bf 3.23 & -0.01 & \bf 2.23 & \bf 2.23 \\ 
   & 0.25 & $\widehat{\theta}_{n}^{(pl)}$ & -0.21 & 4.97 & 4.98 & -0.01 & 3.18 & 3.18 & \bf 0.12 & 2.20 & 2.20 \\ 
   & 0.50& $\widehat{\theta}_{n}^{(ik)}$ & \bf -0.90 & \bf 4.07 & \bf 4.16 & \bf -0.43 & \bf 2.59 & \bf 2.63 & \bf -0.14 & \bf 1.79 & \bf 1.79 \\ 
   & 0.50& $\widehat{\theta}_{n}^{(pl)}$ & -0.79 & 3.88 & 3.96 & -0.33 & 2.44 & 2.46 & -0.08 & 1.68 & 1.69 \\ 
   & 0.75 & $\widehat{\theta}_{n}^{(ik)}$ & -1.60 & \bf 2.61 & 3.06 & -0.76 & \bf 1.55 & 1.73 & -0.39 & \bf 1.06 & \bf 1.13 \\ 
   & 0.75 & $\widehat{\theta}_{n}^{(pl)}$ & \bf -2.20 & 2.48 & \bf 3.31 & \bf -1.05 & 1.43 & \bf 1.77 & \bf -0.56 & 0.97 & 1.12 \\ 
   \hline
\multirow{6}{*}{\large 4} & 0.25 & $\widehat{\theta}_{n}^{(ik)}$ & \bf -0.63 & \bf 5.03 & \bf 5.07 & \bf -0.23 & \bf 3.26 & \bf 3.27 & 0.00 & \bf 2.22 & \bf 2.22 \\ 
   & 0.25 & $\widehat{\theta}_{n}^{(pl)}$ & -0.28 & 4.97 & 4.97 & -0.01 & 3.22 & 3.22 & \bf 0.12 & 2.19 & 2.20 \\ 
   & 0.50& $\widehat{\theta}_{n}^{(ik)}$ & \bf -0.91 & \bf 4.06 & \bf 4.16 & \bf -0.38 & \bf 2.61 & \bf 2.64 & \bf -0.09 & \bf 1.78 & \bf 1.78 \\ 
   & 0.50& $\widehat{\theta}_{n}^{(pl)}$ & -0.81 & 3.82 & 3.91 & -0.32 & 2.45 & 2.47 & -0.07 & 1.68 & 1.68 \\ 
   & 0.75 & $\widehat{\theta}_{n}^{(ik)}$ & -0.97 & \bf 2.51 & 2.69 & -0.42 & \bf 1.55 & \bf 1.61 & -0.17 & \bf 1.05 & \bf 1.06 \\ 
   & 0.75 & $\widehat{\theta}_{n}^{(pl)}$ & \bf -1.47 & 2.34 & \bf 2.76 & \bf -0.70 & 1.44 & 1.60 & \bf -0.33 & 0.96 & 1.02 \\ 
   \hline
\hline
\end{tabular}
\caption{Estimation for Student copula with  normal marginals (100 multiples of bias, SD and RMSE)} 
\label{T.t.norm}
\end{footnotesize}
\end{table}


The nonparametric estimates $\widehat{m}_j$ and $\widehat{\sigma}_j$ are constructed as local polynomial estimators of order $p=1$ 
with $K$ being the triweight kernel. The bandwidth $h_n$ is chosen for each estimation separately by the cross-validation method 
from the interval $(D,H)$, where $D=\widehat{\sigma}_Z/n^{1/(3+\varepsilon)}$ 
and $H=\widehat{\sigma}_Z \log^2(n)/n^{1/(4-\varepsilon)}$ for $\varepsilon=0.1$ (cf.\ Remark \ref{remark-bandwidth}) 
and $\widehat{\sigma}_Z$ is an estimate of the standard deviation of the explanatory variable $Z$ (being  $X_i$ or $Y_{i-1}$, 
depending on the model) 
given by $\widehat{\sigma}_Z=\min\{S_Z,\mathrm{IQR}_Z/1.34\}$, where $S_Z$ stands for the sample standard deviation 
and $\mathrm{IQR}_Z$ is the interquartile range.

The weights are given by $w_n(z) = \ind\{z \in [c_{n}^{L}, c_{n}^{U}]\}$, where 
$[c_{n}^{L}, c_{n}^{U}]$ is the largest possible interval such that 
$\inf_{z \in [c_{n}^{L}, c_{n}^{U}]} \widehat{f}_{Z}(z) \geq (\widehat{\sigma}_Z\log^2(n))^{-1}$, where $\widehat{f}_Z$ 
is the kernel density estimator of the marginal  density of~$Z$ \citep[with 
triweight kernel and the bandwidth chosen by the standard normal reference rule, see e.g.][p.~201]{fan_yao_book2005}.


For each setting, we compute the estimate of the copula parameter $\theta$ from the true (but unobserved) 
innovations using the inversion of Kendall's tau method ($\widehat{\theta}_{n}^{(ik,or)}$) and the maximum pseudo-likelihood method ($\widehat{\theta}_{n}^{(pl,or)}$). These oracle estimators are compared 
with their counterparts computed from the residuals ($\widehat{\theta}_{n}^{(ik)}$) and ($\widehat{\theta}_{n}^{(pl)}$). To have more comparable results for different copula 
families the estimates of the parameters are done on the Kendall's tau scale. That is 
we are in fact comparing nonparametric estimates of Kendall's tau with parametric 
estimates, where the parameter is estimated with the help of maximum pseudo-likelihood method. 
The performance of the estimators is measured by the bias, standard deviation (SD), and the root mean square error (RMSE),
which are estimated from the $1\,000$ random samples for chosen sample sizes  $n=200$, $500$ and $1000$.   
Since the obtained quantities are of order $10^{-2}$ and smaller, we report $100$ multiples of bias, SD and RMSE in 
Tables~\ref{T.clayton.norm},\ref{T.frank.norm},\ref{T.gumbel.norm},\ref{T.normal.norm} and \ref{T.t.norm}. 
As 
$\widehat{\theta}_{n}^{(ik)}$ and $\widehat{\theta}_{n}^{(pl)}$ are natural competitors, 
the bigger of the two corresponding performance measures (bias, SD, RMSE) 
is stressed by the \textbf{bold font}.  

In agreement with the results of \cite{genest_et_al_1995} and \cite{tsukahara_2005} the results 
for the (oracle) estimates based on (unobserved) innovations are in favour of MPL method. This 
continues to hold also when working with estimated residuals provided that 
 the dependence  is not very strong (i.e.\ $\tau=0.25$ or $\tau = 0.50$). 
  But if the dependence is strong (i.e.\ $\tau=0.75$) then one should consider using the IK method. 
 This seems to be true in particular for the Clayton copula and to some extent also for the Frank copula 
 and the Gumbel copula. A closer inspection of the results reveals that while the standard deviation of 
  MPL method is almost always slightly smaller than the standard deviation of the IK method, the bias 
  can be substantially larger.  On the other hand the results suggest that 
for the normal and the Student copula one can stay with MPL method even in case of a strong dependence.  

Finally note that for large sample sizes the performance of the estimates based on residuals 
is usually almost as good as of the oracle estimates based on (unobserved) innovations. But there is still 
some  price to pay even for the sample size $n=1000$ and this price relatively increases with the 
level of dependence. The question for possible further research is how to explain the  bad performance of PML method  
 based on residuals for the Clayton copula with a strong dependence.

\section{Application}\label{sec: application}

To illustrate the proposed methods let us consider daily log returns of USD/CZK (US Dollar/Czech Koruna) and GBP/CZK (British Pound/Czech Koruna) exchange rates from 4th January 2010 to 31st December 2012. Note that we take only data till the end of 2012 (total of 758 observations for each series), because in November 2013 the Czech National Bank started its intervention aimed at  CZK/EUR exchange rate.

Daily foreign exchange rates have been successfully modelled using  the nonparametric  autoregression, e.g., in \cite{hardle1998nonparametric} and \cite{yang1999nonparametric}.  
Here, we apply a simple model of two separate nonparametric autoregressions of order 1 and search for a feasible copula for the innovations. 
The conditional means  and  variances  are modelled  using local polynomials with degree $p=1$.  
The weights and the smoothing parameters are chosen 
as  in Section~\ref{sec: simul study}. 
The fitted conditional means and standard deviations are plotted together with the data in Figure~\ref{fig1}. It is visible that the conditional mean functions are rather flat and range around zero.

\begin{figure}[tb]
\centering
\includegraphics[width=0.49\textwidth]{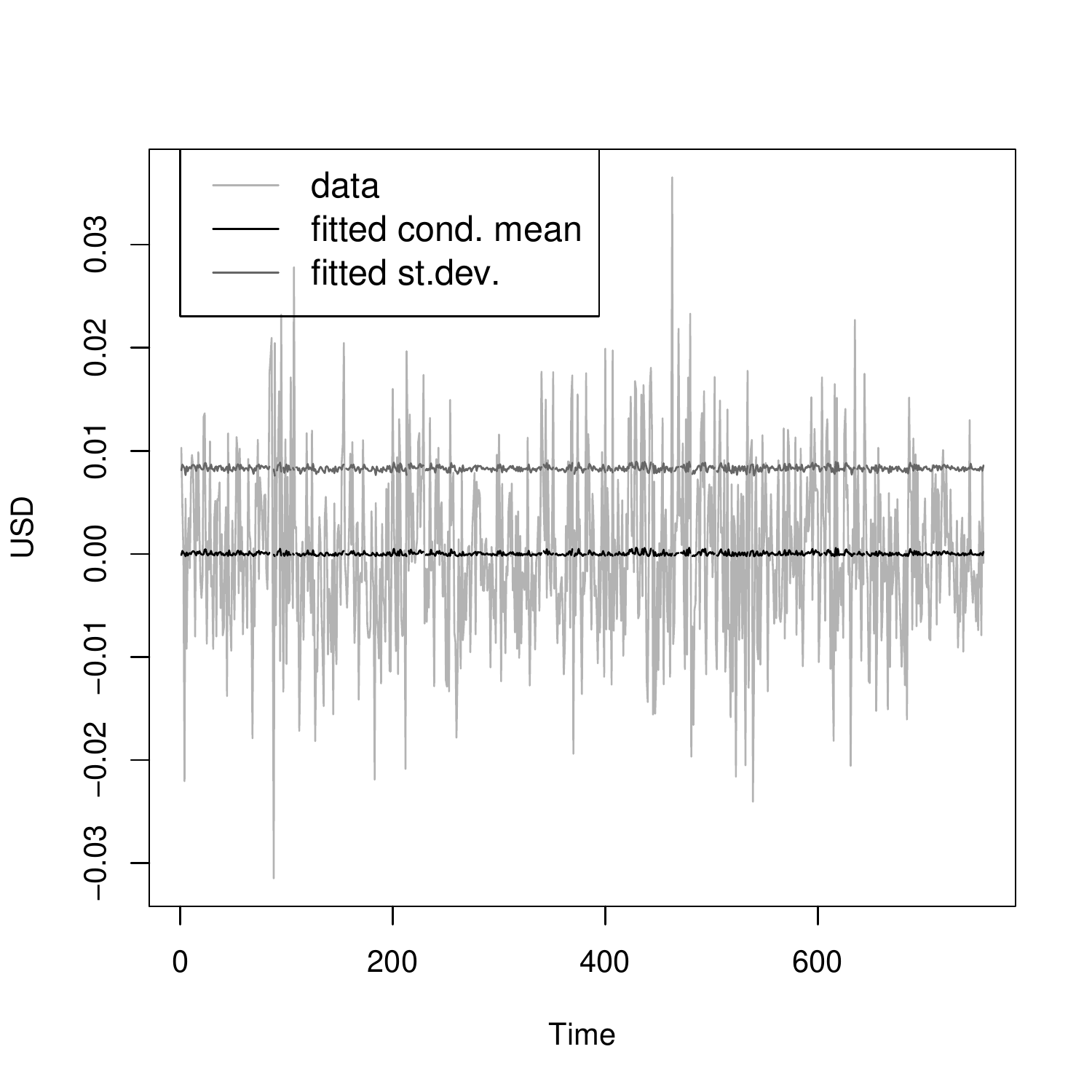}
\includegraphics[width=0.49\textwidth]{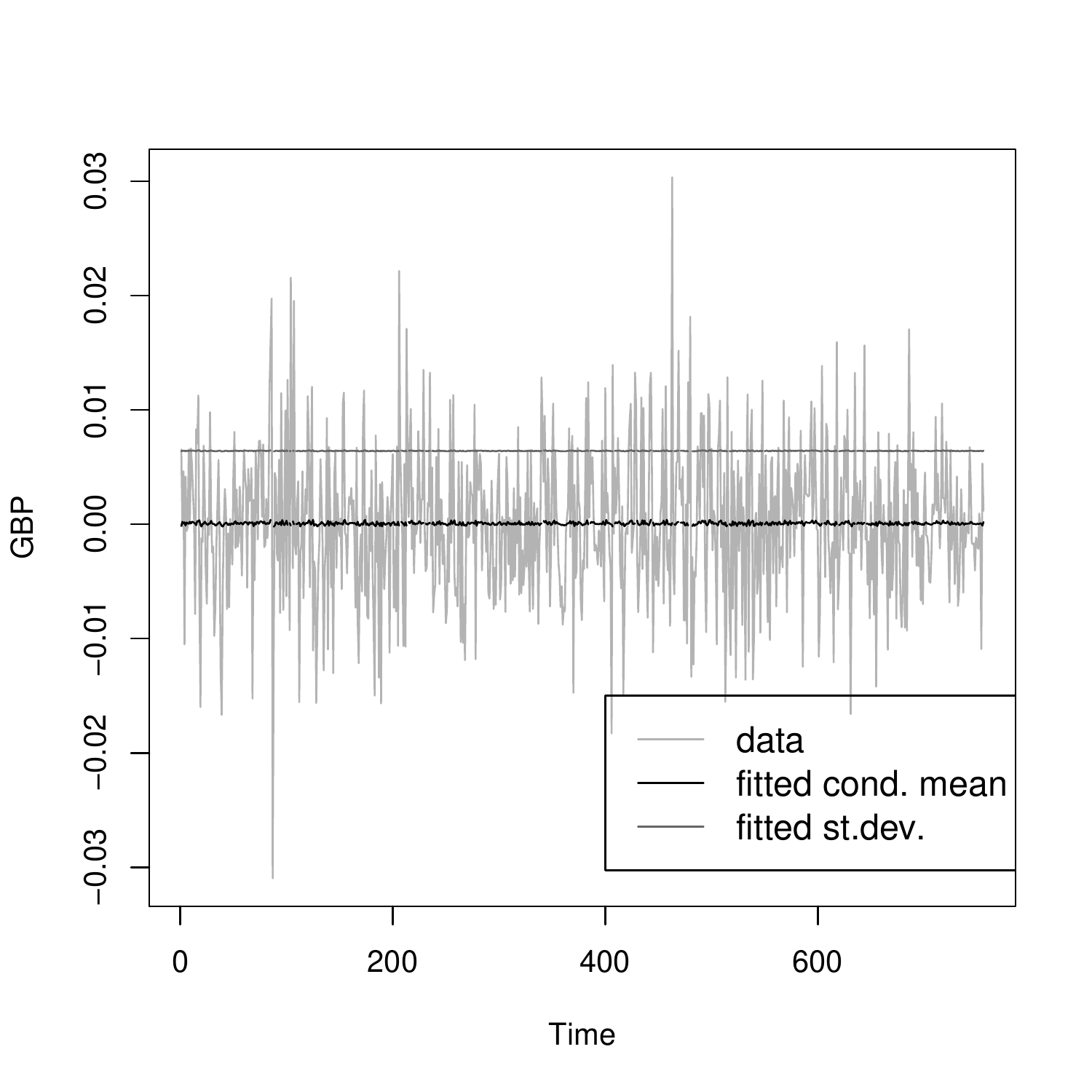}
\vspace{-4mm}
\caption{Fitted conditional mean and variance for the analyzed log returns.}\label{fig1}
\end{figure}

 We use the goodness-of-fit test proposed in Section~\ref{section2.3} in order to decide which copula should be used for modeling the innovations from the two autoregressions. The copula parameter is estimated using the inversion of Kendall's tau method.  The significance  of the test statistics is 
assessed with the help of the bootstrap test based on $B=999$ bootstrap samples. We test Clayton, Frank, Gumbel, normal and Student copula with 4 degrees of freedom respectively and obtain $p$-values $0.000$, $0.000$, $0.001$,  $0.055$, $0.305$.  Hence, we conclude that the Student copula seems to be the best choice for the innovations. The normal copula is also not rejected on the 5\% level, but the corresponding $p$-value is rather borderline, so the Student copula seems to provide a better fit. The maximum pseudo-likelihood method estimates $5.156$ degrees of freedom and parameter $\rho=0.778$.   
Figure~\ref{fig2} shows plot of pseudo-observations 
$(\widetilde{U}_{1i}, \widetilde{U}_{2i})$ given by \eqref{eq: Utilde},  
together with contours of the fitted Student copula. 

\begin{figure}[htb]
\centering
\includegraphics[width=0.5\textwidth]{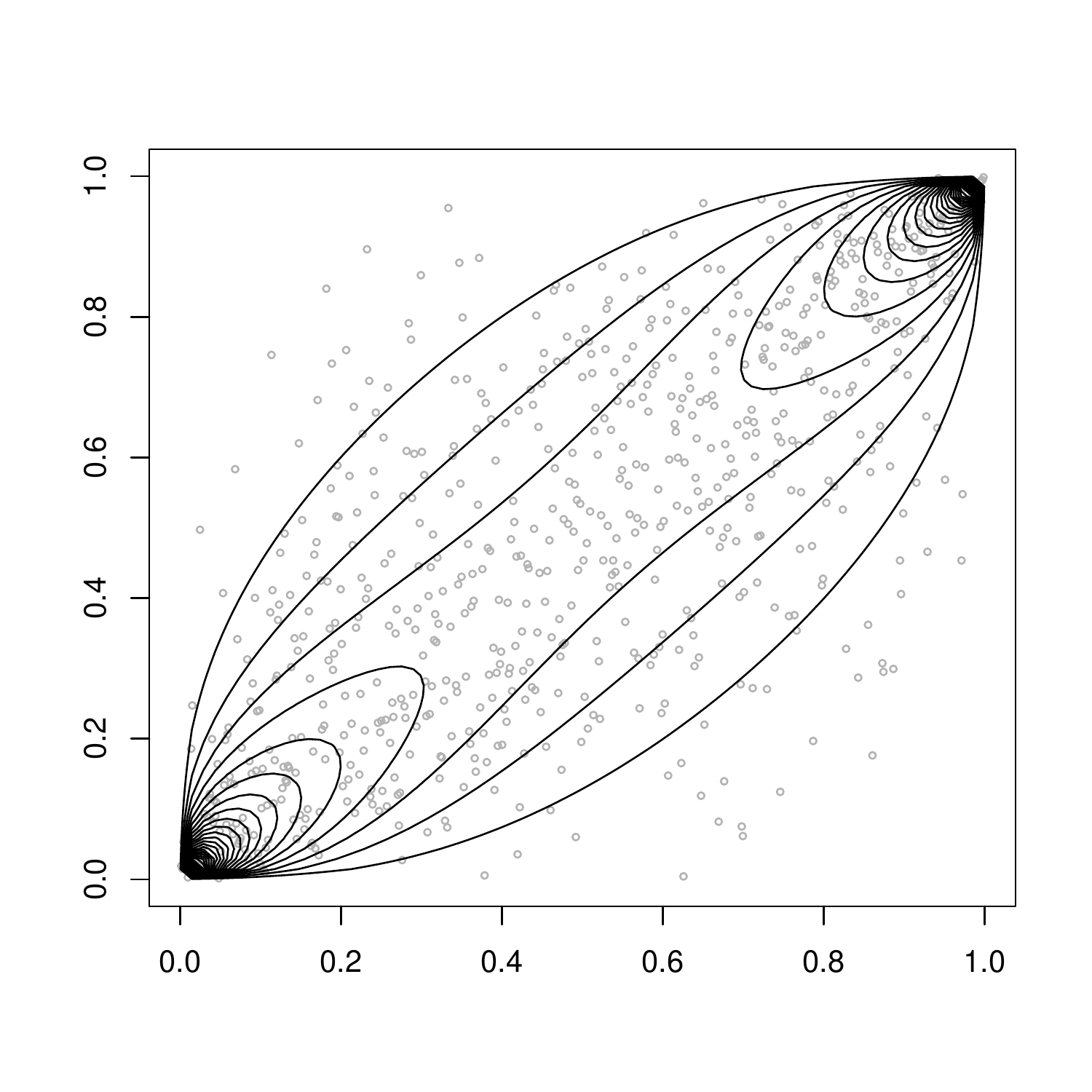} 
\vspace{-4mm} 
\caption{Pseudo-observations 
$(\widetilde{U}_{1i}, \widetilde{U}_{2i})$ given by \eqref{eq: Utilde}    
together with contours of the fitted Student copula (black curves). }\label{fig2}
\end{figure}

\section*{Acknowledgement}
The authors thank the editor, an associate editor and two referees for their very valuable comments that led to a considerable improvement of the contribution. 
The second and the third authors gratefully acknowledge support from the grant GACR 15-04774Y.

\renewcommand{\theequation}{A\arabic{equation}}
\setcounter{equation}{0}
\section*{Appendix A - Proof of Theorem~\ref{thm equiv of Cn}}
%


Recall the definition $W_n=\sum_{j=1}^{n} w_{nj}$. Introduce 
\begin{eqnarray*}
 \hatG_{n}(u_1,u_2) &=&  \frac{1}{W_n} 
  \suman w_{ni} \, \ind \big\{\hateps_{1i} \leq F_{1\eps}^{-1}(u_1),  
 \hateps_{2i} \leq F_{2\eps}^{-1}(u_2) \big\} \\
 &=& \hatF_{\hateps}\big(F_{1\eps}^{-1}(u_1),F_{2\eps}^{-1}(u_2) \big)
\end{eqnarray*}
and note that 
$$
 \widetilde{C}_{n}(u_{1},u_{2}) = \hatG_{n}\big(\hatG_{1n}^{-1}(u_1), 
\hatG_{2n}^{-1}(u_2)
\big),
$$
where $\hatG_{1n}$ and $\hatG_{2n}$ denote the marginals of $\hatG_{n}$. 
Further $\hatG_{n}$ is a distribution function on~$[0,1]^2$ with the 
marginals cdfs satisfying $\hatG_{1n}(0)=\hatG_{2n}(0)=0$. Thus one can 
make use of the Hadamard differentiability of the `copula mapping' 
$\Phi: G \mapsto G(G_{1}^{-1},G_{2}^{-1})$ proved in Theorem~2.4 
of \cite{bucher2013empirical} provided we show that the process 
\begin{equation} \label{eq: Gn}
 \widehat{\mathbb{G}}_{n}(u_1,u_2) = 
  \sqrt{n}\big(\hatG_{n}(u_1,u_2) - C(u_1,u_2)\big),
 \quad (u_1,u_2) \in [0,1]^2,  
\end{equation}
converges in distribution in the space $\ell^{\infty}([0,1]^2)$ 
to a process~$\mathbb{G}$  
with continuous trajectories such that 
$\mathbb{G}(u,0)=\mathbb{G}(0,u)=\mathbb{G}(1,1)=0$ for each $u \in [0,1]$.

%
\subsection*{A1: Decomposition and weak convergence of \texorpdfstring{$\widehat{\mathbb{G}}_{n}$}{Gn}}
%
%
%
Denote 
\begin{align*}
 G_{n}^{(or)}(u_1,u_2) &= 
  \frac{1}{n} \suman  \ind \big\{\eps_{1i} \leq F_{1\eps}^{-1}(u_1),  
 \eps_{2i} \leq F_{2\eps}^{-1}(u_2) \big\},  
\\
 \widetilde{G}_{n}^{(or)}(u_1,u_2) &= 
  \frac{1}{W_n} \suman  w_{ni}\,\ind \big\{\eps_{1i} \leq F_{1\eps}^{-1}(u_1),  
 \eps_{2i} \leq F_{2\eps}^{-1}(u_2) \big\}.  
\end{align*}
%

Now one can decompose the process~$\widehat{\mathbb{G}}_{n}$ as 
%
%
%
\begin{equation} \label{eq: decomposition of Ghat} 
 \widehat{\mathbb{G}}_{n}  = 
\sqrt{n}\, \big(\widehat{G}_{n} - \widetilde{G}_{n}^{(or)} \big) 
+ \sqrt{n}\, \big(\widetilde{G}_{n}^{(or)} - G_{n}^{(or)}\big)  
 + \sqrt{n}\, \big(G_{n}^{(or)} - C \big)  
=  \widetilde{\mathbb{G}}_{n} 
 + \widetilde{\mathbb{G}}_{n}^{(or)} 
 + \mathbb{G}_{n}^{(or)}, 
\end{equation}
where $\widetilde{\mathbb{G}}_{n}$, $\widetilde{\mathbb{G}}_{n}^{(or)}$ 
and $\mathbb{G}_{n}^{(or)}$ stand for the first, second and third 
term respectively on the right-hand side of the first equation 
in~\eqref{eq: decomposition of Ghat}.  
%

In Section~\textbf{A2} it will be shown  that the  first term on the right-hand 
side of~\eqref{eq: decomposition of Ghat} satisfies  
uniformly in $(u_1,u_2) \in [0,1]^2$, 
\begin{align} 
 \notag 
 \widetilde{\mathbb{G}}_{n}(u_1,u_2) 
&=  
 \frac{C^{(1)}(u_1,u_2)\,f_{1\eps}\big(F_{1\eps}^{-1}(u_1)\big)}{\sqrt{n}} \suman \Big[\eps_{1i} 
 + \tfrac{F_{1\eps}^{-1}(u_1)}{2}  (\eps_{1i}^2-1) \Big] 
\\
\label{eq: as repr of Gntilde}
&
 \quad +  \frac{C^{(2)}(u_1,u_2)\,f_{2\eps}\big(F_{2\eps}^{-1}(u_2)\big)}{\sqrt{n}} \suman \Big[\eps_{2i} 
 + \tfrac{F_{2\eps}^{-1}(u_2)}{2}  (\eps_{2i}^2-1) \Big] 
 + o_{P}(1) 
\end{align} 
where (in agreement with the last two conditions in $\mathbf{(F_{\beps})}$)
for $u_{1} \in \{0,1\}$ 
the first term on the right-hand side of~\eqref{eq: as repr of Gntilde} 
is defined as zero and analogously for $u_{2} \in \{0,1\}$. 

In Section~\textbf{A3}, we will show the asymptotic negligibility of the second term on the 
right-hand side of~\eqref{eq: decomposition of Ghat}, i.e.   
\begin{equation} \label{eq: as negl of tildeGnOR minus GnOR}
  \sup_{(u_1,u_2) \in [0,1]^2}
 \Big| \widetilde{\mathbb{G}}_{n}^{(or)}(u_1,u_2) \Big|
 = \sup_{(u_1,u_2) \in [0,1]^2}
 \Big|\sqrt{n}\, \big(\widetilde{G}_{n}^{(or)}(u_1,u_2) - G_{n}^{(or)}(u_1,u_2)\big)\Big|  
  = o_{P}(1). 
\end{equation}
Now combining \eqref{eq: decomposition of Ghat} with 
\eqref{eq: as repr of Gntilde} and \eqref{eq: as negl of tildeGnOR minus GnOR} 
yields that uniformly in $(u_1,u_2)$ 
\begin{equation} \label{eq: as repr for hatGn}
 \widehat{\mathbb{G}}_{n}(u_1,u_2) = A_{n}(u_1,u_2) + B_{n}(u_1,u_2) 
  + o_{P}(1),  
\end{equation}
where 
\begin{align}
 \label{eq: An}
  A_{n}(u_1, u_2) &= \frac{1}{\sqrt{n}} \suman \big[\ind\{U_{1i} \leq u_1, U_{2i} \leq u_2 \} - C(u_1, u_2) \big], 
 \\
\notag 
  B_{n}(u_1,u_2) &=  \frac{1}{\sqrt{n}} \suman \sum_{j=1}^{2} C^{(j)}(u_1,u_2) \, f_{j\eps}\big(F_{j\eps}^{-1}(u_j)\big) \Big[\eps_{ji} 
 + \tfrac{F_{j\eps}^{-1}(u_j)}{2}  (\eps_{ji}^2-1) \Big].   
\end{align}
The asymptotic representation~\eqref{eq: as repr for hatGn} together with standard techniques yields the weak convergence 
of the process $\widehat{\mathbb{G}}_{n}$. 

\bigskip

\noindent Now thanks to Hadamard differentiability 
of the copula functional and Theorem~3.9.4 of \cite{vaart_wellner},
\begin{align}
\notag 
 \sqrt{n}\,&\big[\widetilde{C}_{n}(u_1,u_2) - C(u_1,u_2) \big] 
  = \sqrt{n}\big[\hatG_{n}\big(\hatG_{1n}^{-1}(u_1), 
 \hatG_{2n}^{-1}(u_2)\big) - C(u_1,u_2) \big]
\\
\label{eq: as repr of Ctilde after HD}
 & = \widehat{\mathbb{G}}_{n}(u_1,u_2) 
  - C^{(1)}(u_1,u_2) \widehat{\mathbb{G}}_{n}(u_1,1) 
  - C^{(2)}(u_1,u_2) \widehat{\mathbb{G}}_{n}(1,u_2) + o_{P}(1). 
\end{align}
Note that for each $(u_1,u_2) \in [0,1]^2$ 
\begin{equation} \label{eq: Bn diminishes}
 B_{n}(u_1,u_2) - C^{(1)}(u_1,u_2) B_{n}(u_1,1) - C^{(2)}(u_1,u_2) B_{n}(1,u_2) = 0.   
\end{equation} 
Further combining \eqref{eq: as repr of Ctilde after HD} with 
\eqref{eq: as repr for hatGn}, \eqref{eq: An} and \eqref{eq: Bn diminishes} gives 
\begin{align}
\notag 
 \sqrt{n}\,&\big[\widetilde{C}_{n}(u_1,u_2) - C(u_1,u_2) \big] 
\\
\label{as: final as repr for hatGn}
 & = A_{n}(u_1,u_2) 
  - C^{(1)}(u_1,u_2) A_{n}(u_1,1) 
  - C^{(2)}(u_1,u_2) A_{n}(1,u_2) + o_{P}(1). 
\end{align}
Now the right-hand side of~\eqref{as: final as repr for hatGn} coincides with 
the asymptotic representation of the `oracle' copula process  
$ \sqrt{n} \big[C_n^{(or)} - C\big]$, which implies the statement of Theorem~\ref{thm equiv of Cn}.


\subsection*{A2: Showing \texorpdfstring{\eqref{eq: as repr of Gntilde}}{(A3)}}
Let us introduce the process 
$$
 Z_{n}(f) = \frac{1}{\sqrt{n}} \suman f(\XX_{i}, \eps_{1i}, \eps_{2i}), 
$$ 
that is indexed be the following set of functions
\begin{multline*} 
 \mathcal{F} = \Big\{ (\xx,y_1,y_2) \mapsto
  \ind\big\{\xx \in [-c,c]^d \big\}\,\ind\big\{ y_1 \leq z_{1}\,b_{1}(\xx)+a_{1}(\xx),
   y_2 \leq z_{2}\,b_{2}(\xx)+a_{2}(\xx) \big\},
\\ 
  c \in \RR^+, z_{1},z_{2} \in \RR,\, 
a_1, a_2\in \mathcal{G}, b_1,b_2 \in\widetilde{\mathcal{G}}
 \Big\},
\end{multline*}
where
\begin{eqnarray}
\label{G}
\mathcal{G}&=&\big\{ a:\RR^d\to\RR\mid a\in C_1^{d+\delta}\big(\RR^d\big), 
 \, \sup_{\xx} \|\xx\|^{\nu}\, |a(\xx)|\leq 1 \big\},
\\
\label{tilde-G}
\widetilde{\mathcal{G}}&=&\big\{ b:\RR^d\to\RR\mid 
  b\in \widetilde{C}_2^{d+\delta}\big(\RR^d\big), 
\sup_{\xx} \|\xx\|^{\nu}\,|b(\xx)-1|\leq 1 \big\} 
\end{eqnarray}
and for $\delta$ from assumption $\mathbf{(Bw)}$ and some $\nu$ large enough such that 
\begin{equation}\label{nu}
 \frac{b}{b-1}\Big(\frac{d}{\nu}+\frac{d}{d+\delta}\Big)<1. 
\end{equation}

Denote the centred process as  
\begin{equation} \label{eq: bar Zn}
 \bar{Z}_{n}(f) = Z_{n}(f) - \Es Z_{n}(f), \qquad f \in \mathcal{F}, 
\end{equation} 
and note that $f$ may be formally identified by $(c, z_1,z_2,a_1,b_1,a_2,b_2)$. 
We will use the notation $f\hat=(c, z_1,z_2,a_1,b_1,a_2,b_2)$. 
Further in agreement with the notation used in \cite{vaart_wellner_2007} 
by $\bar{Z}_{n}(f_{n})$ for random~$f_{n}$ we understand 
the value of the mapping $f \mapsto \bar{Z}_{n}(f)$ evaluated at $f_{n}$.

Consider the semi-norm given by  
\begin{equation*} 
 \|f\|_{2,\beta}^{2} = \int_{0}^{1} \beta^{-1}(u)\,Q_{f}^{2}(u)\,du, 
\end{equation*}
where 
$$
 \beta^{-1}(u) = \inf\big\{x>0 :\beta_{\lfloor x\rfloor}\leq u\}, \quad  
 Q_{f}(u) = \inf\big\{x>0: \PP\big(\big|f(\eps_{11},\eps_{21},\XX_{1})\big| > x \big) \leq u \big\}. 
$$
%
From assumption $\mathbf{(\boldsymbol{\beta})}$ one obtains that $\beta^{-1}(u)\leq cu^{-1/b}$ for some constant $c$. 
Further denote 
\[
 P|f-g| = \E \big|f(\XX_{1}, \eps_{11}, \eps_{12}) - g(\XX_{1}, \eps_{11}, \eps_{12})\big|
 = \PP\Big(\big|f(\XX_{1}, \eps_{11}, \eps_{12}) - g(\XX_{1}, \eps_{11}, \eps_{12})\big| > 0\Big).  
\]
As $\mathcal{F}$ consists of indicator functions  
for $f,g\in\mathcal{F}$ one has $Q_{f-g}(u)=\ind\{0<u<P|f-g|\}$.  Thus one obtains for $\epsilon<1$
\begin{equation} \label{eq: norm 2beta}
 \|f-g\|_{2,\beta}^{2} \leq c \int_{0}^{P|f-g|} u^{-1/b}\,du 
  = \frac{cb}{b-1}\big(P|f-g|\big)^{1-1/b}.  
\end{equation}
Starting with brackets of $\|\cdot\|_{2}$-length $\epsilon^{2b/(b-1)}$ of the function classes $\mathcal{G}$, $\widetilde{\mathcal{G}}$ and $\{ \xx \mapsto
  \ind\{\xx \in [-c,c]^d \}\mid c \in \RR^+\}$ it is then easy to construct brackets for $\mathcal{F}$ with $\|\cdot\|_{2,\beta}$-length $\epsilon$ 
\citep[compare with the proof of Lemma~1 in][]{dette2009goodness}. 
Thus one obtains 
\begin{align} 
\notag 
\log\Big(N_{[\,]}(\epsilon,\mathcal{F}, \|\cdot\|_{2,\beta})\Big) &\leq
  \log\Big(O(\epsilon^{-2db/(b-1)})N_{[\,]}\left(\epsilon^{2b/(b-1)},\mathcal{G},\|\cdot\|_2\right)N_{[\,]}\left(\epsilon^{2b/(b-1)},\widetilde{\mathcal{G}},\|\cdot\|_2\right)\Big)
\\
&\leq O\big(\log(\epsilon) \big) + O\Big(\epsilon^{-2\frac{b}{b-1}\big(\frac{d}{\nu}+
\frac{d}{d+\delta}\big)}\Big),
\label{eq: bound on brack number of F}
\end{align}
 where the rate follows from Lemma \ref{lem-covering-functionclass} in Appendix C.
Further one bracket is sufficient for $\epsilon\geq 1$. Thus by \eqref{eq: bound on brack number of F} 
and (\ref{nu}),
$$
 \int_{0}^{\infty} \sqrt{\log N_{[\,]}(\epsilon,\mathcal{F}, \|\cdot\|_{2,\beta})}\,d\epsilon<\infty. 
$$
%
From \cite{dedecker_louhichi}, Section 4.3, it follows that the centred process $\bar Z_{n}$ 
given by~\eqref{eq: bar Zn} is 
asymptotically $\|.\|_{2,\beta}$-equicontinuous.

To apply this result in order to prove \eqref{eq: as repr of Gntilde} note that 
\begin{eqnarray*}
 \widetilde {\mathbb{G}}_{n}(u_1,u_2) &=&
 \frac{\sqrt{n}}{W_{n}} \sum_{i=1}^n  w_{ni}\,\Big(\ind \Big\{\eps_{1i} \leq F_{1\eps}^{-1}(u_1)\tfrac{\widehat{\sigma}_1(\Xb_i)}{\sigma_1(\Xb_i)} + \tfrac{\widehat{m}_1(\Xb_i)-m_1(\Xb_i)}{\sigma_1(\Xb_i)},  \\
 &&\qquad\qquad\qquad \quad \eps_{2i} \leq F_{2\eps}^{-1}(u_2)\tfrac{\widehat{\sigma}_2(\Xb_i)}{\sigma_2(\Xb_i)} + \tfrac{\widehat{m}_2(\Xb_i)-m_2(\Xb_i)}{\sigma_2(\Xb_i)} \Big\}
\\
&&\qquad\qquad\qquad -\ind \big\{\eps_{1i} \leq F_{1\eps}^{-1}(u_1),\eps_{2i} \leq F_{2\eps}^{-1}(u_2)\big\}\Big) 
 \end{eqnarray*}
 and introduce the process
\begin{align*}
 \check{\mathbb{G}}_{n}(u_1,u_2) &= 
   \frac{1}{\sqrt{n}} \sum_{i=1}^n  w_{ni}\,\Big(\ind \Big\{\eps_{1i} \leq F_{1\eps}^{-1}(u_1)\hatb_1(\Xb_i)+\hata_1(\Xb_i), 
\eps_{2i} \leq F_{2\eps}^{-1}(u_2)\hatb_2(\Xb_i)+\hata_2(\Xb_i) \Big\}\\
&\qquad\qquad\qquad\quad  -\ind \Big\{\eps_{1i} \leq F_{1\eps}^{-1}(u_1),\eps_{2i} \leq F_{2\eps}^{-1}(u_2)\Big\}\Big) 
\end{align*}
 with $\hata_j,\hatb_j$, $j=1,2$, from Lemma~\ref{lemma hatm and hatsigma} 
and Remark~\ref{rem: extending a and b} in Appendix C. 
 Then one obtains by monotonicity arguments applying Lemma  \ref{lemma hatm and hatsigma}(i) that, on an event with probability converging to one
\begin{eqnarray*}
Z_n\big(f_n^\ell\big)-Z_n\big(g_n) \leq 
\big(\tfrac{W_n}{n}\,\widetilde{\mathbb{G}}_{n}(u_1,u_2)- \check{\mathbb{G}}_{n}(u_1,u_2)\big)\leq  Z_n\big(f_n^u\big)-Z_n\big(g_n)
 \end{eqnarray*}
 for some deterministic positive sequence $\gamma_n=o(n^{-1/2})$.
 Here,
\begin{align*}
 f_{n}^\ell &\hat= \big(c_n, F_{1\eps}^{-1}(u_1), F_{2\eps}^{-1}(u_2),\hata_1-\gamma_n, \hatb_1-\gamma_n\sign( F_{1\eps}^{-1}(u_1)), \hata_2-\gamma_n,\hatb_2-\gamma_n\sign( F_{2\eps}^{-1}(u_2))\big),
\\
 f_{n}^u &\hat= \big(c_n, F_{1\eps}^{-1}(u_1), F_{2\eps}^{-1}(u_2),\hata_1+\gamma_n, 
 \hatb_1+\gamma_n\sign( F_{1\eps}^{-1}(u_1)), \hata_2+\gamma_n,\hatb_2+\gamma_n\sign( F_{2\eps}^{-1}(u_2))\big),
\\
g_{n} &\hat= \big(c_n, F_{1\eps}^{-1}(u_1), F_{2\eps}^{-1}(u_2),\hata_1,\hatb_1,\hata_2,\hatb_2\big).  
\end{align*}
 We only consider the upper bound, the lower one can be handled completely analogously.  
 First note that $Z_n\big(f_n^u\big)-Z_n\big(g_n)=\bar Z_n\big(f_n^u\big)-\bar Z_n\big(g_n) + R_n$, where with probability converging to one,
 \begin{equation}\label{expectation Z_n}
 \qquad |R_n|\leq 2\,\sqrt{n}\,\max_{j=1,2}\sup_{u\in\mathbb{R}, s\in \{-1,1\}\atop v\in(1/2,1), w \in (-1,1)}\big|F_{j\eps}\big(uv+w\big)-F_{j\eps}\big(u(v+s\gamma_n)+w+\gamma_n \big)\big|
  =o(1),
\end{equation} 
where the last equality follows by a Taylor expansion, assumption $\mathbf{(F_{\beps})}$ and $\gamma_n=o(n^{-1/2})$.
Now introducing the notation ($j=1,2$)
 \begin{equation}\label{eq: Fjhat_inv xx} 
  F_{j\eps}^{-1}(u,\xx,\gamma) = 
 F_{j\eps}^{-1}(u)\big[\hatb_j(\xx)+\gamma\sign\big(F_{j\eps}^{-1}(u)\big)\big] + \hata_j(\xx)+\gamma
\end{equation}
one can show as in (\ref{eq: norm 2beta}) that for a sufficiently large~$M$  
%
\begin{align}
\notag 
  \|f_n^u - g_{n}\|_{2,\beta} &\leq 
  M\, 
  \PP\Big(\Big| 
   \ind\Big\{\XX_{1} \in \JJ_{n}, \,\eps_{11} \leq   F_{1\eps}^{-1}(u_1,\XX_{1},\gamma_n), 
   \eps_{21} \leq  F_{2\eps}^{-1}(u_2,\XX_{1},\gamma_n) 
    \Big\} 
\\*
\notag 
& \qquad \quad \ 
   - \ind\Big\{\XX_{1} \in \JJ_{n}, \,\eps_{11} \leq   F_{1\eps}^{-1}(u_1,\XX_1,0), 
   \eps_{21} \leq   F_{2\eps}^{-1}(u_2,\XX_1,0)
    \Big\}
  \Big| > 0
 \Big)^{1-1/b}  
\\
\notag 
& \leq M\, 
  \PP\Big(
\XX_{1} \in \JJ_n, 
    F_{1\eps}^{-1}(u_1,\XX_1,0) \leq \eps_{11} \leq   F_{1\eps}^{-1}(u_1,\XX_{1},\gamma_n)\Big)^{1-1/b}  
\\
\notag
    & \qquad + M\, 
  \PP\Big(
\XX_{1} \in \JJ_n, 
    F_{2\eps}^{-1}(u_2,\XX_1,0) \leq \eps_{21} \leq  F_{2\eps}^{-1}(u_2,\XX_{1},\gamma_n)\Big)^{1-1/b}  
%
\end{align}
and this can be bounded by $Mn^{-(1-1/b)/2}$ times the bound on the right hand side of (\ref{expectation Z_n}) and thus converges to zero in probability
uniformly in~$u_1,u_2$. 
Therefore there exists a deterministic sequence $\delta_n\searrow 0$ with 
$\PP\big(\sup_{u_1,u_2} \|f_n^u - g_{n}\|_{2,\beta}\leq\delta_n\big)\to 1$ as $n\to\infty$. 
Further by Lemma~\ref{lemma hatm and hatsigma} 
and Remark~\ref{rem: extending a and b} of Appendix~C  one has 
$\PP\big(f_{n}^{u}$, $f_{n}^{l}$, $g_{n} \in \mathcal{F}\big) \to 1$ as $n\to\infty$.  
Now from $\|.\|_{2,\beta}$-equicontinuity of $\bar Z_n$ one obtains for every $\epsilon>0$ that
\begin{eqnarray*}
 \PP\bigg(\sup_{u_1,u_2}\big|\bar Z_n(f_n^u)-\bar Z_n(g_n) \big|>\epsilon\bigg)
 &\leq& \PP\bigg(\sup_{f,g\in\mathcal{F}\atop  \|f-g\|_{2,\beta}\leq\delta_n }\big|\bar Z_n(f)-\bar Z_n(g) \big|>\epsilon\bigg) +o(1) \;=\; o(1)
\end{eqnarray*}
and thus 
$  \big|\bar Z_n(f_n^u)-\bar Z_n(g_n) \big| = o_P(1)$
 uniformly with respect to $u_1,u_2$. In combination with (\ref{expectation Z_n}), 
analogous considerations for the lower bound $Z_n\big(f_n^\ell\big)-Z_n\big(g_n)$ and 
the fact that 
\begin{equation} \label{eq: Wn divided n}
  \frac{W_{n}}{n} = 1 + o_{P}(1), 
\end{equation}
we obtain
\begin{equation*} 
 \sup_{u_1,u_2}\,\big|\tfrac{W_n}{n}\,\widetilde{\mathbb{G}}_{n}(u_1,u_2)- \check{\mathbb{G}}_{n}(u_1,u_2)\big| 
  = o_{P}(1).  
\end{equation*}
Further thanks to  \eqref{eq: Wn divided n} it is sufficient to show that 
the process $\check{\mathbb{G}}_{n}(u_1,u_2)$ has the asymptotic 
representation given by the right-hand side of~\eqref{eq: as repr of Gntilde}. 

Thus the remaining proof of~\eqref{eq: as repr of Gntilde} is divided into two parts. 
First we prove that 
\begin{equation} \label{eq: as neglib of centred tilde Gn}
 \sup_{u_1,u_2}\,\big|\check {\mathbb{G}}_n(u_1,u_2) - \Es\!^\ast[ \check {\mathbb{G}}_n(u_1,u_2)]\big| 
  = o_{P}(1)   
\end{equation}
and then we calculate $\Es\!^\ast[ \check {\mathbb{G}}_n(u_1,u_2)]$. Here, with slight abuse of  notation, $\Es\!^\ast$ denotes expectation, considering the functions $\hata_j$, $\hatb_j$ as deterministic.

\subsubsection*{Showing~\eqref{eq: as neglib of centred tilde Gn}}
Note that we have
\begin{equation} \label{eq: bar An through bar Zn}
 \check {\mathbb{G}}_n(u_1,u_2) - \E^{\ast}[ \check {\mathbb{G}}_n(u_1,u_2)] = \bar Z_{n}(f_{n}) - \bar Z_{n}(g_{n}), 
\end{equation}
where
\begin{equation*}
 f_{n} \hat{=} \big(c_n, F_{1\eps}^{-1}(u_1), F_{2\eps}^{-1}(u_2), \hata_1,\hatb_1,\hata_2,\hatb_2\big), \quad  
g_{n} \hat{=} \big(c_{n}, F_{1\eps}^{-1}(u_1),F_{2\eps}^{-1}(u_2), 0, 
   1, 0, 1\big),   
\end{equation*}
%
with $0$ and $1$ standing for functions that are constantly equal to 
zero and one respectively. 
Similarly to before
one can show that for a sufficiently large~$M$ 
%
\begin{align*}
  \|f_n - g_{n}\|_{2,\beta}  
  & \leq M \Big( \E \Big[\Big|F_{1\eps}\big(F_{1\eps}^{-1}(u_1,\XX_{1},0)\big) 
             - u_1\Big|\, \ind\big\{\XX_{1} \in \JJ_n \big\}\Big]
\\
  & \qquad + \E \Big[\Big|F_{2\eps}\big(F_{2\eps}^{-1}(u_2,\XX_{1},0)\big)   
            - u_2\Big|\,\ind\big\{\XX_{1} \in \JJ_n \big\}\Big]\Big)^{1-1/b}
\end{align*}
using notation (\ref{eq: Fjhat_inv xx}).
%

%
%
%
Now note that with Lemma~\ref{lemma hatm and hatsigma}\,(iii) in Appendix C
 we obtain $  \|f_n - g_{n}\|_{2,\beta} = o_{P}(1)$ 
uniformly in~$u_1,u_2$. 
Finally with the help of \eqref{eq: Wn divided n}, \eqref{eq: bar An through bar Zn} and 
the asymptotic $\|.\|_{2,\beta}$-equicontinuity of the process $\bar Z_{n}$ one can 
conclude (\ref{eq: as neglib of centred tilde Gn}).   

\subsubsection*{Calculating $\Es\!^\ast[ \check {\mathbb{G}}_n(u_1,u_2)]$}  
To simplify the notation and to prevent the confusion let the random vector $\Xb$ have the same distribution 
as $\Xb_{1}$. With the help of a second-order Taylor series expansion of the right-hand side 
one gets 
\begin{align}
\notag 
 \Es\!^\ast[ \check {\mathbb{G}}_n(u_1,u_2)] &=
\sqrt{n}\,\Es\!^\ast \big[ w_{n}(\XX)\big\{F_{\eps}\big(F_{1\eps}^{-1}(u_{1},\XX,0), 
  F_{2\eps}^{-1}(u_{2},\XX,0)\big) 
 - F_{\eps}\big(F_{1\eps}^{-1}(u_{1}), 
  F_{2\eps}^{-1}(u_{2})\big)\big\}\big] 
%
 \\
\label{eq: Taylor for En nonparam}
&=
\sqrt{n}\,\sum_{j=1}^{2} \Es\!^\ast\big\{w_{n}(\XX) 
 \,F_{\eps}^{(j)}\big(F_{1\eps}^{-1}(u_{1}), 
  F_{2\eps}^{-1}(u_{2})\big)  
 Y_{j\XX}(u_{j}) \big\}
%
\\
\notag   
& \quad + \tfrac{\sqrt{n}}{2}\,\sum_{j=1}^{2}\sum_{k=1}^{2} \Es\!^\ast\big\{w_{n}(\XX) 
 \,F_{\eps}^{(j,k)}\big(F_{1\eps}^{-1}(u_{1\XX}), 
  F_{2\eps}^{-1}(u_{2\XX})\big)  
 Y_{j\XX}(u_{j})\,Y_{k\XX}(u_{k}) \big\},  
%
\end{align}
%
%
where 
\begin{equation*}
 Y_{j\xx}(u) = F_{j\eps}^{-1}(u,\xx,0) - F_{j\eps}^{-1}(u)
  =  \hata_j(\xx) 
 + F_{j\eps}^{-1}(u)(\hatb_j(\xx)-1),  
 \quad j=1,2,
\end{equation*}
and the point $u_{j\xx}$ lies between the points
$F_{j\eps}\big(F_{j\eps}^{-1}(u_{j},\xx,0)\big)$ and $u_j$.
Now using Lemma~\ref{lemma hatm and hatsigma}\,(iv) in Appendix C for $j=1,2$ 
\begin{align*}
& \sqrt{n}\,  \Es\!^\ast\big\{w_{n}(\XX) 
 \,F_{\eps}^{(j)}\big(F_{1\eps}^{-1}(u_{1}), 
  F_{2\eps}^{-1}(u_{2})\big)  
 Y_{j\XX}(u_{j}) \big\}
 \\
%
& = \sqrt{n}\,F_{\eps}^{(j)}\big(F_{1\eps}^{-1}(u_{1}), 
  F_{2\eps}^{-1}(u_{2})\big)\!\Big[
  \Es\!^\ast\big[ \hata_j(\XX)\ind\{\XX\in\JJ_{n}\}\big]
  \!+\! F_{j\eps}^{-1}(u_{j}) 
   \Es\!^\ast\big[(\hatb_j(\XX)-1)\ind\{\XX\in\JJ_{n}\}
   \big]
 \Big]
\\
& = F_{\eps}^{(j)}\big(F_{1\eps}^{-1}(u_{1}), 
  F_{2\eps}^{-1}(u_{2})\big)\,\frac{1}{\sqrt{n}}\suman \Big[ \eps_{ji} 
  + \tfrac{F_{j\eps}^{-1}(u_{j})}{2}\,(\eps_{ji}^2-1)    \Big] + o_{P}(1)
\\
& = \frac{C^{(j)}(u_{1},u_{2})\,f_{j\eps}\big(F_{j\eps}^{-1}(u_{j})\big)}
   {\sqrt{n}}\suman \Big[ \eps_{ji} 
  + \tfrac{F_{j\eps}^{-1}(u_{j})}{2}\,(\eps_{ji}^2-1)    \Big] + o_{P}(1)
\end{align*}
uniformly in $(u_1,u_2)$. 
\smallskip

To conclude the proof of~\eqref{eq: as repr of Gntilde} we need to show that `the second order terms' 
in~\eqref{eq: Taylor for En nonparam} are asymptotically negligible. 
To show that note that by assumption~$\mathbf{(F_{\beps})}$ 
and Lemma~\ref{lemma hatm and hatsigma}\,(iii) there 
exists a finite constant $M$ such that with probability going to one 
\allowdisplaybreaks{
\begin{align*}
\notag 
\Big|F_{\eps}^{(j,k)}&\big(F_{1\eps}^{-1}(u_{1\xx}), 
  F_{2\eps}^{-1}(u_{2\xx})\big)\Big| \,
\big(1 + \big|F_{j\eps}^{-1}(u_j)\big| \big)  
\big(1 + \big|F_{k\eps}^{-1}(u_k)\big| \big)  
\nobreak 
 \\
\notag 
&= \Big|F_{\eps}^{(j,k)}\big(F_{1\eps}^{-1}(u_{1\xx}), 
  F_{2\eps}^{-1}(u_{2\xx})\big)\Big| 
\big(1 + \big|F_{j\eps}^{-1}(u_{j\xx})\big| \big)  
\big(1 + \big|F_{k\eps}^{-1}(u_{k\xx})\big| \big) 
\nobreak  
\\
\nobreak 
\notag 
& \qquad \qquad \frac{\big(1 + \big|F_{j\eps}^{-1}(u_j)\big| \big)  
\big(1 + \big|F_{k\eps}^{-1}(u_k)\big| \big)
}{\big(1 + \big|F_{j\eps}^{-1}(u_{j\xx})\big| \big)  
\big(1 + \big|F_{k\eps}^{-1}(u_{k\xx})\big| \big) } 
 \\
\notag 
& \leq M\, 
\frac{\big(1 + \big|F_{j\eps}^{-1}(u_{j})\big| \big)  
\big(1 + \big|F_{k\eps}^{-1}(u_{k})\big| \big) 
}{\big(1 + \big|F_{j\eps}^{-1}(u_{j\xx})\big| \big)  
\big(1 + \big|F_{k\eps}^{-1}(u_{k\xx})\big| \big) }  
\\
\notag 
& \leq 
\frac{M\,\big(1 + \big|F_{j\eps}^{-1}(u_{j})\big| \big)  
\big(1 + \big|F_{k\eps}^{-1}(u_{k})\big| \big) 
}{\big(1 + \big|F_{j\eps}^{-1}(u_{j})\big(1+o_{P}(n^{-\frac{1}{4}})\big) + 
o_{P}(n^{-\frac{1}{4}})\big| \big)  
\big(1 + \big|F_{k\eps}^{-1}(u_{k})\big(1+o_{P}(n^{-\frac{1}{4}})\big) + o_{P}(n^{-\frac{1}{4}})\big| \big) } 
\\
\nobreak 
& \leq 2\,M
\end{align*}
}
uniformly in $(u_1,u_2) \in [0,1]^2$ and $\xx \in \JJ_{n}$. 

Thus to prove 
$$
 \Es\!^\ast \big\{w_{n}(\XX) 
 \,F_{\eps}^{(j,k)}\big(F_{1\eps}^{-1}(u_{1\XX}), 
  F_{2\eps}^{-1}(u_{2\XX})\big)  
 Y_{j\XX}(u_{j})\,Y_{k\XX}(u_{k}) \big\} = o_{P}(n^{-1/2}),  
$$
it is sufficient to use once more Lemma~\ref{lemma hatm and hatsigma}\,(iii).

\subsection*{A3: Showing~\texorpdfstring{\eqref{eq: as negl of tildeGnOR minus GnOR}}{(A4)}}
Recall that $W_{n} = \suman w_{ni}$ and decompose 
\allowdisplaybreaks{
\begin{align}
\notag 
 \sqrt{n}& \big(\widetilde{G}_{n}^{(or)}(u_1,u_2) 
 - G_{n}^{(or)}(u_1,u_2)\big)  
\\*
\notag 
 & = \sqrt{n}\big(\widetilde{G}_{n}^{(or)}(u_1,u_2) - C(u_1,u_2)\big)
  - \sqrt{n}\big(G_{n}^{(or)}(u_1,u_2) - C(u_1,u_2)\big) 
\\
\notag 
 & = \frac{n}{\sqrt{n}}\,\suman  \Big(\frac{w_{ni}}{W_{n}} - \frac{w_{ni}}{n} +  \frac{w_{ni}}{n} - \frac{1}{n}\Big)\,\big[\ind \big\{\eps_{1i} \leq F_{1\eps}^{-1}(u_1),  
 \eps_{2i} \leq F_{2\eps}^{-1}(u_2) \big\} - C(u_1,u_2) \big]
\\
\label{eq: tildeGnOR minus tildeGnOR decomposition}
&= \frac{\big(\frac{n}{W_{n}} - 1\big)}{\sqrt{n}}\,\suman  w_{ni}\,\big[\ind \big\{\eps_{1i} \leq F_{1\eps}^{-1}(u_1),  
 \eps_{2i} \leq F_{2\eps}^{-1}(u_2) \big\} - C(u_1,u_2) \big]
\\*
\notag 
& \quad + \frac{1}{\sqrt{n}}\,\suman  (w_{ni}-1)\,\big[\ind \big\{\eps_{1i} \leq F_{1\eps}^{-1}(u_1),  
 \eps_{2i} \leq F_{2\eps}^{-1}(u_2) \big\} - C(u_1,u_2) \big]
\\
\notag 
& = \big(\tfrac{n}{W_{n}} - 1\big)\,B_{n1}(u_1,u_2) + B_{n2}(u_1,u_2), 
\end{align}}
$\! $where $B_{n1}(u_1,u_2)$ stands for the first term on the right-hand 
side of the equation~\eqref{eq: tildeGnOR minus tildeGnOR decomposition} 
(except for the factor $\frac{n}{W_{n}} - 1$) and $B_{n2}(u_1,u_2)$
for the second term. Further 
using standard techniques one 
can show that both $B_{n1}(u_1,u_2)$ and $B_{n2}(u_1,u_2)$ viewed as 
processes on $[0,1]^2$ are asymptotically equi-continuous. 
To this end, note that $B_{n1}(u_1,u_2)$ corresponds to the process $\bar Z_n(f)$ as defined in Section~\textbf{A1} above with 
$ f \hat= \big(c_n, F_{1\eps}^{-1}(u_1), F_{2\eps}^{-1}(u_2),0, 1, 0,1\big)$. Alternatively, results by \cite{bickel1971convergence} can be applied. 
Moreover as 
$$
 \tfrac{n}{W_{n}} - 1 = o_{P}(1), \quad \text{and} \quad 
 \E \big(w_{n}(\XX_{1})-1 \big) = - \PP(\XX_{1} \not \in \JJ_{n}) = o(1), 
$$
one can conclude that both processes $\big(\tfrac{n}{W_{n}} - 1\big)\,B_{n1}(u_1,u_2)$ and $B_{n2}(u_1,u_2)$ are uniformly asymptotically negligible 
in probability, which together with~\eqref{eq: tildeGnOR minus tildeGnOR decomposition} implies~\eqref{eq: as negl of tildeGnOR minus GnOR}. 
%


\bigskip

\renewcommand{\theequation}{B\arabic{equation}}
\setcounter{equation}{0}
\section*{Appendix B - Proof of Theorem~\ref{thm equiv of param estim}}

Thanks to assumption~$\mathbf{(\bphi)}$, 
the estimator $\hatbtheta_n$ is a solution to the estimating 
equations~\eqref{eq: estim equations}.  

In what follows, first we prove the existence of a consistent root 
of the estimating equations~\eqref{eq: estim equations} 
and then we derive that this root satisfies 
\begin{equation} \label{eq: as representation}
 \sqrt{n}\,\big(\hatbtheta_n - \btheta \big) 
 = \bGamma^{-1} \frac{1}{\sqrt{n}}\suman 
     \bphi\big(\Uhat,\Vhat;\btheta\big)    
  + o_{P}(1),   
\end{equation}
where $(\Uhat, \Vhat)$ are introduced in~\eqref{eq: Uhat}. 
The statement of the theorem now follows for $p=1$ by 
Proposition~A\,1(ii) of \cite{genest_et_al_1995} and for $p > 1$ 
by Theorem~1 of \cite{ogvp_testing_rao_2017}.

\subsection*{Proving consistency}
Put 
$
 \widetilde{C}_{n}'(u_1,u_2) = \frac{1}{W_n} \suman w_{ni}\,\ind\big\{\Utilde \leq u_1, \Vtilde \leq u_{2} \big\},   
$
where pseudo-observations $(\Utilde, \Vtilde)$ are defined in~\eqref{eq: Utilde}. Note that 
\begin{equation} \label{eq: equivalency of hatC and tildeC}
 \sup_{(u_1,u_2) \in [0,1]^2} \big|\widetilde{C}_{n}(u_1,u_2) - \widetilde{C}_{n}'(u_1,u_2)\big|  
 = O_{P}\big(\tfrac{1}{W_n}\big) = O_{P}\big(\tfrac{1}{n}\big).  
\end{equation}
%
Fix $l \in \{1,\dotsc,p\}$. 
By Corollary~A.7 of \cite{berghaus2017weak} one gets 
\begin{align}
\notag
  \frac{1}{W_{n}} &\suman w_{ni}\, 
 \phi_{l}\big(\Utilde, \Vtilde;\tt \big) = \int_{0}^{1}\int_{0}^{1} \phi_{l}(v_1,v_{2}; \tt)\,d\widetilde{C}_{n}'(v_1,v_2) 
\\ 
\notag 
 &=  \int_{0}^{1}\int_{0}^{1} \widetilde{C}_{n}'(v_1,v_2)\,d\phi_{l}(v_1,v_{2}; \tt) 
  +   \phi_{l}(1, 1; \tt\big) 
\\
\label{eq: per partes empirical} 
 &\quad - \int_{0}^{1}\widetilde{C}_{n}'(v_1,1)\,d\phi_{l}(v_1,1; \tt) 
  - \int_{0}^{1}\widetilde{C}_{n}'(1,v_2)\,d\phi_{l}(1,v_2; \tt).  
\end{align}
Note that thanks to assumption $\mathbf{(\bphi^{(j)})}$ (uniformly in $\tt \in V(\btheta)$)
\begin{align}
\notag 
 \int_{0}^{1}&\widetilde{C}_{n}'(v_1,1)\,d\phi_{l}(v_1,1; \tt) 
 = \frac{1}{W_n}\suman w_{ni}\,\int_{0}^{1}\ind\{\Utilde \leq v_1\}\,d\phi_{l}(v_1,1; \tt)  
\\ 
\notag 
 &= \frac{1}{W_n}\sum_{i=1}^{W_n}\,\int_{\tfrac{i}{W_{n}+1}}^{1} d\phi_{l}(v_1,1; \tt) 
 = \phi_{l}(1,1; \tt)  - \frac{1}{W_n}\sum_{i=1}^{W_n} \phi_{l}(\tfrac{i}{W_n+1},1; \tt) 
\\ 
\label{eq: Cn with respect to df marginal 1}
& = \phi_{l}(1,1; \tt)  - \int_{0}^{1} \phi_{l}(v_1,1; \tt)\,dv_1 + O_{P}\big(\tfrac{1}{W_n}\big) 
\end{align}
and analogously also 
\begin{equation} 
\label{eq: Cn with respect to df marginal 2}
 \int_{0}^{1}\widetilde{C}_{n}'(1,v_2)\,d\phi_{l}(1,v_2; \tt) 
 = \phi_{l}(1,1; \tt)  - \int_{0}^{1} \phi_{l}(1,v_2; \tt)\,dv_2 + O_{P}\big(\tfrac{1}{W_n}\big).  
\end{equation}
Now combining \eqref{eq: per partes empirical}, \eqref{eq: Cn with respect to df marginal 1} 
and \eqref{eq: Cn with respect to df marginal 2} yields  
\begin{equation}
  \frac{1}{W_{n}}  \suman w_{ni}\, 
 \phi_{l}\big(\Utilde, \Vtilde;\tt 
 \big) 
%
\label{eq: per partes empirical final}  
 = \int_{0}^{1}\int_{0}^{1} \widetilde{C}_{n}'(v_1,v_2)\,d\phi_{l}(v_1,v_{2}; \tt) 
  + A_{l}(\tt) 
 + O_{P}\big(\tfrac{1}{n}\big),  
\end{equation}
where 
\begin{equation} \label{eq: Al}
A_{l}(\tt) = - \phi_{l}(1, 1; \tt\big) 
+ \int_{0}^{1} \phi_{l}(v_1,1; \tt)\,dv_1 + \int_{0}^{1} \phi_{l}(1,v_2; \tt)\,dv_2.  
\end{equation}

Analogously one gets 
\begin{equation}
 \label{eq: per partes true copula final}  
 \Es \phi_{l}\big(U_{11}, U_{21};\tt \big) 
  = \int_{0}^{1}\int_{0}^{1} C(v_1,v_2)\,d\phi_{l}(v_1,v_{2}; \tt) + A_{l}(\tt).    
\end{equation}
Now using \eqref{eq: equivalency of hatC and tildeC},  
\eqref{eq: per partes empirical final}, \eqref{eq: per partes true copula final} 
and assumption~$\mathbf{(\bphi)}$ gives that 
uniformly in $\tt \in V(\btheta)$ 
\begin{align}
\notag 
 \frac{1}{W_{n}}& \suman w_{ni}\, 
 \phi_{l}\big(\Utilde, \Vtilde;\tt \big) - \Es \phi_l(U_{11},U_{21}; \tt) 
\\
\notag 
&= \int_{0}^{1} \int_{0}^{1}  \big[\widetilde{C}_{n}'(v_1,v_2) - C(v_1,v_2) \big]\,d\phi_{l}(v_1,v_{2}; \tt) + O_{P}\big(\tfrac{1}{n}\big) = o_{P}(1), 
\end{align}
where we have used Theorem~\ref{thm equiv of Cn} and assumption~$\mathbf{(\bphi)}$.  
The existence of a consistent root of estimating equations~\eqref{eq: estim equations}  
now follows by assumptions~$\mathbf{(Id)}$ and~$\mathbf{(\bGamma)}$. 

Analogously one can show the existence of a consistent root of estimating equations~\eqref{eq: estim equations orac}.

\subsection*{Showing~\texorpdfstring{\eqref{eq: as representation}}{(B1)}}

Let $\hatbtheta_n$ be a consistent root of the estimating 
equations~\eqref{eq: estim equations}. Then by the mean value 
theorem applied to each coordinate of the vector-valued function 
$$
\boldsymbol{\Psi}_{n}(\tt) = \frac{1}{W_n} \suman w_{ni}\, 
 \bphi\big(\Utilde, \Vtilde;\tt \big) 
$$
one gets  
\begin{align*}
 \boldsymbol{0}_{p} &= \frac{1}{W_n} \suman w_{ni}\, 
 \bphi\Big(\Utilde, \Vtilde;\hatbtheta_n
 \Big)  
\\
 &=  \frac{1}{W_n} \suman w_{ni}\, 
 \bphi\Big(\Utilde, \Vtilde;\btheta
 \Big) 
 + \frac{1}{W_n} \suman w_{ni}\, 
 \mathbb{D}_{\bphi}\Big(\Utilde, \Vtilde;\btheta_{n}^{*}
 \Big) \, \big(\hatbtheta_n - \btheta \big),  
\end{align*}
where $\mathbb{D}_{\bphi}$ stands for 
$\tfrac{\partial \bphi(u_1,u_2;\tt)}{\partial \tt}$ and 
$\btheta_{n}^{*}$ is between $\hatbtheta_n$  and $\btheta$. 
Note that as the mean value theorem is applied to a vector 
valued function there are in fact $p$ different points 
$\btheta_{n}^{*,1},\dotsc,\btheta_{n}^{*,p}$ for each coordinate 
of the function $\boldsymbol{\Psi}_{n}(\tt)$ but all of them 
are consistent so for simplicity of notation we do not distinguish 
them. 

Thus to finish the proof of~\eqref{eq: as representation}  
it is sufficient to show that 
\begin{equation} \label{eq: consistency of Gamma}
 \frac{1}{W_n} \suman w_{ni}\, 
 \mathbb{D}_{\bphi}\big(\Utilde, \Vtilde;\btheta_{n}^{*}
 \big) = \bGamma + o_{P}(1)
\end{equation}
and 
%
\begin{equation}
 \frac{\sqrt{n}}{W_n} \suman w_{ni}\, 
 \bphi\big(\Utilde, \Vtilde;\btheta \big) 
%
\label{eq: as normality of score}
= 
\frac{1}{\sqrt{n}}\suman 
 \bphi\big(\Uhat,\Vhat;\btheta \big) + o_{P}(1). 
\end{equation}
%


When proving \eqref{eq: consistency of Gamma} 
one can mimic the proof of consistency of $\hatbtheta_n$ 
and show that there exists~$V(\btheta)$ (an open neighbourhood of 
$\btheta$ such that)
\begin{equation*} 
 \sup_{\tt \in V(\btheta)}\bigg\| \frac{1}{n} \suman w_{ni}\, 
 \mathbb{D}_{\bphi}\big(\Utilde, \Vtilde;\tt \big) 
 - \Es\mathbb{D}_{\bphi}(U_{11}, U_{21};\tt)  
   \bigg\|= o_{P}(1).   
\end{equation*}
Using the consistency of $\hatbtheta_n$ and assumption~$\mathbf{(\bGamma)}$ 
yields~\eqref{eq: consistency of Gamma}.

\medskip 

Thus one can concentrate on proving \eqref{eq: as normality of score}.  
Put 
$
 C_{n}'^{(or)}(u_1,u_2) 
  = \frac{1}{n} \suman \ind\big\{\Uhat \leq u_1, \Vhat \leq u_{2} \big\},   
$
where $(\Uhat, \Vhat)$ are defined in~\eqref{eq: Uhat}. Note that 
\begin{equation} \label{eq: equivalency of hatCor and tildeCor}
 \sup_{(u_1,u_2) \in [0,1]^2} \big|C_{n}^{(or)}(u_1,u_2) - C_{n}'^{(or)}(u_1,u_2)\big|  
 = O_{P}\big(\tfrac{1}{n}\big).  
\end{equation}
Analogously as~\eqref{eq: per partes empirical final} 
one can also show that for $l=1,\dotsc,p$  
\begin{align}
\notag 
 \frac{1}{n} \suman 
 \phi_{l}\big(\Uhat, \Vhat;\btheta \big) 
& = \int_{0}^{1} \int_{0}^{1}  \phi_l(v_1,v_2; \btheta)\,dC_{n}'^{(or)}(v_1,v_2)  
\\ 
\label{eq: per partes empirical orac} 
 & = \int_{0}^{1} \int_{0}^{1}  
      C_{n}'^{(or)}(v_1,v_2)\,d\phi_l(v_1,v_2; \btheta)\,
  + A_{l}(\btheta) + O_{P}\big(\tfrac{1}{n}\big),  
\end{align}
where $A_{l}(\btheta)$ is given in~\eqref{eq: Al}. 

Now using \eqref{eq: equivalency of hatC and tildeC}, 
\eqref{eq: per partes empirical final}, 
\eqref{eq: equivalency of hatCor and tildeCor},  
\eqref{eq: per partes empirical orac}, 
Theorem~\ref{thm equiv of Cn} and $\mathbf{(\bphi)}$  one gets 
%
\begin{align*}
\notag 
 \frac{\sqrt{n}}{W_{n}}& \suman w_{ni}\, 
 \phi_{l}\big(\Utilde, \Vtilde;\btheta \big) - \frac{1}{\sqrt{n}} \suman 
 \phi_{l}\big(\Uhat, \Vhat;\btheta\big)
\\
 &=  \int_{0}^{1} \int_{0}^{1}  
   \, \sqrt{n}\,\big[\widetilde{C}_{n}'(u_1,u_2) - C_{n}'^{(or)}(u_1,u_2) \big]\,
  d\phi_{l}(u_1,u_2; \btheta) + o_{P}\big(\tfrac{1}{\sqrt{n}}\big) 
 = o_{P}(1), 
\end{align*}
which verifies~\eqref{eq: as normality of score} and finishes the proof 
of~\eqref{eq: as representation}.

\renewcommand{\theequation}{C\arabic{equation}}
\setcounter{equation}{0}
\section*{Appendix C - Auxiliary results}

%

%
\begin{lemma} \label{lemma hatm and hatsigma}
Assume that $\mathbf{(\boldsymbol{\beta})}$, $\mathbf{(F_{\beps})}$, $\mathbf{(M)}$, 
$\mathbf{(F_{\XX})}$, $\mathbf{(Bw)}$, 
$\mathbf{(k)}$, $\mathbf{(J_{n})}$ and $\mathbf{(m\bsigma)}$ are 
satisfied.  Then 
%
%
%
there exist random functions~$\hata_{j}$ and $\hatb_{j}$ on $\JJ_n$ such that for $j=1,2$ 
\begin{enumerate}
\vspace{0.2cm}
\item
$\displaystyle \sup_{\xx \in \JJ_n}\Big|\frac{\hatm_{j}(\xx) - m_{j}(\xx)}{\sigma_j(\xx)} - \hata_{j}(\xx)\Big| = o_{P}\big(n^{-1/2}\big), 
 \  
 \sup_{\xx \in \JJ_n}\Big|\frac{\hatsigma_{j}(\xx) }{\sigma_j(\xx)} - 
     \hatb_{j}(\xx)\Big| = o_{P}\big(n^{-1/2}\big),  
$

\vspace{0.2cm}
\item
%
$\displaystyle \|\hata_j\|_{d+\delta}=o_P(1),\quad \|\hatb_j-1\|_{d+\delta}=o_P(1)$
for $\delta>0$ from assumption $\mathbf{(Bw)}$, 

\vspace{0.2cm}
\item
$\displaystyle
 \sup_{\xx \in \JJ_n}\big|\hata_j(\xx)\big| = o_{P}\big(n^{-1/4}\big), 
 \qquad 
 \sup_{\xx \in \JJ_n}\big|\hatb_j(\xx) - 1\big| = o_{P}\big(n^{-1/4}\big),
$

\vspace{0.2cm}
\item
$\displaystyle
 \int_{\JJ_n} \hata_j(\xx)f_{\XX}(\xx)\, d\xx = \frac{1}{n}\sum_{i=1}^n \eps_{ji}+o_{P}\big(n^{-1/2}\big)$, 

$ \displaystyle 
 \int_{\JJ_n} \Big(\hatb_j(\xx) - 1\Big)\,f_{\XX}(\xx)\, d\xx 
  = \frac{1}{2n}\sum_{i=1}^n \big(\eps_{ji}^2 - 1 \big)+o_{P}\big(n^{-1/2}\big).
$
\end{enumerate}

%
\end{lemma}
\begin{proof}
For ease of presentation we set $j=1$ and assume $\hh_{n} = \big(h_n, \dotsc, h_n\big)$. 
We will first prove the assertions for $\hatm_1 $. The proof basically 
goes along the lines of the proof of Lemma~1 by \cite{MSW2009}, but changes 
are necessary due to the dependency of observations in our model and because our covariate density is not assumed to be bounded away from zero on its support. 
Recall that $\setI(d,p)$ denotes the set of multi-indices $\ii=(i_1,\dotsc,i_d)$ with $i.=i_1+\dots+i_d\leq p$ 
and we set $\setI=\setI(d,p)$, where $p$ is the order of the polynomials used in the local polynomial estimation. 
Further introduce 
$\JJ_{n}^{+}=[-c_n-h_n,c_n+h_n]^d$  
and note thanks to assumption $\mathbf{(Bw)}$ 
\begin{equation} \label{eq: alpha one}
 \alpha_n^{(1)}:=\inf_{\xx\in \JJ_{n}^{+}}f_{\XX}(\xx) \geq \frac{1}{(\log n)^{q}}
  \quad \text{for some } q > 0,
\end{equation}
as for all sufficiently large~$n$ the set $\JJ_{n}^{+}$ is a subset of 
 $\JJ_{2n}$. 
Finally define \linebreak $\alpha_n^{(2)}:=\min_{j=1,2}\inf_{\xx\in \JJ_n}\sigma_j(\xx)$ 
which is by assumption $\mathbf{(m\bsigma)}$ either bounded away from zero or converges 
to zero not faster than a negative power of $\log n$. 

\smallskip 

\noindent \underline{Proof of assertion (i) for $\hatm_{1}$}. 
Fix some $\xx\in\JJ_n$ and let $\widehat{\bbeta}$ denote the solution of the minimization problem~\eqref{minimization}. Then $\widehat{\bbeta}$  satisfies the normal equations
$$
 A_{\ii}(\xx)+B_{\ii}(\xx)-\sum_{\kk \in \setI} \widehat{Q}_{\ii\kk}(\xx)\widehat{\beta}_{\kk} = 0 
 \quad\forall \ii \in \setI,
$$
where
\begin{eqnarray*}
A_{\ii}(\xx) &=& \frac{1}{n}\sum_{\ell=1}^n\sigma_1(\XX_\ell)\,
\eps_{1\ell}\,\psi_{\ii,\hh_n}(\XX_\ell-\xx)\,K_{\hh_n}(\XX_\ell-\xx), 
\\
B_{\ii}(\xx) &=& \frac{1}{n}\sum_{\ell=1}^nm_1(\XX_\ell)\psi_{\ii,\hh_n}(\XX_\ell-\xx)K_{\hh_n}(\XX_\ell-\xx),
\\
\widehat{Q}_{\ii\kk}(\xx) &=& \frac{1}{n}\sum_{\ell=1}^n\psi_{\ii,\hh_n}(\XX_\ell-\xx)\,\psi_{\kk,\hh_n}(\XX_\ell-\xx)\,K_{\hh_n}(\XX_\ell-\xx). 
\end{eqnarray*} 
From Theorem 2 in \cite{hansen2008} we obtain for $\varrho_n=\big(\log n/(nh_n^d)\big)^{1/2}$,
\begin{equation}\label{convergence hat Q}
\sup_{\xx\in \JJ_n}\big|\widehat{Q}_{\ii\kk}(\xx) - Q_{\ii\kk}(\xx)\big| = O_P(\varrho_n),
\end{equation}
where we define $Q_{\ii\kk}(\xx)=\E\,\big[\widehat{Q}_{\ii\kk}(\xx)\big]$, $\ii,\kk \in \setI$. Note that
$$
 Q_{\ii\kk}(\xx)=\int \psi_{\ii,(1,..1)}(\uu)\,\psi_{\kk,(1,..1)}(\uu)\,f_{\XX}(\xx+h_n\uu)\,K(\uu)\,d\uu,
$$
and for  $\xx\in\JJ_n$, consider the matrices $\QQ(\xx)$ with entries $Q_{\ii\kk}(\xx)$, $\ii,\kk \in \setI$.
Analogously put $\widehat{\QQ}(\xx)$ for the matrix with entries $\widehat{Q}_{\ii\kk}(\xx)$, $\ii,\kk \in \setI$. 

It follows from  \eqref{eq: alpha one}
that $0<\lambda_n\leq \aa\tr \QQ(\xx)\,\aa\leq \Lambda<\infty$ for all vectors $\aa$ of unit Euclidean length,
 where $\lambda_n$ is a sequence of positive real numbers of the same rate as $\alpha_n^{(1)}$ in  \eqref{eq: alpha one}.  
Thus $\QQ(\xx)$  has eigenvalues in the interval $[\lambda_n,\Lambda]$, and on the event
\begin{equation*} 
 E_n=\bigg\{\sup_{\xx\in \JJ_n} \big\|\widehat{\QQ}(\xx)-\QQ(\xx)\big\|\leq \frac{\lambda_n}{2}\bigg\}
\end{equation*}
one has $\aa\tr \widehat{\QQ}(\xx)\,\aa\geq \lambda_n/2$ for all $\aa$ of unit Euclidean length, 
such that the matrix $\widehat{\QQ}(\xx)$ is invertible as well.
Here and throughout $\|\QQ\|$ denotes the spectral norm of a matrix $\QQ$. 
Note that $\PP(E_n)\to 1$ by \eqref{convergence hat Q} and $\varrho_n=o\big(\alpha_n^{(1)}\big)$, which holds under assumption $\mathbf{(Bw)}$.
For the remainder of the proof we assume that the event $E_n$ takes place because its complement does not matter for the assertions of the lemma.  It follows from the normal equations that for $\xx\in\JJ_n$,
$$
 \hatm_1(\xx)= \ee \tr \, \widehat{\QQ}^{-1}(\xx)\big(\AA(\xx) + \BB(\xx)\big),
$$
where $\ee=(1,0,\dotsc,0)\tr $ and $\AA(\xx)$ and $\BB(\xx)$ denote the vectors with components $A_{\ii}(\xx)$ and $B_{\ii}(\xx)$, $\ii \in \setI$, respectively.
Now define
\begin{equation} \label{eq: ahat}
 \hata_1(\xx) = \ee \tr \, \QQ^{-1}(\xx)\AA(\xx)\frac{1}{\sigma_1(\xx)},\quad \xx\in \JJ_n,
\end{equation}
then we have  the decomposition 
\begin{equation} \label{eq: mhat minus ahat}
 \frac{\hatm_1(\xx)-m_1(\xx)}{\sigma_1(\xx)} - \hata_1(\xx)=r_1(\xx)+r_2(\xx)
\end{equation}
with remainder terms
\begin{eqnarray*}
r_1(\xx) &=& \ee \tr \big(\widehat{\QQ}^{-1}(\xx)-\QQ^{-1}(\xx)\big) \AA(\xx)\frac{1}{\sigma_1(\xx)},
\\
r_2(\xx)&=& \ee \tr \, \widehat{\QQ}^{-1}(\xx)\big(\BB(\xx)-\widehat{\QQ}(\xx)\overline{\bbeta}(\xx)\big)
 \frac{1}{\sigma_1(\xx)},
\end{eqnarray*}
where $\overline{\bbeta}(\xx)$ is the vector with components $\overline{\beta}_{\ii}(\xx)=h_n^{i.}\,D^{\ii} m_1(\xx)$, 
$\ii \in \setI$.
From Theorem 2 in \cite{hansen2008} we obtain  
\begin{eqnarray}\label{rate A}
\sup_{\xx\in \JJ_n} \big|A_{\ii}(\xx)\big| = O_P(\varrho_n), \quad \ii \in \setI.
\end{eqnarray}
%
For the treatment of the inverse matrices in $r_1(\xx)$ we use Cramer's rule and write
\begin{eqnarray*}
\widehat{\QQ}^{-1}(\xx)-\QQ^{-1}(\xx) &=& \frac{1}{\det\big(\widehat{\QQ}(\xx)\big)}\big(\widehat{\bold{C}}(\xx)\big)\tr-\frac{1}{\det\big({\QQ}(\xx)\big)}\big({\bold{C}}(\xx)\big)\tr
\\
&=& \frac{\det\big({\QQ}(\xx)\big)-\det\big(\widehat{\QQ}(\xx)\big)}
 {\det\big(\widehat{\QQ}(\xx)\big)\det\big({\QQ}(\xx)\big)}\big(\widehat{\bold{C}}(\xx)\big)\tr 
 + \frac{1}{\det\big({\QQ}(\xx)\big)}\,\big(\widehat{\bold{C}}(\xx)-{\bold{C}}(\xx)\big)\tr,
\end{eqnarray*}
where $\widehat{\bold{C}}(\xx)$ and ${\bold{C}}(\xx)$ denote the cofactor matrices of $\widehat{\QQ}(\xx)$ and ${\QQ}(\xx)$, respectively. Due to the boundedness of the functions $Q_{\ii\kk}$ each element of $\widehat{\bold{C}}(\xx)-{\bold{C}}(\xx)$ can be absolutely bounded by $O_P(\varrho_n)$ by \eqref{convergence hat Q} and the same rate is obtained for $\big|\det\big({\QQ}(\xx)\big)-\det\big(\widehat{\QQ}(\xx)\big)\big|$, uniformly in~$\xx$. Using the lower bound $\lambda_n^{|\setI|}$ for the determinant of $\QQ(\xx)$, and assumption $\mathbf{(m\bsigma)}$ to bound $1/\sigma_1$ gives the rate  
\begin{equation}\label{rate-r1}
\sup_{\xx\in\JJ_n} \big|r_1(\xx)\big| 
 = O_P\Big(\tfrac{(\varrho_n)^2}{(\alpha_n^{(1)})^{2|\setI|}\,\alpha_n^{(2)}}\Big)
 = o_P\big(n^{-1/2}\big)
\end{equation}
%
by assumption $\mathbf{(Bw)}$. In order to show negligibility of $r_2(\xx)$ first note that the  spectral norm of $\widehat{\QQ}^{-1}(\xx)$ is given  by the reciprocal of the square root of the smallest eigenvalue of $\widehat{\QQ}(\xx)\tr \widehat{\QQ}(\xx)$. With 
$$
 \big(\aa\tr \widehat{\QQ}(\xx)\tr \widehat{\QQ}(\xx)\aa \big)^{1/2}
 = \big\|\widehat{\QQ}(\xx)\aa \big\|
   \geq (\aa\tr  \QQ(\xx)\tr  \QQ(\xx)\aa)^{1/2} 
   - \big\|\widehat{\QQ}(\xx)-\QQ(\xx)\big\|
 \geq\frac{\lambda_n}{2}
$$
(on $E_n$) for all $\aa$ with $||\aa||=1$,
we obtain the rate $O(\lambda_n^{-1})$ for $\big\|\widehat{\QQ}^{-1}(\xx)\big\|$.
Further, by Taylor expansion of $m_1(\XX_\ell)$ of order $p+1$ in the definition of $B_{\ii}(\xx)$  
and using assumption $\mathbf{(m\bsigma)}$ we have
\begin{eqnarray*}
\big\|B_{\ii}(\xx) - \big(\widehat{\QQ}(\xx)\overline{\bbeta}(\xx)\big)_{\ii} \big\| &\leq&
M_n\,\frac{1}{n}\sum_{\ell=1}^nh_n^{p+1}\,\|\psi_{\ii,\hh_n}(\XX_\ell-\xx)\|\,K_{\hh_n}(\XX_\ell-\xx)\\
&=& O\big(M_n h_n^{p+1}\big)\,\widehat{f}_{\XX}(\xx),
\end{eqnarray*}
%
where the kernel density estimator 
$\widehat{f}_{\XX}(\xx)=\frac{1}{n}\sum_{\ell=1}^nK_{\hh_n}(\XX_\ell-\xx)$ 
converges to $f_{\XX}(\xx)$ uniformly in $\xx\in \JJ_n$, see Theorem 6 by \cite{hansen2008}. 
 Altogether
we have 
\begin{equation}\label{rate-r2}
\sup_{\xx\in\JJ_n} \big|r_2(\xx)\big| 
  = O_P\Big(\tfrac{M_nh_n^{p+1}}{\alpha_n^{(1)}\alpha_n^{(2)}}\Big)=o_P\big(n^{-1/2}\big)
\end{equation}
using assumption $\mathbf{(Bw)}$. 

Now assertion (i) for $\hatm_1$ follows from \eqref{eq: ahat}, \eqref{eq: mhat minus ahat}, 
\eqref{rate-r1} and \eqref{rate-r2}. 

\noindent \underline{Proof of assertion (ii) for $\hata_1$}. 
%
Note that $p\geq d$ and thus $\hata_1$ is ($d+1$)-times partially differentiable and 
\begin{align*}
\|\hata_1\|_{d+\delta}&= \max_{\ii \in \setI(d,d)} \sup_{\xx \in \JJ_n} \big|D^{\ii}\hata_1(\xx)\big| 
\\
&\quad + \max_{\ii \in \setI(d,d)\atop i.=d}\max\left\{ \sup_{\xx,\xx' \in \JJ_n\atop \|\xx-\xx'\|\leq h_n} \frac{|D^{\ii}\hata_1(\xx)-D^{\ii}\hata_1(\xx')|}{\|\xx-\xx'\|^{\delta}}, \sup_{\xx,\xx' \in \JJ_n\atop \|\xx-\xx'\|> h_n} \frac{|D^{\ii}\hata_1(\xx)-D^{\ii}\hata_1(\xx')|}{\|\xx-\xx'\|^{\delta}}\right\}
\\
&\leq  \max_{\ii \in \setI(d,d)} \sup_{\xx \in \JJ_n} \big|D^{\ii}\,\hata_1(\xx)\big| \\
& + \max_{\ii \in \setI(d,d+1)\atop i.=d+1}\sup_{\xx \in \JJ_n} \big|D^{\ii}\hata_1(\xx)\big| 
 h_n^{1-\delta} + 2\max_{\ii \in \setI(d,d)\atop i.=d}\sup_{\xx \in \JJ_n} \big|D^{\ii}\hata_1(\xx)\big| h_n^{-\delta}
\end{align*}
by the mean value theorem.
Again by Theorem 2 of \cite{hansen2008} we have
$$
 \sup_{\xx\in \JJ_n} \big|h_n^{i.} D^{\ii} A_{\kk}(\xx)\big| = O_P(\varrho_n),\quad \ii \in \setI(d,d+1).
$$
Further note that 
$$
 \frac{\partial }{\partial x_k}\QQ^{-1}(\xx)=\QQ^{-1}(\xx) \bigg(\frac{\partial }{\partial x_k}\QQ(\xx)\bigg)\QQ^{-1}(\xx)
$$
and that the spectral norm of $\QQ^{-1}(\xx)$ can be bounded by $O\big(1/\alpha_n^{(1)}\big)$ with considerations as before. 
We apply the product rule for derivatives to obtain
\begin{align*}
\|\hata_1\|_{d+\delta}
&\leq \max_{\ell=1,\dotsc,d}\sum_{j=0}^\ell\sum_{k=0}^j O_P \Big(\tfrac{\varrho_n h_n^{-k}M_n^{\ell-j}}{(\alpha_n^{(1)})^{j-k+1}(\alpha_n^{(2)})^{\ell-j+1}}\Big)
%
+\sum_{j=0}^{d+1}\sum_{k=0}^j O_P \Big(\tfrac{\varrho_n h_n^{-k}M_n^{d+1-j}}{(\alpha_n^{(1)})^{j-k+1}(\alpha_n^{(2)})^{d-j+2}}\Big)h_n^{1-\delta}\\
&\qquad  + 2\sum_{j=0}^d\sum_{k=0}^j O_P \Big(\tfrac{\varrho_n h_n^{-k}M_n^{d-j}}{(\alpha_n^{(1)})^{j-k+1}(\alpha_n^{(2)})^{d-j+1}}\Big)h_n^{-\delta}
\\
&= O_P \Big(\tfrac{\varrho_n h_n^{-d}}{\alpha_n^{(1)}\alpha_n^{(2)}}\Big) 
 + O_P \Big(\tfrac{\varrho_n h_n^{-d-1}}{\alpha_n^{(1)}\alpha_n^{(2)}}\Big)h_n^{1-\delta}
 + O_P \Big(\tfrac{\varrho_n h_n^{-d}}{\alpha_n^{(1)}\alpha_n^{(2)}}\Big)h_n^{-\delta}
\\
&=
O_P\Big(\tfrac{\varrho_n}{h_n^{d+\delta}\alpha_n^{(1)}\alpha_n^{(2)}}\Big)\;=\;o_P(1)
\end{align*}
by assumption $\mathbf{(Bw)}$. 
 Assertion (ii) for $\hata_1$ follows.

\smallskip 

\noindent \underline{Proof of assertion (iii) for $\hata_1$}. 
From the definition of $\hata_1$ and (\ref{rate A}) we obtain that
\begin{equation*} 
\sup_{\xx\in\JJ_n}\big|\hata_1(\xx)\big| 
 =  O_P\Big(\tfrac{\varrho_n}{\alpha_n^{(1)}\alpha_n^{(2)}}\Big) = o_P\big(n^{-1/4}\big)
\end{equation*}
and thus (iii) follows for $\hata_1$. 

\smallskip 

\noindent \underline{Proof of assertion (iv) for $\hata_1$}. 
To prove (iv) note that
\begin{eqnarray*}
\int_{\JJ_n}\hata_1(\xx)f_{\XX}(\xx)\,d\xx  &=& \frac{1}{n}\sum_{i=1}^n \eps_{1i}\,\Delta_n(\XX_i)
\end{eqnarray*}
with
\begin{eqnarray*}
\Delta_n(\XX_i) &=& \sigma_1(\XX_i)\int\frac{1}{\sigma_1(\xx)} \ee\tr\, \QQ^{-1}(\xx)\, 
 \bpsi_{\hh_n}(\XX_i-\xx)\,K_{\hh_n}(\XX_i-\xx)\,f_{\XX}(\xx)\, d\xx
\\*
&=& \sigma_1(\XX_i)\int_{\big[\frac{\XX_i-c_n}{h_n},\frac{\XX_i+c_n}{h_n}\big]} 
 \ee\tr\,  \QQ^{-1}(\XX_i-\uu h_n)\,\bpsi(\uu)\,K(\uu)\,\frac{f_{\XX}(\XX_i-\uu h_n)}{\sigma_1(\XX_i-\uu h_n)}\, d\uu.
\end{eqnarray*}
From the support properties of the kernel function it follows that $\Delta_n(\XX_i)\ind\{\XX_i\not\in \JJ_{n}^{+}\}=0$. 
Further, for $\JJ_{n}^{-}=[-c_n+h_n,c_n-h_n]^d$ note that
\begin{equation} \label{eq: In minus Kn}
\frac{1}{n}\sum_{i=1}^n \eps_{1i}\,\Delta_n(\XX_i)\,\ind \{\XX_i\in \JJ_{n}^{+}\setminus\JJ_{n}^{-}\} 
= o_P\big(n^{-1/2}\big)
\end{equation}
%
%
because the expectation is zero and the variance is bounded by
\begin{multline*}
\frac{1}{n}\sup_{\xx\in \JJ_n}\frac{1}{\sigma_1(\xx)}\frac{1}{\lambda_n}
 \sup_{\xx\in\mathbb{R}^d}f_{\XX}(\xx) 
 \,\Es \big[\ind\big\{\XX_1\in \JJ_{n}^{+}\setminus \JJ_{n}^{-}\big\}\,\sigma_1^2(\XX_1)\big]  
\\ 
\leq \frac{1}{n}\, O\Big(\tfrac{1}{\alpha_n^{(1)}\alpha_n^{(2)}}\Big)
 M_{n} \,\PP\big(\XX_1\in \JJ_{n}^{+}\setminus \JJ_{n}^{-}\big)  
= \frac{1}{n}\,O\Big(\tfrac{M_n\,h_n}{\alpha_n^{(1)}\alpha_n^{(2)}}\Big)\;=\; o\big(n^{-1}\big).
\end{multline*}
It remains to consider
\begin{eqnarray*}
\frac{1}{n}\sum_{i=1}^n \eps_{1i}\Delta_n(\XX_i)\ind \{\XX_i\in \JJ_{n}^{-}\},
\end{eqnarray*}
with
$\Delta_n(\XX_i)= \Delta_n^{(1)}(\XX_i)+\Delta_n^{(2)}(\XX_i)$,
where
\begin{eqnarray*}
\Delta_n^{(1)}(\XX_i)
&=& \int_{[-1,1]^d} \ee \tr \,  \QQ^{-1}(\XX_i-\uu h_n)\,\bpsi(\uu)\,K(\uu)\,f_{\XX}(\XX_i-\uu h_n)\, d\uu,
\\
\Delta_n^{(2)}(\XX_i)
&=& \int_{[-1,1]^d} \ee \tr \,  \QQ^{-1}(\XX_i-\uu h_n)\,\bpsi(\uu)\,K(\uu)\,f_{\XX}(\XX_i-\uu h_n)\Big(\tfrac{\sigma_1(\XX_i)}{\sigma_1(\XX_i-\uu h_n)}-1\Big)\, d\uu.
\end{eqnarray*}
Now, by applying the mean value theorem for $\sigma_1$, for $\XX_i\in\JJ_{n}^{-}$, $\Delta_n^{(2)}(\XX_i)$ 
can be bounded absolutely by $O\big(M_n\,h_n/(\alpha_n^{(1)}\alpha_n^{(2)})\big)=o(1)$. 
Thus analogously as when showing~\eqref{eq: In minus Kn} one can use Markov's inequality to get 
\begin{equation*}
\int_{\JJ_n}\hata_1(\xx)f_{\XX}(\xx)\,d\xx - \frac{1}{n}\sum_{i=1}^n \eps_{1i} 
 = \frac{1}{n}\sum_{i=1}^n \eps_{1i} \big(\Delta_n^{(1)}(\XX_i)-1 \big) \ind \{\XX_i\in \JJ_{n}^{-}\} 
  + o_P\big(n^{-1/2}\big).
\end{equation*}
To obtain the desired negligibility it remains to show  
 $\E \big[\big(\Delta_n^{(1)}(\XX_i)-1\big)^2\ind \{\XX_i\in \JJ_{n}^{-}\} \big]\to 0$. 
To this end we write
$$
 1 = \int_{[-1,1]^d} \ee \tr \, \QQ_*^{-1}(\xx-h_n\uu)\bpsi(\uu)K(\uu)f_{\XX}(\xx-h_n\uu)\,d\uu,
$$
%
where the matrix $\QQ_*(\xx)$ has entries  
$$
 Q_{*ik}(\xx) = f_{\XX}(\xx)\int\psi_{\ii}(\uu)\psi_{\kk}(\uu)K(\uu)\,d\uu, \qquad \ii,\kk \in \setI. 
$$ 
Note that $\QQ_*(\xx)$ has the smallest eigenvalue of order $\lambda_n$. Thus we obtain the bound
\begin{align*}
\E&\big[\big(\Delta_n^{(1)}(\XX_i)-1\big)^2 \ind \{\XX_i\in \JJ_{n}^{-}\} \big]
\\
&=
 \E\bigg[\Big(\int_{[-1,1]^d} \ee \tr \,\big(\QQ^{-1}(\XX_i-\uu h_n) - \QQ_*^{-1}(\XX_i-\uu h_n)\big)\,
\\
& \qquad \qquad \qquad \qquad  \qquad \bpsi(\uu)K(\uu)f_{\XX}(\XX_i-\uu h_n)\, d\uu\Big)^2\,\ind \{\XX_i\in \JJ_{n}^{-}\}\bigg]
\\
&\leq O(1)\int_{\JJ_{n}^{-}}\int_{[-1,1]^d} \big\|\QQ^{-1}(\xx-\uu h_n) - \QQ_*^{-1}(\xx-\uu h_n)\big\|^2\, d\uu\,f_{\XX}(\xx)\,d\xx\\
&\leq O(1)\sup_{\xx\in\JJ_{n}}\|\QQ^{-1}(\xx)\|^2\sup_{\xx\in\JJ_{n}}\|\QQ_*^{-1}(\xx)\|^2 
\\
& \qquad \qquad \qquad \qquad \qquad \int_{\JJ_{n}^{-}}\int_{[-1,1]^d} \|\QQ(\xx-\uu h_n)-\QQ_*(\xx-\uu h_n)\|^2\, d\uu\,f_{\XX}(\xx)\,d\xx.
\end{align*}
Now with bounds for the matrix norms similar to before, and inserting the definitions of 
$\QQ$ and $\QQ_*$ we obtain
\begin{align*}
&\E\big[(\Delta_n^{(1)}(\XX_i)-1)^2\ind \{\XX_i\in \JJ_{n}^{-}\} \big]\\
&\leq O\Big(\tfrac{1}{(\alpha_n^{(1)})^4}\Big)\int_{\JJ_{n}^{-}}\int_{[-1,1]^d}\int_{[-1,1]^d} |f_{\XX}(\xx-\uu h_n+h_n\vv)-f_{\XX}(\xx-\uu h_n)|^2 K(\vv)\,d\vv\, d\uu\,f_{\XX}(\xx)\,d\xx\\
&= O\Big(\tfrac{h_n^2}{(\alpha_n^{(1)})^4}\Big)\;=\;o(1)
\end{align*}
by the mean value theorem and assumptions $\mathbf{(F_X)}$ and $\mathbf{(k)}$. 

\medskip

\noindent \underline{Proof of assertions (i)--(iv) for $\widehat\sigma_1$}.  
Recall the definition $\widehat\sigma_1^2=\widehat s_1-\hatm_1^2$, where $\widehat s_1$ is the local polynomial estimator based on $(\XX_i,Y_{1i}^2)$, $i=1,\dotsc,n$. With the notation  $s_1(\xx)=\E[Y_{1i}^2\mid \XX_i=\xx]=\sigma_1^2(\xx)+m_1^2(\xx)$ we obtain
\begin{eqnarray*}
\frac{\widehat\sigma_1(\xx)}{\sigma_1(\xx)} 
&=& 1+\frac{\widehat s_1(\xx)-s_1(\xx)}{2\sigma_1^2(\xx)}-\frac{\hatm_1(\xx)-m_1(\xx)}{\sigma_1(\xx)}\frac{m_1(\xx)}{\sigma_1(\xx)}+r(\xx),
\end{eqnarray*}
where 
\begin{eqnarray*}
r(\xx)&=&-\frac{1}{2}\frac{(\hatm_1(\xx)-m_1(\xx))^2}{\sigma_1^2(\xx)}-\frac{(\widehat\sigma_1^2(\xx)-\sigma_1^2(\xx))^2}{2\sigma_1^{2}(\xx)(\widehat\sigma_1(\xx)+\sigma_1(\xx))^2}.
\end{eqnarray*}
Put
$$
 \widehat{c}_1(\xx) = \ee \tr \, \QQ^{-1}(\xx)\widetilde{\AA}(\xx)\frac{1}{2\,\sigma_1^{2}(\xx)},
\quad \xx\in \JJ_n,
$$
where $\widetilde{\AA}(\xx)$ denotes the vector with components 
$$
 \widetilde{A}_{\ii}(\xx) = \frac{1}{n}\sum_{\ell=1}^{n}\big[2\,m_{1}(\XX_{\ell})\sigma_{1}(\XX_{\ell})\,\eps_{1\ell}+ \sigma_1^{2}(\XX_\ell)\,
\big(\eps_{1\ell}^{2}-1\big)\big]\psi_{\ii,\hh_n}(\XX_\ell-\xx)\,K_{\hh_n}(\XX_\ell-\xx), 
 \quad \ii \in \setI.
$$
Along the lines of the proof of (i) and (ii) for $\hatm_1$ 
one can prove that 
%

%
$$
 \sup_{\xx\in\JJ_n}\bigg|\frac{\widehat s_1(\xx)-s_1(\xx)}{2\sigma_1^2(\xx)}-\widehat c_1(\xx)\bigg| 
=O_P\Big(\tfrac{\varrho_n^2}{(\alpha_n^{(1)})^{|\setI|}(\alpha_n^{(2)})^2}\Big)+O_P\Big(\tfrac{M_nh_n^{p+1}}{ \alpha_n^{(1)}(\alpha_n^{(2)})^2}\Big)=o_P\big(n^{-1/2}\big)
$$
%
and
$$
 \sup_{\xx\in\JJ_n}|\widehat c_1(\xx)|=O_P\Big(\tfrac{\varrho_n}{\alpha_n^{(1)}(\alpha_n^{(2)})^2}\Big)=o_P\big(n^{-1/4}\big). 
$$
Now noticing that
$\widehat\sigma_1^2(\xx)-\sigma_1^2(\xx)=\widehat s_1(\xx)-s_1(\xx)-(\hatm_1(\xx)-m_1(\xx))(\hatm_1(\xx)+m_1(\xx))$
we obtain the rate
$$ 
 \sup_{\xx\in\JJ_n}|r(\xx)|=o_P\big(n^{-1/2}\big)+O_P\Big(\tfrac{\varrho_n^2}{(\alpha_n^{(1)})^2(\alpha_n^{(2)})^4}\Big)=o_P\big(n^{-1/2}\big)
$$
and (i) follows for $\widehat\sigma_1$. 

If we define $\hatb_1(\xx)=1+\widehat c_1(\xx)-\hata_1(\xx)m_1(\xx)/\sigma_1(\xx)$, then (ii) and (iii) follow analogously to before. The only difference is an additional factor $\sigma_1(\xx)$ in the denominator that needs to be considered. 

To show validity of (iv) note that the regression model
$Y_{1i}^2=s_1(\XX_i)+\eta_i$ holds with error term $\eta_i=\sigma_i^2(\XX_i)(\varepsilon_{1i}^2-1)+2m_1(\XX_i)\sigma_1(\XX_i)\varepsilon_{1i}$. From this one obtains analogously to the derivation of (iv) for $\hata_1$ that
\begin{eqnarray*}
\int_{\JJ_n}\frac{\widehat{s}_1(\xx)}{2\sigma_1^2(\xx)}f_{\XX}(\xx)\,d\xx
&=& \frac{1}{2n}\sum_{i=1}^n (\varepsilon_{1i}^2-1)+\frac{1}{n}\sum_{i=1}^n \frac{m_1(\XX_i)}{\sigma_1(\XX_i)}\varepsilon_{1i}+o_P\big(n^{-1/2}\big).
\end{eqnarray*}
But the second sum is also the dominating term in $\int_{\JJ_n}\hata_1(\xx)m_1(\xx)/\sigma_1(\xx)f_{\XX}(\xx)\,d\xx$, which is again shown analogously to the proof of (iv) for $\hata_1$. Thus (iv) follows for $\hatb_1$. 
\end{proof}

\medskip

\begin{remark} \label{rem: extending a and b}
Note that due to property (iii) of Lemma \ref{lemma hatm and hatsigma} and 
\eqref{eq: alpha one} 
we have for $\xx\in\JJ_n=[-c_n,c_n]^d$,
$$\|\xx\|^\nu|\,\hata_1(\xx)|\leq O((\log n)^{\nu/d})o(n^{-1/4})=o(1)$$
for every $\nu>0$.
In the proof of Lemma  \ref{lemma hatm and hatsigma}, $\hata_1(\xx)$ was only defined for $\xx\in\JJ_n$. Now we define $\hata_1$ on $\mathbb{R}^d$ in a way that  if $\hata_1\in C_1^{d+\delta}(\JJ_n)$ and $\|\xx\|^\nu|\hata_1(\xx)|\leq 1$, then $\hata_1\in\mathcal{G}$ defined in (\ref{G}).
Then $\PP\big(\hata_1\in\mathcal{G}\big)\to 1$ by  Lemma \ref{lemma hatm and hatsigma}.  Analogously $\hatb_1$ is defined on $\mathbb{R}^d$ such that $\PP\big(\hatb_1\in\widetilde{\mathcal{G}} \big)\to 1$ for $\widetilde{\mathcal{G}}$ from (\ref{tilde-G}).

\end{remark}

\medskip

\begin{lemma}\label{lem-covering-functionclass}
Let $\mathcal{H}=\mathcal{G}$ or $\widetilde{\mathcal{G}}$ denote one of the function classes defined in 
(\ref{G}) and (\ref{tilde-G}) (depending on $\nu>0$ and $\delta\in (0,1]$), then we have 
%
$$
 \log N(\epsilon,\mathcal{H},\|\cdot\|_\infty)=O\Big(\epsilon^{-\big(\frac{d}{\nu}+\frac{d}{d+\delta}\big)}\Big)
$$
for $\epsilon\searrow  0$, 
and thus the same bound holds for $\log N_{[\,]}\left(\epsilon,\mathcal{H},\|\cdot\|_2\right)$.
\end{lemma}

\begin{proof}
Let $\mathcal{H}=\mathcal{G}$ (the proof is similar for $\widetilde{\mathcal{G}}$) and let $\epsilon>0$. Choose $D=D(\epsilon)=\epsilon^{-1/\nu}$. 
Let $B$ denote the ball of radius $D$ around the origin. Let $a_{1},\dotsc,a_m:B\to\RR$ denote the centers of $\epsilon$-balls with respect to the supremum norm that cover $C_1^{d+\delta}(B)$, that is 
$ m= N(\epsilon,C_1^{d+\delta}(B),\|\cdot\|_\infty)$. 
 Then for each $a\in\mathcal{G}$ we have $a|_B\in C_1^{d+\delta}(B)$ and thus there exists $j_0\in\{1,\dotsc,m\}$ such that $\sup_{\xx\in B}|a(\xx)-a_{j_0}(\xx)|\leq \epsilon$. 
Now define $a_j(\xx)=0$ for $\xx\in\RR^d\setminus B$, $j=1,\dotsc,m$. Then 
$$
 \sup_{\xx\in\RR^d}|a(\xx)-a_{j_0}(\xx)|\leq \max\Big\{\epsilon, \sup_{\|\xx\|\geq D}|a(\xx)|\Big\}\leq \epsilon
$$
because $\|\xx\|^\nu|a(\xx)|\leq 1$ by definition of $\mathcal {G}$. We obtain $N(\epsilon,\mathcal{G},\|\cdot\|_\infty)\leq m$
and due to \cite{vaart_wellner}, Theorem 2.7.1,  we have for some universal $K$
$$
\log m\leq K\,\lambda^d\big(B^1\big)\,\epsilon^{-\frac{d}{d+\delta}}
=  O\big((D+2)^d\big)\,\epsilon^{-\frac{d}{d+\delta}}
= O\Big(\epsilon^{-\big(\frac{d}{\nu}+\frac{d}{d+\delta}\big)}\Big), 
$$
where $B^1 = \big\{\xx:\, \|\xx-B\| < 1 \big\}$

 Thus the first assertion follows. The second assertion follows by \cite{vaart_wellner}, proof of Cor.\ 2.7.2. 
\end{proof}

\bibliographystyle{apalike}\bibliography{short,ReferencesU}

\newcommand{\noop}[1]{}
\begin{thebibliography}{}

\bibitem[Berghaus et~al., 2017]{berghaus2017weak}
Berghaus, B., B{\"u}cher, A., and Volgushev, S. (2017).
\newblock Weak convergence of the empirical copula process with respect to
  weighted metrics.
\newblock {\em Bernoulli}, 23(1):743--772.

\bibitem[Bickel and Wichura, 1971]{bickel1971convergence}
Bickel, P.~J. and Wichura, M.~J. (1971).
\newblock Convergence criteria for multiparameter stochastic processes and some
  applications.
\newblock {\em Ann. Math. Statist.}, 42:1656--1670.

\bibitem[Brahimi and Necir, 2012]{brahimi2012semiparametric}
Brahimi, B. and Necir, A. (2012).
\newblock A semiparametric estimation of copula models based on the method of
  moments.
\newblock {\em Stat. Methodol.}, 9(4):467--477.

\bibitem[B{\"u}cher and Volgushev, 2013]{bucher2013empirical}
B{\"u}cher, A. and Volgushev, S. (2013).
\newblock Empirical and sequential empirical copula processes under serial
  dependence.
\newblock {\em J. Multivariate Anal.}, 119:61--70.

\bibitem[Chan et~al., 2009]{chan2009statistical}
Chan, N.-H., Chen, J., Chen, X., Fan, Y., and Peng, L. (2009).
\newblock Statistical inference for multivariate residual copula of {GARCH}
  models.
\newblock {\em Statist. Sinica}, 19:53--70.

\bibitem[Chen and Fan, 2006]{chen2006estimation}
Chen, X. and Fan, Y. (2006).
\newblock Estimation and model selection of semiparametric copula-based
  multivariate dynamic models under copula misspecification.
\newblock {\em J. Econometrics}, 135:125--154.

\bibitem[Dedecker and Louhichi, 2002]{dedecker_louhichi}
Dedecker, J. and Louhichi, S. (2002).
\newblock Maximal inequalities and empirical central limit theorems.
\newblock In Dehling, H., Mikosch, T., and Sorensen, M., editors, {\em
  Empirical process techniques for dependent data}, pages 137--160.
  Birkh\"auser Boston.

\bibitem[Dette et~al., 2009]{dette2009goodness}
Dette, H., Pardo-Fern{\'a}ndez, J.~C., and Keilegom, I.~V. (2009).
\newblock Goodness-of-fit tests for multiplicative models with dependent data.
\newblock {\em Scand. J.~Statist.}, 36(4):782--799.

\bibitem[Fan and Gijbels, 1996]{fan_gijbels_1996}
Fan, J. and Gijbels, I. (1996).
\newblock {\em Local Polynomial Modelling and Its Applications}.
\newblock Chapman \& Hall/CRC, London.

\bibitem[Fan and Yao, 2005]{fan_yao_book2005}
Fan, J. and Yao, Q. (2005).
\newblock {\em Nonlinear Time Series : Nonparametric and Parametric Methods}.
\newblock Springer series in statistics. Springer, New York.

\bibitem[Gao, 2007]{gao2007nonlinear}
Gao, J. (2007).
\newblock {\em Nonlinear Time Series: Semiparametric and Nonparametric
  Methods}.
\newblock Chapman \& Hall/CRC, Boca Raton.

\bibitem[Genest et~al., 1995]{genest_et_al_1995}
Genest, C., Ghoudi, K., and Rivest, L.-P. (1995).
\newblock A semiparametric estimation procedure of dependence parameters in
  multivariate families of distributions.
\newblock {\em Biometrika}, 82:543--552.

\bibitem[Genest et~al., 2009]{genest-et-al-gof-2008}
Genest, C., R{\'e}millard, B., and Beaudoin, D. (2009).
\newblock Goodness-of-fit tests for copulas: A review and a power study.
\newblock {\em Insurance Math. Econom.}, 44:199--213.

\bibitem[Gijbels et~al., 2017]{ogvp_testing_rao_2017}
Gijbels, I., Omelka, M., Pe\v{s}ta, M., and Veraverbeke, N. (2017).
\newblock Score tests for covariate effects in conditional copulas.
\newblock {\em J.~Multivariate Anal.}, 159:111--133.

\bibitem[Gijbels et~al., 2015]{ogv_sjs_2015}
Gijbels, I., Omelka, M., and Veraverbeke, N. (2015).
\newblock Estimation of a copula when a covariate affects only marginal
  distributions.
\newblock {\em Scand. J. Statist.}, 42:1109--1126.

\bibitem[Hansen, 2008]{hansen2008}
Hansen, B.~E. (2008).
\newblock Uniform convergence rates for kernel estimation with dependent data.
\newblock {\em Econometric Theory}, 24:726--748.

\bibitem[H{\"a}rdle et~al., 1998]{hardle1998nonparametric}
H{\"a}rdle, W., Tsybakov, A., and Yang, L. (1998).
\newblock Nonparametric vector autoregression.
\newblock {\em J.~Statist. Plann. Inference}, 68(2):221--245.

\bibitem[Kim et~al., 2007]{kim2007semiparametric}
Kim, G., Silvapulle, M.~J., and Silvapulle, P. (2007).
\newblock Semiparametric estimation of the error distribution in multivariate
  regression using copulas.
\newblock {\em Aust. N. Z. J. Stat.}, 49(3):321--336.

\bibitem[Kim et~al., 2008]{kim2008estimating}
Kim, G., Silvapulle, M.~J., and Silvapulle, P. (2008).
\newblock Estimating the error distribution in multivariate heteroscedastic
  time-series models.
\newblock {\em J.~Statist. Plann. Inference}, 138(5):1442--1458.

\bibitem[Kojadinovic and Holmes, 2009]{kojadinovic2009tests}
Kojadinovic, I. and Holmes, M. (2009).
\newblock Tests of independence among continuous random vectors based on
  {C}ram{\'e}r--von {M}ises functionals of the empirical copula process.
\newblock {\em J.~Multivariate Anal.}, 100(6):1137--1154.

\bibitem[Koul and Zhu, 2015]{koul2015goodness}
Koul, H.~L. and Zhu, X. (2015).
\newblock Goodness-of-fit testing of error distribution in nonparametric
  {ARCH(1)} models.
\newblock {\em J.~Multivariate Anal.}, 137:141--160.

\bibitem[Masry, 1996]{masry1996multivariate}
Masry, E. (1996).
\newblock Multivariate local polynomial regression for time series: uniform
  strong consistency and rates.
\newblock {\em J. Time Series Anal.}, 17(6):571--599.

\bibitem[McNeil et~al., 2005]{embrechts2005quantitative}
McNeil, A.~J., Frey, R., and Embrechts, P. (2005).
\newblock Quantitative risk management: Concepts, techniques, and tools.

\bibitem[M{\"u}ller et~al., 2009]{MSW2009}
M{\"u}ller, U.~U., Schick, A., and Wefelmeyer, W. (2009).
\newblock Estimating the error distribution function in nonparametric
  regression.
\newblock {\em Statist. Probab. Lett.}, 79:957--964.

\bibitem[Nelsen, 2006]{nelsen_2006}
Nelsen, R.~B. (2006).
\newblock {\em An Introduction to Copulas}.
\newblock Springer, New York.
\newblock Second Edition.

\bibitem[Patton, 2012]{patton2012review}
Patton, A.~J. (2012).
\newblock A review of copula models for economic time series.
\newblock {\em J.~Multivariate Anal.}, 110:4--18.

\bibitem[Portier and Segers, 2018]{portier2018weak}
Portier, F. and Segers, J. (2018).
\newblock On the weak convergence of the empirical conditional copula under a
  simplifying assumption.
\newblock {\em J.~Multivariate Anal.}, 166:160--181.

\bibitem[R{\'e}millard et~al., 2012]{remillard2012copula}
R{\'e}millard, B., Papageorgiou, N., and Soustra, F. (2012).
\newblock Copula-based semiparametric models for multivariate time series.
\newblock {\em J.~Multivariate Anal.}, 110:30--42.

\bibitem[Segers, 2012]{Segers:2012}
Segers, J. (2012).
\newblock Weak convergence of empirical copula processes under nonrestrictive
  smoothness assumptions.
\newblock {\em Bernoulli}, 18:764--782.

\bibitem[Tsukahara, 2005]{tsukahara_2005}
Tsukahara, H. (2005).
\newblock Semiparametric estimation in copula models.
\newblock {\em Canad. J. Statist.}, 33:357--375.

\bibitem[{van der Vaart} and Wellner, 1996]{vaart_wellner}
{van der Vaart}, A.~W. and Wellner, J.~A. (1996).
\newblock {\em Weak Convergence and Empirical Processes}.
\newblock Springer, New York.

\bibitem[{van der Vaart} and Wellner, 2007]{vaart_wellner_2007}
{van der Vaart}, A.~W. and Wellner, J.~A. (2007).
\newblock Empirical processes indexed by estimated functions.
\newblock In Cator, E.~A., Jongbloed, G., Kraaikamp, C., Lopuha{\"a}, H.~P.,
  and Wellner, J.~A., editors, {\em IMS Lecture Notes Monograph Series 2007:
  Asymptotics: Particles, Processes and Inverse Problems}, volume~55, pages
  234--252. Institute of Mathematical Statistics.

\bibitem[Veraverbeke et~al., 2011]{ogv_sjs_2011}
Veraverbeke, N., Omelka, M., and Gijbels, I. (2011).
\newblock Estimation of a conditional copula and association measures.
\newblock {\em Scand. J.~Statist.}, 38:766--780.

\bibitem[Yang et~al., 1999]{yang1999nonparametric}
Yang, L., H{\"a}rdle, W., and Nielsen, J. (1999).
\newblock Nonparametric autoregression with multiplicative volatility and
  additive mean.
\newblock {\em J. Time Series Anal.}, 20(5):579--604.

\end{thebibliography}

\end{document}